\documentclass[12pt]{article}
\usepackage[left=2cm,top=2cm,right=2cm,bottom=2cm]{geometry}

\usepackage[cp1251]{inputenc}
\usepackage[british]{babel}

\usepackage{amsfonts,amsmath,amssymb,graphicx,latexsym,xcolor,mathtools}
\usepackage{hyperref,enumitem,enumerate,float}
\usepackage{amsthm,verbatim}
 \usepackage{cite}
 \usepackage[colorinlistoftodos]{todonotes}

\numberwithin{equation}{section}
\pdfoutput=1

\theoremstyle{plain}
\newtheorem{theorem}{Theorem}[section]

\newtheorem{proposition}[theorem]{Proposition}

\theoremstyle{definition}

\newenvironment{defi}
{%
	\pushQED{\qed}\begin{defi/}}
	{\popQED\end{defi/}}
\newenvironment{exa}
{%
	\pushQED{\qed}\begin{exa/}}
	{\popQED\end{exa/}}
\newenvironment{rema}
{%
	\pushQED{\qed}\begin{rema/}}
	{\popQED\end{rema/}}

\newtheorem{defi/}[theorem]{Definition}
\newtheorem{rema/}[theorem]{Remark}
\newtheorem{exa/}[theorem]{Example}

\newcommand{\bfxi}{\tmmathbf{\xi}}
\newcommand{\bfeta}{\tmmathbf{\eta}}
\newcommand{\eps}{\varepsilon}
\newcommand{\ptl}{\partial}

\newcommand{\bde}{\boldsymbol{e}}
\newcommand{\bdeta}{\boldsymbol{\eta}}
\newcommand{\bdzeta}{\boldsymbol{\zeta}}

\newcommand{\bdG}{\boldsymbol{\Gamma}}

\newcommand{\bdx}{\boldsymbol{x}}

\newcommand{\bdxi}{\boldsymbol{\xi}}
\newcommand{\bdn}{\boldsymbol{n}}

\newcommand{\tva}{\text{\tmverbatim{a}}}
\newcommand{\tvb}{\text{\tmverbatim{b}}}

\usepackage{tikz}
\newcommand*\mycircle[1]{%
	\begin{tikzpicture}[baseline=(C.base)]
		\node[draw,circle,inner sep=0.2pt](C) {#1};
\end{tikzpicture}}

\makeatletter
\newcommand*{\transpose}{%
	{\mathpalette\@transpose{}}%
}
\newcommand*{\@transpose}[2]{%
	\raisebox{\depth}{$\m@th#1\intercal$}%
}
\makeatother

\newcommand{\mathd}{\mathrm{d}}

\newcommand{\tmem}[1]{{\em #1\/}}
\newcommand{\tmmathbf}[1]{\ensuremath{\boldsymbol{#1}}}
\newcommand{\tmop}[1]{\ensuremath{\operatorname{#1}}}

\newcommand{\tmverbatim}[1]{{\ttfamily{#1}}}

\pdfoutput=1

\ifpdf
\DeclareGraphicsExtensions{.eps,.pdf,.png,.jpg}
\else
\DeclareGraphicsExtensions{.eps}
\fi

\definecolor{myred}{RGB}{160,0,0}
\definecolor{mygreen}{RGB}{0,160,0}
\definecolor{myblue}{RGB}{0,0,160}
\hypersetup{
	colorlinks = true,
	linkcolor = {myred},
	citecolor = {mygreen},
	urlcolor  = {myblue},
}

\newcommand{\RED}{} 

\title{A contribution to the mathematical theory
	of diffraction.\\
	Part I: A note on double Fourier integrals}
\author{Rapha\"{e}l C. Assier$^{*}$, Andrey V. Shanin$^{\dagger}$ and Andrey I. Korolkov$^{\dagger}$\\
	\footnotesize{$^{*}$ Department of Mathematics, University of Manchester, Oxford Road, Manchester, {\rm M13 9PL}, UK}\\
	\footnotesize{$^{\dagger}$ Department of Physics (Acoustics Division), Moscow State University, Leninskie Gory, {\rm 119992}, Moscow, Russia}
}

\begin{document}	

%
%
%

\maketitle


\begin{abstract}
 We consider a large class of physical fields $u$ written as double inverse Fourier transforms of some \RED{functions} $F$ of two complex variables. Such integrals occur very often in practice, especially in diffraction theory. Our aim is to provide a closed-form far-field asymptotic expansion of $u$. In order to do so, we need to generalise the well-established complex analysis notion of contour indentation to integrals of functions of two complex variables. It is done by introducing the so-called bridge and arrow notation. Thanks to another integration surface deformation, we show that, to achieve our aim,  we only need to study a finite number of real points in the Fourier space: the contributing points. This result is called the locality principle. We provide an extensive set of results allowing one to decide whether a point is contributing or not. Moreover, to each contributing point, we associate an explicit closed-form far-field asymptotic component of $u$. We conclude the article by validating this theory against full numerical computations for two specific examples.  
\end{abstract}


\section{Introduction and motivation}

Many successful mathematical techniques used in diffraction theory (for
example the Wiener--Hopf \cite{Noble1958,LawrieAbrahams2007} or the
Sommerfeld--Malyuzhinets techniques \cite{SMtechnique2007}) rely on
one-dimensional (1D) complex analysis and are based on integral transformations
that transform the physical problem at hand into a functional equation in the
complex plane. As a result, if the method is successful, the physical solution
$u$ is given as an integral in the complex plane of a known \RED{function}
$F$ depending on a complex variable $\xi$, say.

Often, both the physical field $u$ and the \RED{function} $F$ do also
depend on a strictly positive parameter, $k$ say, that in diffraction theory
can be thought of as the wavenumber. In order to implement these complex
analysis techniques with ease, it is often assumed as a starting point that
the parameter $k$ has a small positive imaginary part. As a result, the \RED{function $F$} under consideration is free from singularity on the real
line in the $\xi$ complex plane and the resulting physical field is expressed
as an integral over this real line.

To recover the relevant physical field however, one must then take the limit
$\tmop{Im} [k] \rightarrow 0$. A common problem occurring at this stage is
that the singularities of the \RED{function $F$} hit the real line in that
limit. As a consequence, and thanks to Cauchy's theorem, the integration
contour needs to (and can) be deformed or {\tmem{indented}} to avoid these
singularities, as illustrated in Figure \ref{fig:1dindentation}. For practical
purposes, it is very important to know whether the contour passes above or
below the singularities of the \RED{function $F$}.

\begin{figure}[h]
  \centering {\includegraphics[width=0.8\textwidth]{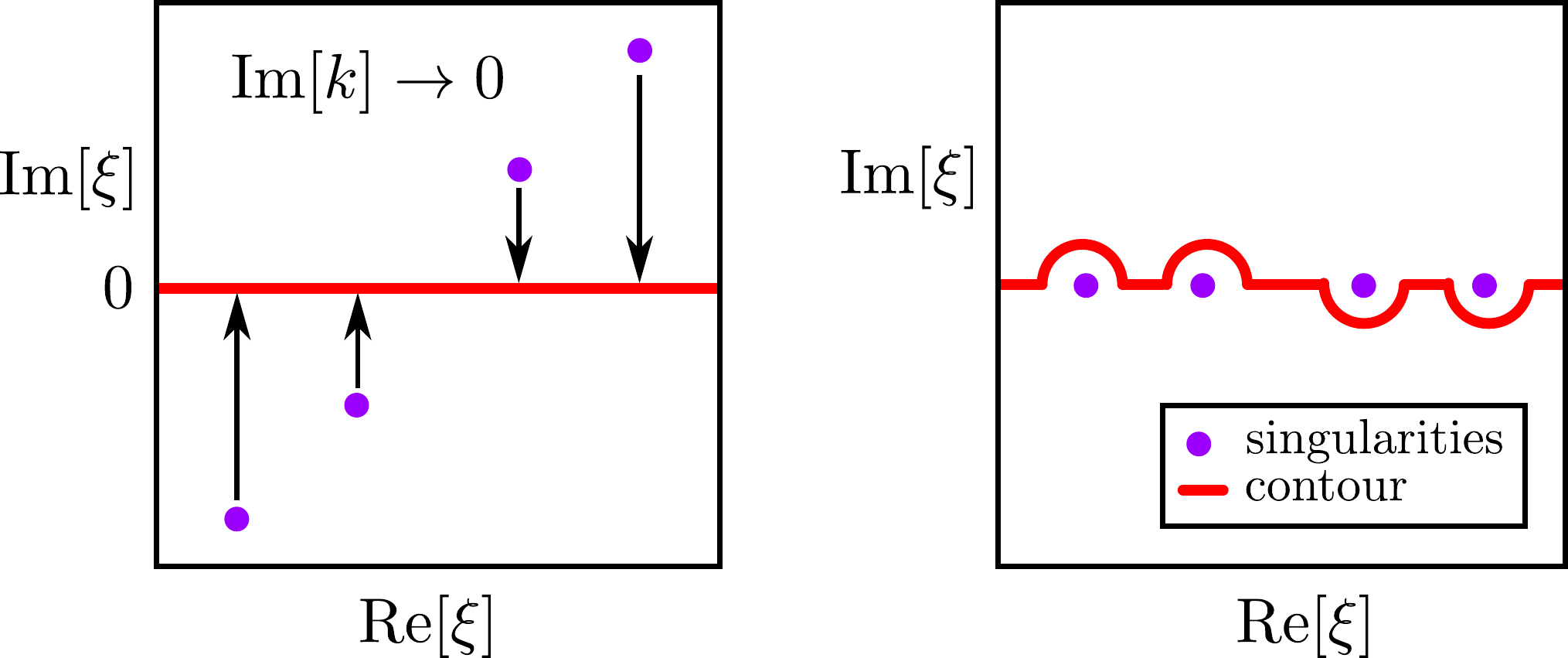}}
  \caption{Illustration of 1D complex analysis contour indentation.}
  \label{fig:1dindentation}
\end{figure}

The resulting integral cannot always be computed exactly. However, some very
powerful tools, such as the steepest descent method for instance, allow one
to obtain exact expressions for the asymptotic behaviour of the physical
solution in the far-field \RED{(see e.g.\ {\cite{Borovikov1994}})} .

The generic method described above has been very successful in diffraction
theory over the past century, resulting in elegant solutions to many canonical
problems such as diffraction by a half-plane or by a wedge, among other
building blocks of the so-called Geometrical Theory of Diffraction
{\cite{Keller1962}}, but it is not the purpose of the present work to review
this in an exhaustive manner.

The drawback of this method, however is that so far, it has mainly permitted
to solve two-dimensional time-harmonic problems, one-dimensional transient
problems or problems that can easily be reduced to those.

The main purpose of the present article is to develop a mathematical framework
to extend this method to higher-dimensional problems, such as
three-dimensional time-harmonic problems or two-dimensional transient
problems. As a consequence, we have to work in a two-dimensional (2D) complex space
and consider physical solutions $u(\bdx)$, with $\tmmathbf{x}=(x_1,x_2)\in\mathbb{R}^2$, written as a complex integral of a
\RED{function} $F(\bdxi)$, with $\tmmathbf{\xi}= (\xi_1,
\xi_2) \in \mathbb{C}^2$ over an
integration surface embedded in a 2D complex space.
More specifically, our aim is to estimate 2D Fourier integrals of the form
(\ref{eq:initial-integral}) given below. Throughout the article, the \RED{function} $F$ is assumed to be a known function. The asymptotic estimation of the double integral is built under the assumption that the physical variable $\tmmathbf{x}$ is such that 
$\tmmathbf{x} = r  \tmmathbf{\tilde x}$,  
where $ \tmmathbf{\tilde x}$  
is a fixed unit real vector, and $r$ is a large real parameter ($r \to \infty$), 
i.e.\ the \textit{far~field} is considered.


This work is part of the authors ongoing efforts
to apply the theory of functions of several complex variables {\cite{Shabat2}} to
diffraction theory, see for instance
{\cite{2dcontinuation2021,Assier2019a,Assier2018a,Assier2019b,Assier2019c,Ice2021,Shanin2019SommerfeldTI,kunz2021diffraction}}. A fundamental motivation behind the present work is that the methods developed here can be applied to 
the important canonical problem of diffraction by a quarter-plane studied in 
{\cite{Shanin2005,shanin1,Assier2012,Assier2012b,Assier2016,Assier2018a,Assier2019a,Assier2019c,Shanin2012}} for instance. In this case, even though the \RED{function} $F$ remains unknown, it has been shown (in \cite{Assier2018a}, say) that it has some properties known a priori. Namely, the function $F(\bdxi)$ can be analytically continued  far enough from the integration surface, and it has singular sets defined by the equations
$\xi_1^2 + \xi_2^2 = k^2$, $\xi_1 = a_j$ and $\xi_2 = b_j$ for some known constants $k$, $a_j$ and $b_j$. The behaviour of the function $F$ near these sets is known and can be either of polar or branching type. Note that $\mathbb{C}^2$ has real dimension~4, while the singular sets have real dimension~2, i.e.\ they are surfaces. Moreover, in this case, the singularities possess the so-called {\em real property\/}. A formal definition is given below in Definitions \ref{def:realproperty} and \ref{def:realpropertyfunction},
but on the practical level this means that the intersections of the singular sets 
with the real plane $(\tmop{Re} [\xi_1], \tmop{Re} [\xi_2])$, called their \textit{real traces}, are some one-dimensional curves if $k$ is real. This is not generally the case for the intersection of two two-dimensional surfaces in four dimensions, but we will focus on this special case since it is important in practice for diffraction theory.

\RED{A lot of research has been conducted on the asymptotic approximation of multiple integrals. It was reported in influential textbooks such as \cite{Borovikov1994, Wong2001-cd, Jones1982-xs, Bleistein1987-xa, Felsen1994-hb,Lighthill78}. Upon writing $\bdxi=(\xi_1,\dots,\xi_n)$, the concern is on estimating multiple integrals of the form $\int_D F(\bdxi) e^{\lambda G(\bdxi)} \mathd \bdxi$ as a large real parameter $\lambda$ tends to infinity for some finite or infinite domain $D$ in $\mathbb{R}^n$. The standard hypothesis used in these texts is that both functions $F$ and $G$ are smooth within $D$ and on its boundary. The estimation of the integral can then be done by only considering certain \textit{critical points}: saddle points in the interior of $D$ (i.e.\ points where $\nabla G=\boldsymbol{0}$) and problematic points linked to the geometry of the boundary of $D$. The former can be dealt with using the multidimensional saddle point method (also known as stationary phase method), and many different configurations need to be considered for the latter. Most books cited above base their presentation on the latter on the influential article \cite{Jones1958-nf}. Additional analysis has been performed to derive approximations for the case of saddle points being close to each other \cite{ursell_1980}, in which the possible construction of steepest descent surfaces is being discussed. Previous work by Fedoryuk \cite{fedoryuk1977saddle}, also considered such surfaces.

	Even though we will not follow this route in the present work, it has to be mentioned that, more recently, efforts have been made to obtain hyperasymptotic expansions (including exponentially decaying terms in the expansion, see e.g. \cite{Berry1991-na}) of such multiple integrals with simple saddle points in \cite{kaminski1994exponentially} and \cite{Howls1997-me}. These sophisticated approaches, also concerned with the construction of steepest descent surfaces, are particularly interesting to us as they make use of advanced ideas in multidimensional complex analysis and emphasise the link between multidimensional integral evaluations and Pham's work \cite{Pham2011} on the Picard-Lefschetz theory.  An alternative method, based on a Mellin-Barnes integral approach can be found in \cite{paris_kaminski_2001}.
	  
	However, it is rare to find similar studies where the function $F$ is allowed to have singularities within the domain $D$. Notable exceptions can be found in Jones' book \cite{Jones1982-xs}, in which he deals with \textit{isolated} singularities of $F$ and cases when $F$ is the reciprocal of a polynomial (also considered by Lighthill in \cite{Lighthill78}), as well as in \cite{Jones1971-yw}, where he deals with a specific 2D integral. However, the general case is not treated. Indeed, according to Jones in \cite{Jones1982-xs}: ``\textit{The asymptotic behaviour of Fourier transforms in n dimensions is considerably more complicated than that in one dimension. Primarily, this is because singularities occur not only at isolated points but also on curves and, in general, on hypersurfaces which may be of any dimension up to $n-1$.}''
	 In the present work, we aim to give a reasonably general treatments of 2D Fourier integrals for which the singularities of the integrand lie on a set of (potentially intersecting) curves.
}

The present article (Part I) is meant to lay
out a general mathematical framework. 
Part II of this work {\cite{Part6B}} will be dedicated to the specific
example of the quarter-plane and will highlight the strength and relative simplicity of the present approach.


\RED{There are three main notions of importance used in this work: the notion of \textit{active
	and inactive points}, the \textit{bridge and arrow} representation and the \textit{additive
	crossing} property. The notion of active and inactive points was used in \cite{Ice2021}, but only for a very
	specific case. In the present work, we give a much more general account of this
	notion, and we prove general results regarding the activity or not of singular
	points.
	
	The bridge and arrow notation was introduced for the first time in the
	Appendix B of \cite{Assier2019c} and has not been used since. It was necessary to introduce
	it at the time to allow us to deal with some technical difficulties linked to
	the stencil equations we were considering then. Though a fairly general account
	of this tool was given in \cite{Assier2019c}, the current presentation has been
	refined and the notation has been simplified. With more emphasis given on the
	complementary role played by the bridge and by the arrow respectively.
	
	The additive crossing property was first discovered/introduced in \cite{Assier2018a}, and
	subsequently arised in the follow-up work \cite{Assier2019c}, and perhaps more surprisingly
	in a somewhat unrelated work \cite{2dcontinuation2021}. However its importance (or lack of importance as we will see) to the far-field asymptotics of the waves under consideration was not considered in those articles. 
	
	Most importantly, before the present work, no effective link seemed to exist
	between these three notions. Though, as we will make clear in this article, they are
	deeply connected. Moreover, this connection can be exploited to obtain fairly
	general results about the asymptotic behaviour of 2D Fourier integrals.
}

The key result of the present work is as follows. The asymptotic estimation of a 2D Fourier 
integral can be performed by applying the {\em locality principle}: the terms that are not 
exponentially vanishing in the far-field can be obtained by considering small neighbourhoods of several 
``special points'' only. These special points are found to be the intersections of the real traces of the singular sets 
(not necessarily all of them) and the so-called \textit{saddles on singularities}.  The 
leading terms of the asymptotic estimations 
can then be found by computing some simple standard integrals. 

The rest of the article is organised as follows. In Section \ref{sec:sec2}, we specify the type of
integrals to be considered and make some crucial assumptions on the \RED{functions $F$} under consideration, namely that they have the so-called real property. In Section \ref{sec:bridgeandarrow},
we develop a mathematical technique to describe how a two-dimensional
surface of integration bypasses singularities in $\mathbb{C}^2$ and show that
this process can be effectively described by a concise notation: {\tmem{the
bridge and arrow}}. In Section \ref{sec:sec4}, we introduce the notion of active and inactive points and show that the only points that have the potential of contributing towards the far-field asymptotics of $u$ are the active non-additive transverse crossings and the active saddles \RED{on singularities} of the real trace of $F$. In Section \ref{sec:nearactive} we perform a local consideration of these points and provide an explicit closed-form estimation of local Fourier-type integrals in their vicinity.
 In Section \ref{sec:integrationawayfromactive} we construct a global deformed integration surface and use it to show that the asymptotic expansion of the physical field $u$ can simply be written as the sum of the local contributions computed in Section \ref{sec:nearactive}.
In Section \ref{sec:simplenontrivialexamples} we illustrate the validity of our theory with two simple but non-trivial examples.
\RED{Finally, in Appendix~\ref{sec:complicated}, we comment briefly on the application of our technique to more complicated integrals.}



\section{Integrals under consideration} \label{sec:sec2}


\subsection{Double Fourier integrals}

Throughout this article, we will consider a function $F (\tmmathbf{\xi}; \varkappa)$
depending on two complex variables denoted $\tmmathbf{\xi}= (\xi_1, \xi_2) \in
\mathbb{C}^2$ and on a small parameter $\varkappa> 0$, 
$\varkappa \to 0$ (this limit will from now on be denoted $\varkappa \searrow 0$). For wave motivated problems, this parameter mimics a 
small energy dissipation in the medium making it possible to use the limiting absorption 
principle. More precisely, the wavenumber parameter $k$ may be written as 
$k= k_0 + i \varkappa$,  
where $k_0$ is real and strictly positive, and $i \varkappa$ is a vanishing imaginary part. 
We also introduce the notation
\[
F(\bfxi) \coloneqq F(\bfxi ; 0).
\]
For $\varkappa>0$, the function $F(\bfxi ; \varkappa)$ is assumed to be holomorphic
in a neighbourhood of $\mathbb{R}^2$, and assumed to grow at most algebraically
at infinity with respect to $\bfxi$. 
This allows one to define a function $u (\tmmathbf{x};\varkappa)$, depending on two real variables denoted $\tmmathbf{x}= (x_1, x_2) \in
\mathbb{R}^2$ and on $\varkappa>0$, by 
\begin{align}
  u (\tmmathbf{x};\varkappa) = 
   \iint_{\mathbb{R}^2} F (\tmmathbf{\xi}; \varkappa) e^{-
  i\tmmathbf{x} \cdot \tmmathbf{\xi}} \mathd \tmmathbf{\xi},  \label{eq:initial-integral}
\end{align}
where, by $\mathd \tmmathbf{\xi}$, we mean\footnote{It is not strictly necessary to use this wedge product of differential forms to write (\ref{eq:initial-integral}), we could just use $\mathd \xi_1 \mathd \xi_2$, however, since we will later deform this integration surface in $\mathbb{C}^2$, it is easier to start with this. See \cite{Jo-2018} for a gentle introduction on differential forms, and \cite{Shabat2} for their use in high-dimensional complex integration.} $\mathd \xi_1 \wedge \mathd \xi_2$.

The two functions $u$ and $F$ are, by construction, related to each other via
a double Fourier transform. Indeed, up to a multiplicative constant depending
on the chosen convention, (\ref{eq:initial-integral}) is a double inverse
Fourier transform. Hence, we will refer to $F$ as the {\tmem{\RED{Fourier transform}}, while $u$ will be called the {\tmem{physical field}}. 

Assume that, as $\varkappa \searrow 0$, the singularities of $F$
hit the real plane so that the integral (\ref{eq:initial-integral}) becomes
ill-defined. To make sense of the function 
\begin{align}
u (\tmmathbf{x}) \coloneqq \lim_{\varkappa \searrow 0 } u (\tmmathbf{x};\varkappa),
\label{eq:limitdefofu}
\end{align} 
we need to {\tmem{indent}} the surface of integration of
(\ref{eq:initial-integral}) {\tmem{around}} the singularities of $F$. 
This can be done thanks to the 2D analogue of Cauchy's theorem \cite{Shabat2}:
one can deform the integration surface continuously, provided that the integrand is an 
analytical function of two variables, and the surface never hits the singular sets of the 
integrand. The value of the integral remains the same during this deformation.   
This theorem is non-trivial\footnote{A similar fact is not valid in general for integration 
over contours with real dimension~1 in $\mathbb{C}^2$, see e.g.\ \cite{2dcontinuation2021} for a special case when it is valid.} and is based on the closedness of a corresponding differential 
2-form.  

As a result of the deformation of the surface of integration we obtain
\begin{align}
  u (\tmmathbf{x})  =  \iint_{\bdG} F (\tmmathbf{\xi})
  e^{- i\tmmathbf{x} \cdot \tmmathbf{\xi}} \, \mathd \tmmathbf{\xi}  
  \label{eq:FourierIntegralIndented}
\end{align}
for some two-dimensional surface of integration $\bdG \subset
\mathbb{C}^2$ that does not intersect any of the singularities of $F$. We 
assume that $\bdG$ coincides with the real plane of $\tmmathbf{\xi}$
everywhere except in some neighbourhood of the singularities of $F$, and
that $\bdG$ bypasses the singularities of $F$ in some sense. The surface of integration $\bdG$ is assumed to be oriented. The orientation is taken in 
	such a way that, for any portion $\boldsymbol{\mathcal{D}}$ of $\bdG$ coinciding with a portion of $\mathbb{R}^2$ the integral of 
	$\iint_{\boldsymbol{\mathcal{D}}}\dots \mathd \tmmathbf{\xi}$ is the same as the integral of $\iint_{\boldsymbol{\mathcal{D}}} \dots \mathd\xi_1 \mathd\xi_2$. 
Here we have used
the words {\tmem{indent}} and {\tmem{around}} in order to make an analogy with
contour deformation in one complex variable, though, as we will see, things
are not that simple in $\mathbb{C}^2$.

The procedure for the estimation of the integral (\ref{eq:FourierIntegralIndented}) can be summarised as follows. The surface $\bdG$ will be deformed 
into a surface $\bdG'$, on which the integrand is exponentially 
vanishing everywhere as $r\to\infty$, except in the neighbourhoods of several {\em special points}.
The estimation of the integral near these points will provide non-vanishing terms
of the physical field $u$ as $r\to\infty$.

Thus, we will have to describe two successive deformations of the integration surfaces: 
\[
\mathbb{R}^2 \longrightarrow \bdG \longrightarrow \bdG'.
\]
The first deformation accompanies the limiting process $\varkappa \searrow 0$, while, during the second, the value of $\varkappa = 0$ remains unchanged.


\subsection{Functions with the real property}

Before being more specific about $\bdG$ and $\bdG'$, we will make a few more
assumptions on the function $F(\bdxi)$. We start by assuming
that $F$ is holomorphic on $\mathbb{C}^2$ everywhere apart from on its
singularity set. \RED{This singularity set, that we call $\sigma$, is a (complex) analytic set and can therefore be written as a finite union of irreducible analytic sets that we call
	{\tmem{irreducible singularity sets\/}} $\sigma_j$:
	\[
	\sigma = \cup_j \sigma_j,
	\]
	each $\sigma_j$ corresponding to a polar or a branch
	set. Since they are themselves analytic sets, each irreducible singularity set $\sigma_j$ can be described by 
	$\sigma_j =\{ \tmmathbf{\xi} \in \mathbb{C}^2, g_j (\tmmathbf{\xi}) = 0 \}$ for some
	holomorphic function $g_j$ that we call the defining function of $\sigma_j$. For simplicity, throughout this article, we also assume that $\sigma_j$ does not have any singular points (i.e.\ there are no points where $\nabla g_j=0$ ). Because of this assumption, such irreducible singularity sets can be shown to be complex manifolds embedded in
	$\mathbb{C}^2$ and are therefore purely two-dimensional\footnote{From now on, whenever we talk about
		dimension, we will talk about real dimension.}.  In other words, the irreducible singularity sets can be seen as smooth two-dimensional surfaces embedded in four dimensions. For more on analytic sets and their irreducible components, we refer the reader to \cite{Shabat2} (Chapter II \S 8) or to the more involved monograph \cite{Chirka1989-bo}. }

We will further assume that each irreducible singularity has
the so-called {\tmem{real property}}.

\begin{defi}
	\label{def:realproperty}
  We say that an irreducible singularity set $\sigma_j$ has the {\em real property} if
  its defining holomorphic function $g_j$  is real whenever its arguments are real. As a consequence, its first order partial derivatives have the same property and we can write: 
  \begin{align*}
    \tmmathbf{\xi} \in \mathbb{R}^2  \Rightarrow  g_j (\tmmathbf{\xi}) \in
    \mathbb{R}, \hspace{1em} \tva(\bdxi)=\frac{\partial g_j}{\partial \xi_1}
    (\tmmathbf{\xi}) \in \mathbb{R} \hspace{1em} \text{ and } \hspace{1em} 
    \tvb(\bdxi)=\frac{\partial g_j}{\partial
    \xi_2} (\tmmathbf{\xi}) \in \mathbb{R}.
  \end{align*}
Moreover, we insist that $g_j(\bdxi)$ is regular for $\bdxi\in\mathbb{R}^2$, that is $\tva^2+\tvb^2\neq0$.
\end{defi}
As a consequence, the intersection between $\mathbb{R}^2$ and an irreducible singularity set
$\sigma_j$ having the real property is a one-dimensional curve. Note that this is not normally the case. Indeed, in general, the intersection of a
manifold $g_j(\bdxi)= 0$ with the $\tmmathbf{\xi}$ real plane is a set of discrete points.
In many practical situations, however, the singularities do possess the real property, 
as is the case for the quarter-plane diffraction problem for instance. 
\RED{We note that if the singularities did not possess the real property, then their intersection with $\mathbb{R}^2$ would be of dimension zero, and therefore they would be isolated within $\mathbb{R}^2$. As discussed in introduction, this configuration has already been studied by Jones in \cite{Jones1982-xs}.}     

\begin{defi}
	The one-dimensional curve resulting from the intersection of an irreducible
	singularity set $\sigma_j$ having the real property and $\mathbb{R}^2$ is
	called the {\em real trace} of $\sigma_j$ and denoted by $\sigma'_j$:
	\[
	\sigma'_j = \sigma_j \cap \mathbb{R}^2.
	\]
	The real trace of a singularity set $\sigma=\cup_j\sigma_j$, denoted $\sigma'$, is defined as the union of the real traces of its irreducible components: $\sigma'=\cup_j \sigma'_j$.
\end{defi}

To describe the limiting process $\varkappa \searrow 0$ we should consider the defining 
functions of the singularities of $F(\bdxi;\varkappa)$, also depending on $\varkappa$, denoted $g_j(\tmmathbf{\xi} ; \varkappa)$, and such that $g_j(\bfxi) = g_j(\bfxi ; 0)$. Moreover, we require  that the partial derivative 
$\ptl_\varkappa g_j(\bfxi ; 0)$ should be purely imaginary when $\bdxi$ is real. We can therefore extend the concept of real property to the functions $F
(\tmmathbf{\xi}; \varkappa)$ of interest as follows.

\begin{defi}
	\label{def:realpropertyfunction}
  A function $F (\tmmathbf{\xi}; \varkappa)$ whose irreducible singularities when
  $\varkappa = 0$ all have the real property is said to have the real
  property.
\end{defi}

We give below two simple examples of functions having the real property.

\begin{exa}
  \label{ex:ex4}The function 
  \[  
  F (\bfxi; \varkappa) = 
  \left( (1 + i \varkappa)^2 - \xi_1^2 - \xi_2^2 \right)^{-1}
  \] 
  has the real property.
  Its only irreducible singularity when $\varkappa=0$, the
  complexified circle, is given by $\sigma = \{ \tmmathbf{\xi} \in
  \mathbb{C}^2, \xi_1^2 + \xi_2^2 - 1 = 0 \}$, while its real trace
  $\sigma'$ is simply the unit circle centred at the origin (see
  Figure \ref{fig:realtracesexamples}, left). This singularity is a polar set.
\end{exa}

\begin{exa}
  \label{ex:ex5}The function 
  \[
  F (\tmmathbf{\xi}; \varkappa) = \frac{( \xi_1 - 1 - i \varkappa)^{1 / 2} }{
  \xi_2 + 1 + i \varkappa}
  \] 
  has the real property. 
  When $\varkappa = 0$, it has two irreducible singularities
  $\sigma_1 = \{ \tmmathbf{\xi} \in \mathbb{C}^2 , \xi_1 = 1 \}$ and
  $\sigma_2 = \{ \tmmathbf{\xi} \in \mathbb{C}^2, \xi_2 = - 1 \}$. Note that
  $\sigma_1$ is a branch set, while $\sigma_2$ is a polar set. Both real
  traces $\sigma_1'$ and $\sigma_2'$ are straight lines (see 
  Figure~\ref{fig:realtracesexamples}, right).
\end{exa}

\begin{figure}[h]
  \centering{\includegraphics[width=0.6\textwidth]{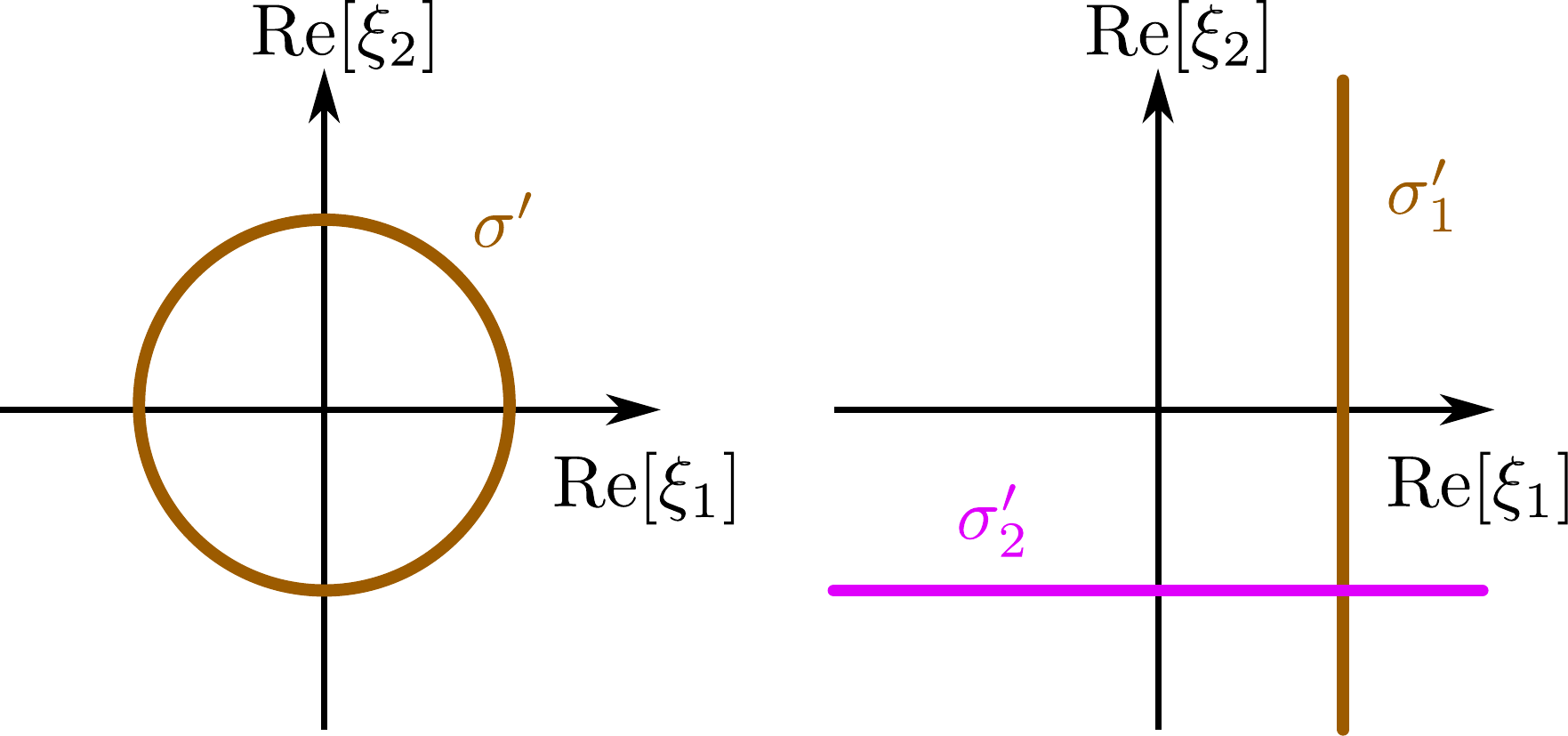}}
  \caption{Illustration of the real traces of the functions given in Examples
  \ref{ex:ex4} (left) and \ref{ex:ex5} (right)}
\label{fig:realtracesexamples}
\end{figure}

Throughout the rest of the article, we consider a function $F (\tmmathbf{\xi};\varkappa)$ with the real property and a singularity set of $F(\bdxi)$ denoted by $\sigma  = \cup_j \sigma_j$, where the irreducible components $\sigma_j$ are described by 
$\{ \tmmathbf{\xi} \in \mathbb{C}^2, g_j (\tmmathbf{\xi}) = 0 \}$ for some
entire functions $g_j$.



\section{The bridge and arrow notation}\label{sec:bridgeandarrow}


In this section, our aim is to describe the relative position of the surface of integration $\bdG$
or $\bdG'$ with respect to the singularities of $F$.
In order to do so, we need a convenient tool to describe such surfaces and the continuous deformation between them.  
Here we focus on the initial surface of integration $\bdG$, but the surface $\bdG'$ can be described in a similar way.

\subsection{A parametrisation of the surface of integration}\label{sec:surfaceparam}

We wish to consider an indented surface of integration $\bdG$
that is close to $\mathbb{R}^2$, i.e. a small perturbation from the
initial integration surface. The surface $\bdG$ is also expected not to intersect with the singularity set $\sigma$ of the \RED{Fourier transform} $F$.
Moreover, we insist that the surface be piecewise smooth and {\em flattable}, i.e.\
that it can can almost be flattened to the real plane. A more rigorous definition of flattability will be given shortly in Definition~\ref{def:flattability}. 

Let 
$\bdG$ be
parametrised by a complex position vector 
$\tmmathbf{\xi}_{\bdG}: \mathbb{R}^2 \rightarrow \mathbb{C}^2$ 
that can be written as
\begin{align}
  \tmmathbf{\xi}_{\bdG} (\xi_1^r, \xi_2^r) 
  =  \left( 
    \xi_1^r + i \eta_1 (\xi_1^r, \xi_2^r),
    \  
    \xi_2^r + i \eta_2 (\xi_1^r, \xi_2^r)
  \right),  \label{eq:parametrisation}
\end{align}
where $(\xi_1^r , \xi_2^r) \equiv \bfxi^r \in \mathbb{R}^2$, $\bfxi^r = {\rm Re}  [\bfxi ]$ and the scalar functions $\eta_1$ and $\eta_2$ are assumed to be small, bounded and piecewise smooth real functions. The non-intersecting property can be reformulated by saying that $\bdxi_{\bdG}(\mathbb{R}^2)\cap \sigma = \emptyset$. 

\begin{defi}\label{def:flattability}
	We say that such a surface $\bdG$ is {\em flattable} (or has the flattability property) with respect to a singularity set $\sigma$ if for any $0< \eps \le 1$, the scaled surface $\bdG_\eps$ defined by the scaled functions $\eps\eta_{1,2}$ and parametrised by \[
	\bfxi_{\bdG_{\eps}}(\bfxi^r)=
	\left( 
	\xi_1^r + i \eps \eta_1 (\bfxi^r), 
	\ 
	\xi_2^r + i \eps \eta_2 (\bfxi^r)
	\right),
	\]
	does not intersect the singularity set $\sigma$.
\end{defi}


Therefore, an admissible\footnote{
			The existence of such admissible contour is directly linked to the notion of
			``complex displacement'' developed in Pham's work {\cite{Pham2011}} chapter
			IX section 1.2 p177.} two-dimensional surface of
integration $\bdG$ embedded in $\mathbb{C}^2$ is completely
determined by a two-dimensional real vector field 
$\tmmathbf{\eta}: \mathbb{R}^2 \rightarrow
\mathbb{R}^2$ defined by $\tmmathbf{\eta}= (\eta_1, \eta_2)$. The class of surfaces we have defined is very restrictive indeed. 
However, we will see that this class is enough to establish the locality principle and build
the estimation of the leading terms of the physical field~$u$.


To summarise, the vector field $\tmmathbf{\eta}$
should be chosen in such a way that $\bdG$ has 
two important properties: 

\begin{itemize}

\item[a)] it has an empty intersection with the singular set of $F$ and is flattable;

\item[b)] it can be obtained as a continuous deformation of $\mathbb{R}^2$ 
in the limiting process $\varkappa \to 0$, i.e.\ 
it is homotopic to the physically motivated integration surface. 

\end{itemize}

A sufficient condition for a) is 
given by Theorem~\ref{th:flattened}, while a sufficient condition for b) 
is formulated below as a rule for choosing the bridge symbol. 
The bridge symbol is a very useful and compact notation that
 was first introduced briefly in the Appendix B of
{\cite{Assier2019c}}, but given its importance in what follows, it deserves a
more prominent and general presentation.

\begin{theorem}
\label{th:flattened}
Let $\bfeta^\sharp (\bfxi^r)$ be a piecewise smooth bounded real vector field, chosen in such a way that, at any point of the real trace $\sigma'$, it is not zero nor is it tangent to any irreducible component of $\sigma'$. Then, there exists a smooth positive real factor function $f(\bfxi^r)$ such that the vector field
\[
\bfeta(\bfxi^r) = f(\bfxi^r) \, \bfeta^\sharp (\bfxi^r)
\]
defines a flattable surface $\bdG$ not crossing $\sigma$. 
\end{theorem}

\begin{proof}
Let $\bdeta^\sharp$ be a vector field with all the required properties.
Let us start by considering a real point $\bdxi^\star\in\mathbb{R}^2$ that does not belong to the real trace $\sigma'$. By definition $\bdxi^\star$ does not belong to $\sigma$ either and in this case, we just need to choose $f(\bdxi^\star)$ small enough (smaller than the distance between $\bdxi^\star$ and $\sigma$) to ensure that $f(\bdxi^\star) \bdeta^\sharp(\bdxi^\star)\notin\sigma$.

Consider now a real point $\bdxi^\star$ belonging to some component $\sigma_j'$ of the real trace $\sigma'$ (i.e.\ $g_j(\bdxi^\star)=0$) and define the two real quantities
\[
\tva^\star \equiv \frac{\ptl g_j(\bfxi^\star)}{\ptl \xi_1} , 
\quad \text{ and } \quad 
\tvb^\star \equiv \frac{\ptl g_j(\bfxi^\star)}{\ptl \xi_2}. 
\]  
Note that the fact that $\bdeta^\sharp$ is not tangent to $\sigma_j'$ at $\bdxi^\star$ can be succinctly rewritten as $\tva^\star \eta_1^\sharp+\tvb^\star \eta_2^\sharp\neq0$. Let us now consider a vector field $\bdeta=f(\bdxi^\star)\bdeta^\sharp$ for some small $f(\bdxi^\star)>0$. We can hence use a Taylor expansion to find that
\[g_j(\xi_1^\star+i \eta_1,\xi_2^\star+i \eta_2)=if(\bdxi^\star)(\tva^\star \eta_1^\sharp+\tvb^\star \eta_2^\sharp)+\mathcal{O}((f(\bdxi^\star))^2).\]
Since the first term in the expansion is non-zero, it is possible to choose $f(\bdxi^\star)$ small enough such that the higher-order terms cannot compensate it, and hence such that $g_j(\bdxi^\star+i\bdeta)\neq0$ and $\bdxi_\Gamma(\bdxi^\star)\notin\sigma_j$.

For each point of $\bdxi^\star\in\mathbb{R}^2$ we have hence constructed a vector field $\eta(\bdxi^\star)=f(\bdxi^\star)\bdeta^\sharp(\bdeta^\star)$ that defines a surface $\bdG$ not intersecting $\sigma$. Note that each time we needed to choose $f$ small enough, therefore the surface $\bdG_\eps$ for $0<\eps\le1$ does also possess the non-intersecting property and $\bdG$ is flattable.

In practice, to make sure that $f$ is continuous as one approaches a singular trace, one should proceed as follows. As described above, find a suitable function $f$ that is continuous along the singular trace $\sigma_j'$. Then consider a small curved strip around this singular set, on which $f$ can be chosen as follows. We insist that $f$ remains constant on each segment perpendicular to $\sigma_j'$. With this we can then patch together, and in a continuous manner the function $f$ at the singular points and at the non-singular points.
%
%
\end{proof}

Let us now illustrate that the condition of $\bfeta^\sharp$ not being tangent to any of the $\sigma_j'$ is important. Consider for instance the case of an irreducible singularity component $\sigma_j$ defined by the equation
\[
g_j (\xi_1 , \xi_2) =  \xi_1 - \xi_2^2, 
\]  
and consider the point $\bfxi^\star = (0,0)$ belonging to its real trace $\sigma_j'$. Note that the vector $(0,1)$ is tangent to 
the parabola $\sigma_j'$ at $\bdxi^\star$. 
Let us choose the vector field $\bfeta^\sharp$ to be constant and equal to $(0, 1)$ in a neighbourhood of $\bdxi^\star$. It is hence tangent to $\sigma_j'$ at $\bfxi^\star$. 
Then, for any positive small smooth bounded function $f(\bdxi^r)$,
there is a real point near $\bfxi^\star$ over which the surface $\bdG$ defined by the vector field $\bdeta(\bdxi^r)=f(\bdxi^r)\bdeta^\sharp(\bdxi^r)$ intersects $\sigma_j$.
Indeed, considering the set of points $\bdxi^r$ that can be written $\bdxi^r=(\xi_1^r,0)$, we find that $g_j(\bdxi_{\bdG}(\bdxi^r))=\xi_1^r+(f(\bdxi^r))^2$. Hence upon choosing $\xi_1^r=-(f(\xi_1^r,0))^2$, we have that $\bdG$ intersects $\sigma_j$ at $\bdxi^r$.

From now on, we assume that each integration surface is described by a vector field 
$\bfeta$ obeying the condition of Theorem \ref{th:flattened}, i.e.\ it is non-zero on $\sigma'$
and is non-tangent to any component $\sigma'_j$. Theorem~\ref{th:flattened} is important for building the surfaces 
$\bdG$ and $\bdG'$. It shows that each surface of integration belonging to the important class considered here can be described by a vector field $\bfeta$ obeying conditions that are easy to check. We stress that a full explicit knowledge of the field $\bfeta$ can sometimes be needed, for instance when one has to compute the integral 
(\ref{eq:FourierIntegralIndented}) numerically, as will be done to illustrate the validity of our theory in Section \ref{sec:simplenontrivialexamples}. 

From a theoretical point of view, the vector field $\bdeta$ is not easy to visualise graphically. Hence, below, we develop two useful notations enabling one to succinctly and graphically represent 
some important properties of $\bfeta$: they are the {\em bridge\/}
notation and the {\em arrow\/} notation. These notations only provide a small portion of information about $\bfeta$. The information provided by the bridges is a purely topological property of $\bdG$: it shows how $\bdG$ bypasses 
the singularities of $F$. 
The role of the bridges is very similar to the semi-loops in 
Figure~\ref{fig:1dindentation} (and they look similar). 
The role of the arrows is more subtle. They provide information about exponential growth and decay of the integrand and have no analogue in 1D complex analysis.


\subsection{Introduction of the bridge and arrow symbol} \label{sec:introbridgesymbol}

Consider a surface $\bdG$ belonging to the flattable class introduced in Definition \ref{def:flattability} and assume that the vector field $\bfeta$ describing this surface obeys the conditions of Theorem~\ref{th:flattened}. 
According to these conditions,
$\bfeta$ is continuous, non-zero on the real trace $\sigma'$ and non-tangent to any of the real trace components $\sigma'_j$. Let us focus on a specific real trace component $\sigma'_j$ and consider a real point $\bfxi^\star\in\sigma_j'$.

Because it is not tangent to $\sigma_j'$, the vector field $\bfeta$ is restricted to be directed to the left or to the right of $\sigma'_j$ everywhere 
(the ``right'' and ``left'' sides of $\sigma_j$ are defined in an arbitrary way here). 

Introduce the bridge and arrow symbols associated to some $\sigma_j$ as it is shown in Figure~\ref{fig:sideofeta}. The bridge symbol 
consists of two straight stems perpendicular to $\sigma_j'$ and an arc that connects the stems. It gives some information on how the surface $\bdG$ is bypassing $\sigma_j$ at $\bdxi^\star$. The arrow, also chosen to be perpendicular to $\sigma_j'$, represents the side (right or left) of $\sigma_j'$ towards which the vector field $\bdeta$ points to.


\begin{figure}[h]
  \centering{\includegraphics[width=0.6\textwidth]{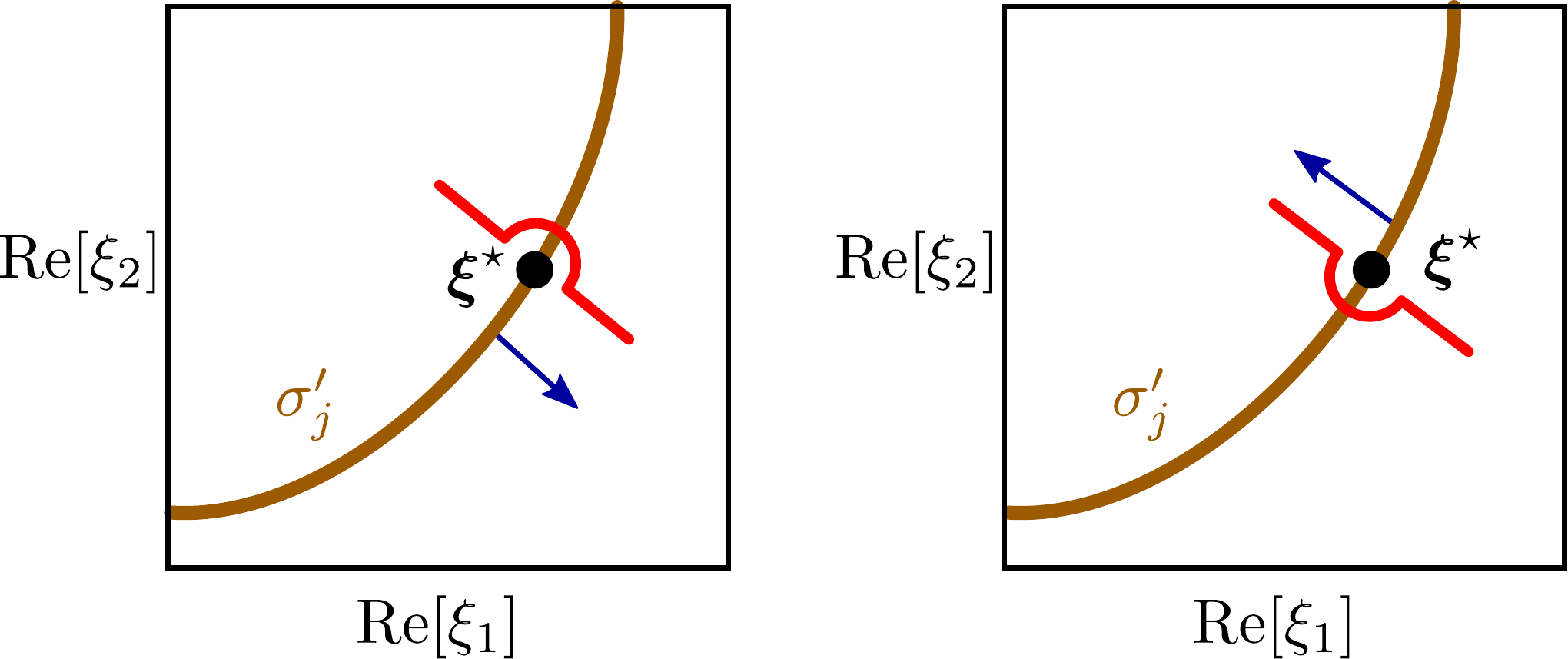}}
  \caption{Introduction of the bridge and arrow notation.}
\label{fig:sideofeta}
\end{figure}

Let us list some important properties of the bridge and arrow symbol. All these points can be proved using the mathematical framework developed in the next subsection.

\begin{itemize}
\item[(i)] The two configurations of Figure \ref{fig:sideofeta} are the only two possible bridge and arrow configurations.

\item[(ii)] The stems of the bridge symbol do not necessarily have to be normal 
to $\sigma_j'$. Indeed, the symbol can be rotated freely as long as the stems do not become tangent to $\sigma_j'$.
Still, just by looking at the bridge symbol it is possible to say whether $\bfeta$ points 
to the right or to the left of $\sigma_j'$. Hence these too symbols, the bridge and the arrow are somewhat redundant as they each convey the same information.

\item[(iii)] The bridge symbol is defined for a single point of $\sigma_j'$, but according 
to the conditions of Theorem~\ref{th:flattened} it can be carried 
along $\sigma_j'$ in a continuous manner. 
\end{itemize}

Hence, in order to understand the position of $\bdG$ relative to $\sigma$, a bridge configuration should be introduced for each singular trace component $\sigma_j'$ of the singularity set $\sigma$. Using the points (i)--(iii), we can derive two more important properties:

\begin{itemize}


\item[(iv)]\label{item:(iv)} In the case of a transverse crossing between two real trace components $\sigma_1'$ and $\sigma_2'$, the bridge symbols on each component can be chosen independently. 

\item[(v)] For a quadratic touch point $\bfxi^\star$ between two real traces $\sigma_1'$ and $\sigma_2'$,
the bridge symbols cannot be selected independently. Indeed, the vector $\bfeta$ at $\bfxi^\star$ is the same for both traces, and the bridge configurations should be compatible at the touching point.
The only two possible bridge configurations for a quadratic touch  are shown in 
Figure~\ref{fig:bridge-tangential}.  

\end{itemize}

\begin{figure}[h]
  \centering{\includegraphics[width=0.6\textwidth]{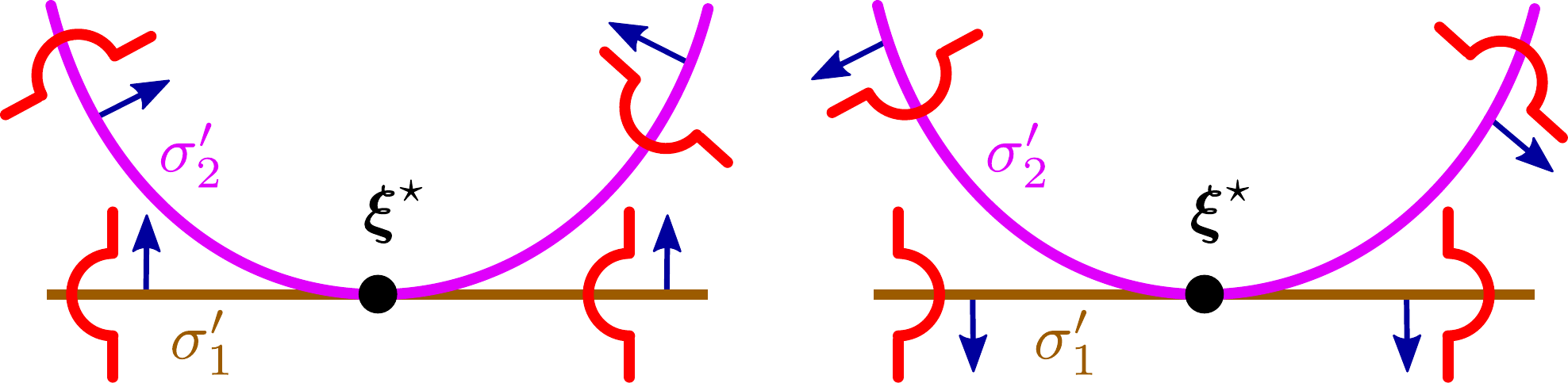}}
  \caption{The two possible bridge configurations at a tangential point.}
  \label{fig:bridge-tangential}
\end{figure}

Let us illustrate the points (i)--(iii) by a non-trivial and, as we will see with the quarter-plane problem in \cite{Part6B}, practically important example.

\begin{exa}
\label{ex:circle}
Let us consider a singularity set $\sigma$ of the form
$
\sigma =\{\bdxi\in\mathbb{C}^2,\, \xi_1^2 + \xi_2^2 - 1 = 0\}.
$ 
As mentioned in Example~\ref{ex:ex4}, its real trace $\sigma'$ is a circle. According to the point (iii), it is enough to choose the bridge and arrow symbols
at a single point of $\sigma'$. By continuity, 
this defines these symbols everywhere on $\sigma'$. 
The only two possible configurations are shown in    
Figure~\ref{fig:ex8}.
\end{exa}

\begin{figure}[h]
	\centering{\includegraphics[width=0.6\textwidth]{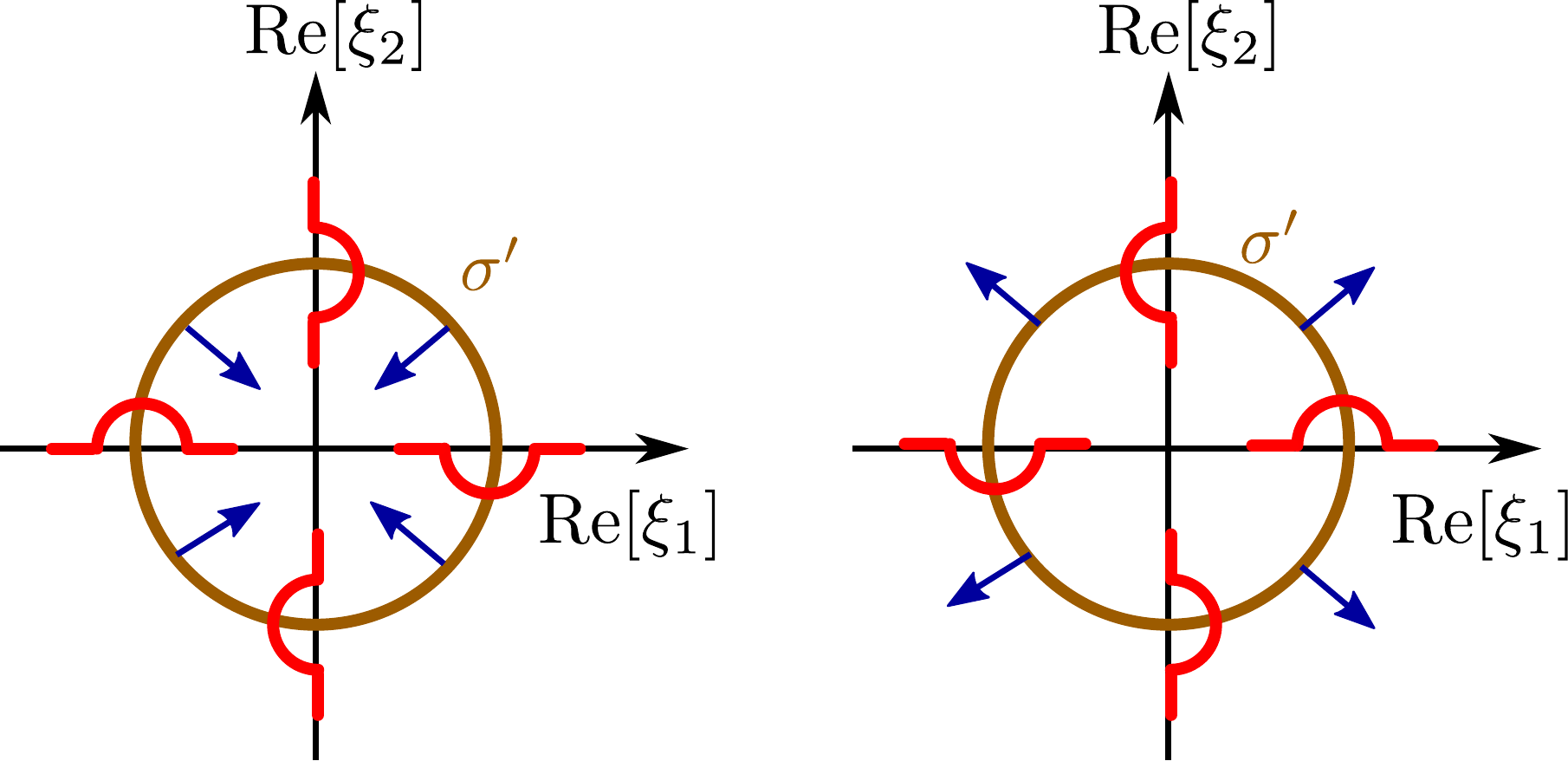}}
	\caption{The two possible bridge and arrow configurations for a circle.}
	\label{fig:ex8}
\end{figure}

To illustrate the point (iv) that the 
bridge symbols are independent at a tangential crossing, 
consider the following example. 

\begin{exa}
Consider the singularitiy set $\sigma=\sigma_1\cup\sigma_2$ where $\sigma_{1}=\{\xi_{1} = 0\}$ and $\sigma_{2}=\{\xi_{2} = 0\}$. Let $\gamma_+$ and $\gamma_-$ be 1D complex contours coinciding with the real axis everywhere 
except in a neighbourhood of the origin that they bypass from above and below respectively. 
Then the surfaces $\bdG_{\pm\pm}=\gamma_{\pm} \times \gamma_{\pm}$, where the 
signs are chosen independently, are all admissible (they do not cross the singularities and are flattable), and correspond to four different bridge configurations. 
\end{exa}

\subsection{A mathematical framework for the bridge symbol}

The convenience of the bridge symbol and its properties listed above can be proved using a relatively simple mathematical framework. Namely, 
one can introduce a local complex transverse coordinate $z$ near 
some real point $\bfxi^\star \in \sigma_j'$,
and the bridge notation shows how a cross-section of $\bdG$ with a transversal coordinate 
line bypasses the singularity. 

Consider a point
$\tmmathbf{\xi}^{\star} = (\xi_1^{\star}, \xi_2^{\star})$ on 
the real trace $\sigma_j'$ of an irreducible component~$\sigma_j$ with defining function $g_j$. 
By definition, this point is real, i.e. $\tmmathbf{\xi}^{\star} \in
\mathbb{R}^2$, and we have
\begin{align}
  g_j (\tmmathbf{\xi}^{\star}) = 0, \hspace{1em}  \frac{\partial g_j}{\partial \xi_1}
  (\tmmathbf{\xi}^{\star}) = \tva^{\star} \in \mathbb{R}, \hspace{1em} 
  \frac{\partial g_j}{\partial \xi_2} (\tmmathbf{\xi}^{\star}) =
  \tvb^{\star} \in \mathbb{R}, \hspace{1em} \left(
  \tva^{\star} \right)^2 + \left( \tvb \right)^2
  \neq 0.  \label{eq:realproperty}
\end{align}

Consider the local change of variable $(\xi_1, \xi_2) \leftrightarrow
(z, t)$ centred at $\tmmathbf{\xi}^{\star}$, where $z$ is a
{\tmem{transverse}} variable to $\sigma_j$ at $\tmmathbf{\xi}^{\star}$ and $t$
is {\tmem{tangent}} to $\sigma_j$, described by
\begin{align}
  z (\tmmathbf{\xi}) = \beta' g_j (\xi_1, \xi_2) \hspace{1em}  \text{ and } \hspace{1em}  t
  (\tmmathbf{\xi}) = \beta'' \left( - \tvb^{\star} (\xi_1 - \xi_1^{\star})
  + \text{\tmverbatim{a}}^{\star} (\xi_2 - \xi_2^{\star}) \right) + \beta''' z,
  \label{eq:changeofvar}
\end{align}
for some constants $\beta', \beta'' \in \mathbb{R} \setminus \{ 0 \}$ and
$\beta''' \in \mathbb{R}$. Note that $z (\tmmathbf{\xi}^{\star}) = t
(\tmmathbf{\xi}^{\star}) = 0$. One interesting point about this change of
variables is that it is a change of {\tmem{complex}} variables, but, when
restricted to the real plane, it is also a proper {\tmem{real}} change of
variables $(\xi_1^r, \xi_2^r) \leftrightarrow (z^r,t^r)$.

Let us now restrict ourselves to the
real ($\xi_1^r, \xi_2^r$) plane with Cartesian unit vectors $\tmmathbf{e}_{\xi_1}$
and $\tmmathbf{e}_{\xi_2}$. Since the change of variables is real, it is
possible to talk about the real unit basis vectors $\tmmathbf{e}_z$ and
$\tmmathbf{e}_t$ associated to it, defined by
\begin{align*}
  \tmmathbf{e}_z = \hat{\tmmathbf{e}}_z / | \hat{\tmmathbf{e}}_z |, \quad
  \tmmathbf{e}_t = \hat{\tmmathbf{e}}_t / | \hat{\tmmathbf{e}}_t |, \quad
  \hat{\tmmathbf{e}}_z = \partial \tmmathbf{r}/\partial z, \quad
  \hat{\tmmathbf{e}}_t = \partial \tmmathbf{r}/\partial
  t,
\end{align*}
where here $\tmmathbf{r}$ is defined to be the real position vector
$\tmmathbf{r}= \xi_1^r \tmmathbf{e}_{\xi_1} + \xi_2^r \tmmathbf{e}_{\xi_2} $. To make the transverse/tangent property a bit clearer, let us introduce the
unit vectors $\tmmathbf{n}^{\star}$ and $\tmmathbf{t}^{\star}$ as follows:
\begin{align}
  \tmmathbf{n}^{\star} = \frac{1}{\sqrt{\left( \text{\tmverbatim{a}}^{\star}
  \right)^2 + \left( \tvb^{\star} \right)^2}} \left( \begin{array}{c}
    \text{\tmverbatim{a}}^{\star}\\
    \tvb^{\star}
  \end{array} \right) \hspace{1em} \text{ and } \hspace{1em} \tmmathbf{t}^{\star} =
  \frac{1}{\sqrt{\left( \text{\tmverbatim{a}}^{\star} \right)^2 + \left(
  \tvb^{\star} \right)^2}} \left( \begin{array}{c}
    - \tvb^{\star}\\
    \text{\tmverbatim{a}}^{\star}
  \end{array} \right),  \label{eq:ntvect}
\end{align}
so that $\tmmathbf{n}^{\star}$ is a unit normal vector to $\sigma_j'$ at
$\tmmathbf{\xi}^{\star}$ and $\tmmathbf{t}^{\star}$ is a unit tangent vector
to $\sigma_j'$. With a little bit of algebra, one can show that
\begin{align}
  \hat{\tmmathbf{e}}_z = \frac{1}{\sqrt{\left( \text{\tmverbatim{a}}^{\star}
  \right)^2 + \left( \tvb^{\star} \right)^2}} \left( \frac{1}{\beta'}
  \tmmathbf{n}^{\star} - \frac{\beta'''}{\beta''} \tmmathbf{t}^{\star} \right)
  \hspace{1em} \text{ and } \hspace{1em} \hat{\tmmathbf{e}}_t = \frac{1}{\beta'' \sqrt{\left(
  \text{\tmverbatim{a}}^{\star} \right)^2 + \left( \tvb^{\star}
  \right)^2}} \tmmathbf{t}^{\star},  \label{eq:ztvect}
\end{align}
illustrating the fact that $\hat{\tmmathbf{e}}_z$ is transverse to
$\sigma'_j$ (but not necessarily perpendicular to it) and $\hat{\tmmathbf{e}}_t$
is tangent to $\sigma_j'$ as can be seen in Figure \ref{fig:unitvectorsezet}. Changing the values of the three constants $\beta'$,
$\beta''$ and $\beta'''$ leads to rotations of $\hat{\tmmathbf{e}}_z$ and
potentially to a change of direction of $\hat{\tmmathbf{e}}_t$.

\begin{figure}[h]
  \centering{\includegraphics[width=0.3\textwidth]{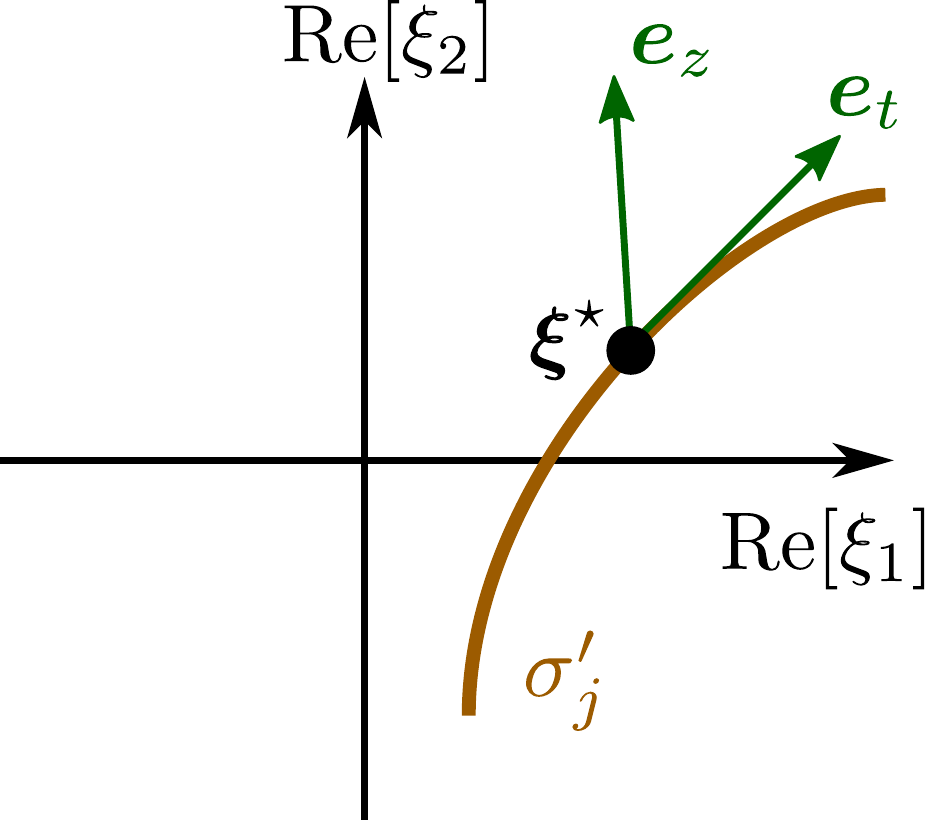}}
  \caption{Illustration of the local change of variable $(\xi_1, \xi_2)
  \leftrightarrow (z, t)$ and the related unit vectors.}
\label{fig:unitvectorsezet}
\end{figure}



	As a consequence of this change of variable, $\tmmathbf{\Gamma}$ can also be seen (at least locally) as a
	surface in the $(z, t)$ two-dimensional complex space that ``lies above'' the
	$(z^r, t^r)$ two-dimensional real plane in the sense that it can be
	parametrised by the two real parameters $(z^r, t^r)$ in the vicinity of the origin.	
	Consider now the part of $\tmmathbf{\Gamma}$ corresponding to the image of the points $(z^r, 0)$ by this parametrisation, for $z^r$ in a neighbourhood of $0$. This object,
	denoted $\tmmathbf{\Gamma} |_{t^r = 0}$, is one-dimensional. Its natural projection on the $z$-complex plane is a contour that coincides with the real axis everywhere except in a neighbourhood of $z=0$, since $\bdG$ coincides with $\mathbb{R}^2$ everywhere except the 
neighbourhood of $\sigma'_j$. 
The contour cannot pass through $z = 0$ since this point belongs to the 
singularity set~$\sigma_j$. Thus, the contour should bypass  $z = 0$ either
from above or from below. 

We will see that to each such change of variables, we can associate a bridge symbol. To obtain this symbol, draw the graph of $\tmmathbf{\Gamma} |_{t^r = 0}$ just atop the graph of $\sigma_j'$ as a small ``icon''.
Then, place the origin of the $z$-plane of the ``icon''
at the point $\tmmathbf{\xi}^{\star}$. Align the positive ${\rm Re}[z]$
direction with the vector $\tmmathbf{e}_z$. Make the positive ${\rm Im}[z]$ direction having angle 
$+ \pi /2$ with respect to ${\rm Re}[z]$ (i.e.\ the mutual orientation of the axes is fixed).
Finally, remove the axes of the $z$ complex plane to obtain a bridge symbol drawn near the point $\tmmathbf{\xi}^{\star}$.
The whole procedure of creating this bridge is illustrated in 
Figures \ref{fig:iniabove} and \ref{fig:inibelow}, considering the case of the contour passing above  and below the point 
$z = 0$ respectively. 
For a given orientation of  
$\tmmathbf{e}_z$, they are therefore only two possible bridge configuration near~$\tmmathbf{\xi}^{\star}$. In the case of Figure \ref{fig:iniabove} (resp. Figure  \ref{fig:inibelow}), we say that $\bdG$ bypasses $\sigma_j$ from above (resp. below) at $\bdxi^\star$ in the direction $\bde_z$.


The bridge notation introduced in this way has all the properties listed in the previous subsection.
The proof can be obtained by studying the value  
$z^{\star} = z (\tmmathbf{\xi}_{\bdG}
(\bfxi^{\star}))$, where $\bdxi_{\bdG}$ is the parametrising complex position vector of $\bdG$ as defined in Section \ref{sec:surfaceparam}. 
Since the surface $\bdG$ with defining vector field $\bdeta$ is flattable, we can assume without loss of generalities that the components of  $\tmmathbf{\eta}^{\star}=(\eta_1^\star,\eta_2^\star)=\bdeta(\bdxi^\star)$ are small. Hence, given that $g_j$ is holomorphic, we can use a first order Taylor expansion and
(\ref{eq:realproperty}) to get
\begin{align}
  z^{\star}  \approx  i  \beta'  \left( \eta_1^{\star}
  \tva^{\star} + \eta_2^{\star} \tvb^{\star} \right) = i
   \beta' \sqrt{\left( \tva^{\star} \right)^2 +
  \left( \tvb^{\star} \right)^2} \tmmathbf{\eta}^{\star} \cdot
  \tmmathbf{n}^{\star},  \label{eq:zstarexpr}
\end{align}
where $\tmmathbf{n}^{\star}$ is the unit normal vector introduced in
(\ref{eq:ntvect}). It should now be clear that if $\bdG$
bypasses $\sigma_j'$ in the direction $\tmmathbf{e}_z$ from above, then
$\tmop{Im} [z^{\star}] > 0$, while if it is from below, then 
$\tmop{Im}[z^{\star}] < 0$.
Note also that, from (\ref{eq:zstarexpr}), we get 
\begin{align}
  \tmop{sign} (\beta' \tmop{Im} [z^{\star}])  =  \tmop{sign}
  (\tmmathbf{\eta}^{\star} \cdot \tmmathbf{n}^{\star}) .  \label{eq:invariant}
\end{align}
This equation links the definition of the bridge symbol by the transverse variable 
and the vector $\bfeta$. 
Indeed, the expression in (\ref{eq:invariant})
is an invariant, since its right-hand side (RHS) does not depend on the parameters $\beta'$, $\beta''$
$\beta'''$ or on the point on $\sigma_j'$, implying the properties (ii) and (iii). 
To prove the property (i), note that the left equation of (\ref{eq:ztvect}) implies that $\text{sign}(\beta')=\text{sign}(\bde_z \cdot \bdn^\star)$. Hence, if $\bde_z$ and $\bdeta^\star$ point toward the same side of $\sigma_j'$ (i.e~ $\text{sign}(\bde_z \cdot \bdn^\star)=\text{sign}(\bdeta^\star \cdot \bdn^\star)$), (\ref{eq:invariant}) implies that  $\text{Im}[z^\star]>0$, proving that, in this case, the bypass in the $\bde_z$ direction is from above. Therefore $\bdG$ always bypasses $\sigma_j'$ from above in the $\bdeta^\star$ direction. This is why the bridge and arrow configuration can only be one of the two types presented in Figure \ref{fig:sideofeta}.

\begin{figure}[h]
  \centering{\includegraphics[width=0.8\textwidth]{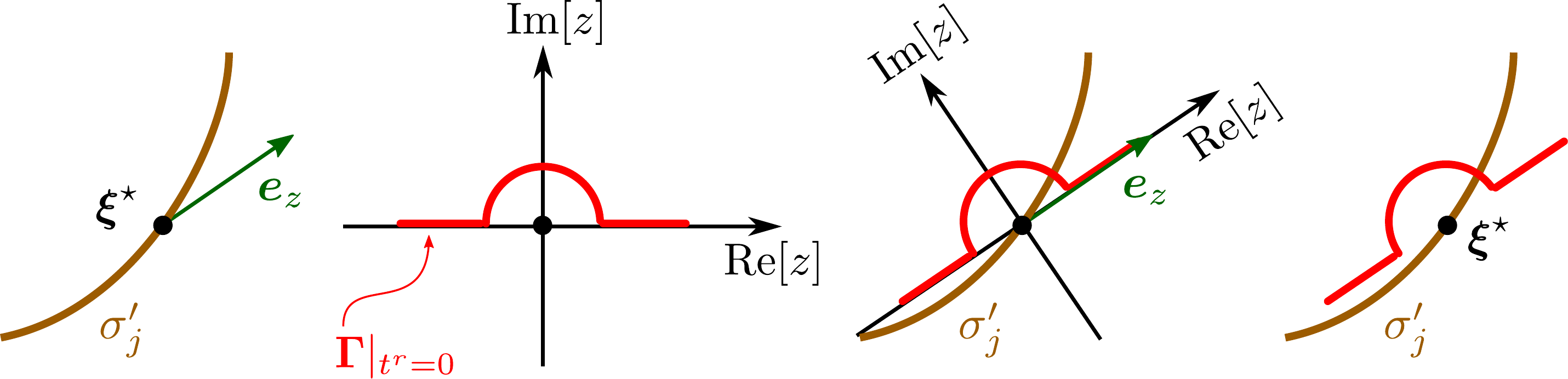}}
  \caption{The surface $\bdG$ bypasses $\sigma_j'$ from above in the direction
  $\bde_z$.}
\label{fig:iniabove}
\end{figure}

\begin{figure}[h]
  \centering{\includegraphics[width=0.8\textwidth]{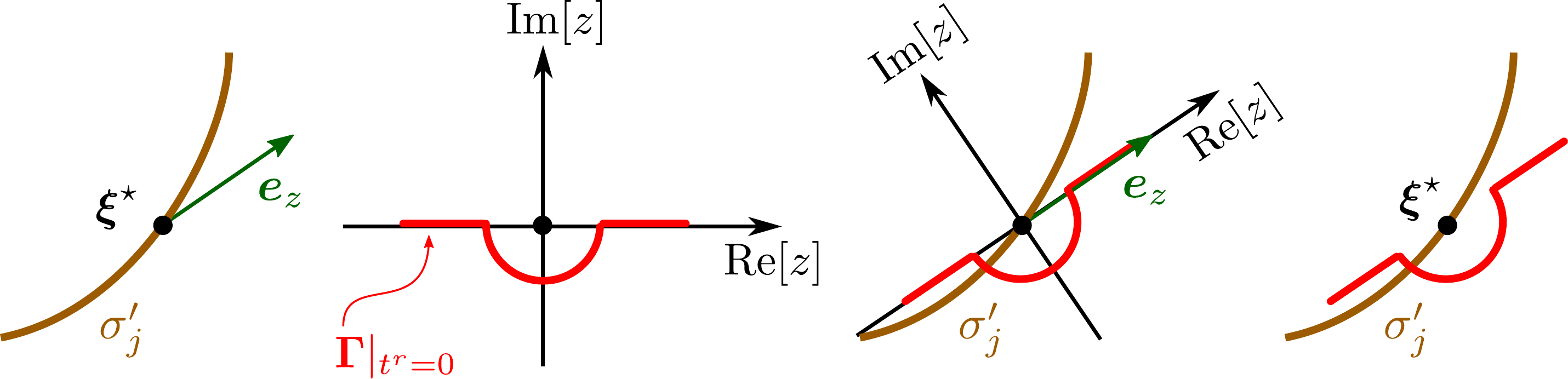}}
  \caption{The surface $\bdG$ bypasses $\sigma_j'$ from below in the direction
  $\tmmathbf{e}_z$.}
\label{fig:inibelow}
\end{figure}

A consequence of this exposition, which should hopefully be clear by now, is that the bridge (and the arrow) notation is a {\em topological invariant\/}
of a flattable surface $\bdG$. This important point can be summarised as follows.

\begin{proposition}
\label{pr:bridge_invariant}
It is impossible 
to perform a homotopy  
of a flattable surface $\bdG$ in $\mathbb{C}^2 \setminus \sigma$
(i.e.\ a continuous deformation without hitting the singularities) that leads to a change of bridge configuration for some singularity components $\sigma_j'$.
\end{proposition}


\subsection{How to choose the bridge configuration}
\label{sec:findinggoodbridges}

In the consideration above, we described an arbitrary 
flattable integration surface $\bdG$ close to the real plane: a surface satisfying the point a) of Secion \ref{sec:surfaceparam}. We now want to characterise the point b):   the surface $\bdG$ should be obtained 
by continuously deforming $\mathbb{R}^2$ in (\ref{eq:initial-integral}), 
through a physically motivated procedure to obtain (\ref{eq:FourierIntegralIndented}).   
Thus, 
the bridge symbol for each irreducible singularity component should be chosen 
in a unique way
by studying the limiting process $\varkappa \searrow 0$.

Let $g_j(\tmmathbf{\xi}; \varkappa)$ be the defining function of $\sigma_j$ and consider again a point $\tmmathbf{\xi}^\star$ on $\sigma'_j$. Assume that $\tva^\star
= \ptl g_j / \ptl \xi_1(\bdxi^\star;0) \neq 0$, 
i.e.\ assume that $\sigma'_j$ is not parallel to the $\xi_1^r$ axis 
(this does not restrict the generality, since a similar procedure can be done
with respect to the coordinate $\xi_2^r$).

The integration surface for $\varkappa >0 $ is the real plane. 
Consider the set $\{\bdxi\in\mathbb{C}^2,\,\xi_2=\xi_2^\star\}$. The intersection of this set with the integral surface, projected in the $\xi_1$ complex plane, is simply the $\xi_1$ real axis, while its intersection with $\sigma_j$, projected in the $\xi_1$ complex plane, is a point $\xi_1^\dagger(\varkappa)$ satisfying the equation 
\[
g_j(\xi_1^\dag(\varkappa) , \xi_2^\star ; \varkappa) = 0,
\] 
together with the natural condition
\[
\xi_1^\dag(\varkappa) \to \xi_1^\star
\quad
\mbox{as}
\quad
\varkappa \searrow 0. 
\] 

Omitting the terms of order $O(\varkappa^2)$, and using the chain rule, we obtain 
\[
\xi^\dag_1 (\varkappa) \approx \xi_1^\star -
\frac{\varkappa}{\tva^\star}
\frac{\ptl g_j(\tmmathbf{\xi^*} , 0)}{\ptl \varkappa}
.
\] 
According to the definition of a singularity with the real property, the partial derivative is purely imaginary. 
Thus, the point 
$\xi_1^\dag(\varkappa)$ has a non-zero imaginary part, and it approaches the real axis either from 
above or from below as $\varkappa \searrow 0$.

Note that this procedure coincides to a choice of change of variables $(z,t)$ such that $\bde_z=\bde_{\xi_1}$, which can be obtained, for instance, by choosing $\beta' = 1/\tva^\star$ and $\beta'''/\beta'' = \text{\tmverbatim{b}}^{\star}$. 

Therefore, if $\xi_1^\dag(\varkappa)$ approaches the real axis from above as $\varkappa\searrow0$, the corresponding portion of $\bdG$ in the $\xi_1$ complex plane should be deformed below the origin, as illustrated in Figure~\ref{fig:initialbridge} (left). In this case, the bridge symbol is chosen such that $\bdG$ bypasses $\sigma_j'$ from below in the $\bde_{\xi_1}$ direction, fixing the bridge configuration in a unique way on the whole of $\sigma_j'$. Respectively, if $\xi_1^\dag(\varkappa)$ approaches the real axis from below, as illustrated in Figure~\ref{fig:initialbridge} (right), the bridge is chosen such that  $\bdG$ bypasses $\sigma_j'$ from above in the $\bde_{\xi_1}$ direction. 

Therefore, in all cases, just by looking at the sign of $\text{Im}[\xi_1^\dagger(\varkappa)]$ we can determine uniquely the appropriate bridge configuration ensuring that the point b) of Section \ref{sec:surfaceparam} is satisfied.

 
\begin{figure}[h]
  \centering{\includegraphics[width=0.6\textwidth]{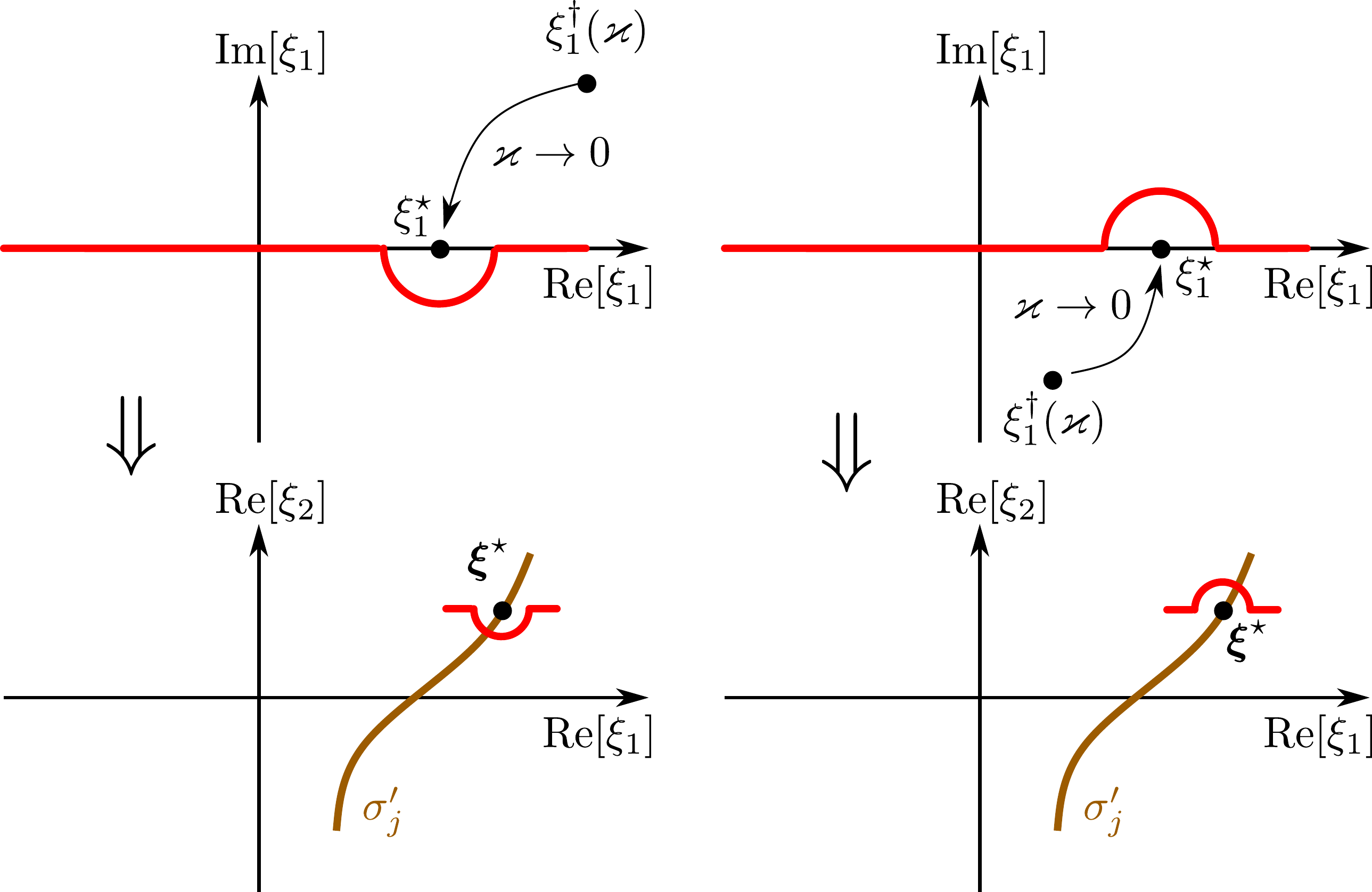}}
  \caption{How to choose a bridge configuration.}
  \label{fig:initialbridge}
\end{figure} 
 

Let us illustrate this procedure by the following example.

\begin{exa}
\label{ex:35}
Take a singularity close to the one studied in Example~\ref{ex:circle}:
\[
g(\bfxi ; \varkappa) = \xi_1^2 + \xi_2^2 - (1 + i \varkappa)^2 .
\] 
Consider the limit $\varkappa \searrow 0$ and choose a correct 
bridge symbol among the two ones shown in Figure~\ref{fig:ex8}.
For this,  
take the point $\bdxi^\star = (1 , 0)$ and  
fix $\xi_2 = 0$. 
In this case, we obtain $\xi_1^\dag (\varkappa) = 1 + i \varkappa$, 
and this point approaches the real 
axis of $\xi_1$ from above as $\varkappa\searrow0$. 
This corresponds to the left part of 
Figure~\ref{fig:initialbridge}, and $\bdG$ should bypass $\sigma_j'$ from below at $\bdxi^\star$ in the direction $\bde_{\xi_1}$. Therefore, the correct bridge configuration is shown in the Figure~\ref{fig:ex8} (left).
\end{exa}


\section{Active and inactive points}\label{sec:sec4}

We have now made sense of the limit (\ref{eq:limitdefofu}) by showing that $u$ can be written as the integral (\ref{eq:FourierIntegralIndented}) for some flattable surface $\bdG$ with a uniquely defined bridge configuration. The overarching aim of this article is to shed light on the asymptotic
behaviour of the physical field 
$u (\tmmathbf{x})$ as $r = | \tmmathbf{x} |
\rightarrow \infty$. In this section, our aim is to show that $\bdG$ can in turn be deformed homotopically to a surface $\bdG'$ that will enable us to make an asymptotic estimation of the integral (\ref{eq:FourierIntegralIndented}) as $r\to\infty$ for a given observation direction $\tilde{\bdx}$ defined such that $\bdx=r \tilde{\bdx}$. Note that since the deformation is homotopic in $\mathbb{C}^2 \setminus \sigma$, then, by Proposition \ref{pr:bridge_invariant}, the two surfaces $\bdG$ and $\bdG'$ have the same bridge configuration.

More specifically, in the rest of the article, we will show that for
a given observation direction $\tilde{\tmmathbf{x}}$ and a given integral
representation (\ref{eq:FourierIntegralIndented}), 
some points $\tmmathbf{\xi}
\in \mathbb{R}^2$ do lead to an asymptotic contribution to the 
far-field of $u(\tmmathbf{x})$ that is not exponentially vanishing 
as $r\to\infty$, while others do not. 
The points leading to these non-vanishing terms are referred to as {\em contributing},
while the others are said to be {\em non-contributing}.
We will see below that, generally, there exists only a discrete set of contributing points, so that the locality principle is fulfilled for the integral: 
{\em the whole  integral can be estimated by studying neighbourhoods of several
	discrete points.}

%

\subsection{Definition of active and inactive points}

A sufficient condition for a point $\bdxi^\star$ to be non-contributing, is for the integrand of (\ref{eq:FourierIntegralIndented}) to be exponentially vanishing as $r\to\infty$ when $\bdxi$ belongs to a patch of the integration surface lying above a neighbourhood of $\bdxi^\star$. This idea motivates the introduction of the so-called active and inactive points, a rigorous definition of which will be given in Definition \ref{def:acitveandinactive}.

Consider a suitable flattable surface $\bdG$ with defining vector field $\bdeta$ together with a point $\bdxi^\star \in\mathbb{R}^2$ and parametrising complex position vector $\bdxi_{\bdG}$. The exponential factor of the integrand of (\ref{eq:FourierIntegralIndented}) can be estimated at this point as follows:
\begin{align}
	| e^{- i\tmmathbf{x} \cdot \tmmathbf{\xi}_{\tmmathbf{\bdG}}
		(\tmmathbf{\xi}^{\star})} |  =  | e^{- ir \tilde{\tmmathbf{x}} \cdot
		(\tmmathbf{\xi}^{\star} + i \tmmathbf{\eta}(\bfxi^\star))} | =
	e^{ r \, \tilde{\tmmathbf{x}} \cdot \tmmathbf{\eta} (\bfxi^\star)}. 
	\label{eq:exponentialgrowthatpoint}
\end{align}
Hence, the most important quantity is $\tilde{\tmmathbf{x}} \cdot
\tmmathbf{\eta} (\bfxi^\star)$. If it is strictly negative the
integrand in (\ref{eq:FourierIntegralIndented}) is exponentially vanishing as $r \to \infty$. We will hence say that the integrand is \emph{exponentially small} 
at all points $\bdxi\in\mathbb{R}^2$ where 
\[
\tmmathbf{\tilde x} \cdot \tmmathbf{\eta}(\bdxi) < 0.
\]
Since
for any compact domain with this property 
one can take $r$ large enough to make the corresponding portion of the integral 
small enough, any such point is non-contributing. This leads to the following definition.

\begin{defi}
	Consider a physical field $u$ and a \RED{function} $F$ with the real
	property related by the integral
	(\ref{eq:FourierIntegralIndented}) for some suitable flattable surface of integration $\bdG$, together with a given observation
	direction $\tilde{\tmmathbf{x}} =\tmmathbf{x}/ r$ and a point $\bdxi^\star\in\mathbb{R}^2$.
	If $\bdG$ can be locally and
	continuously deformed without crossing any singularities of $F$ 
	into a surface $\bdG'$ with defining vector field $\bdeta'$ such that $\tilde{\bdx}\cdot\bdeta'(\bdxi^\star)<0$, we say that $\bdxi^\star$ is \emph{inactive}. A point $\tmmathbf{\xi}^{\star}$ that is not inactive is said to be \emph{active}.
	\label{def:acitveandinactive}
\end{defi}

With this definition and the discussion above, it should be clear that inactive points are non-contributing. The important point here is that our definition of inactivity is local and relatively easy to use. Indeed, to prove the inactivity of a point one should simply 
give an example of a proper local deformation of the integration surface. This concept of active and inactive points has been introduced and used
effectively in {\cite{Ice2021}} in a specific case. With the present article,
we wish to apply this notion in a more general context and for a wider set of
configurations.

In Sections \ref{sec:nearactive} and \ref{sec:integrationawayfromactive}, we will show that, thanks to these local considerations, we can homotopically deform the whole of $\bdG$ into a surface $\bdG'$, on which the integrand is only not exponentially small in the neighbourhoods of some active points. But first, we will provide a list of results regarding the activity or inactivity of points in $\mathbb{R}^2$.


\subsection{Non-singular points are all inactive}
Let $\bfxi^\star \in \mathbb{R}^2$ be a point that does not belong to $\sigma'$, the real trace of the singularity set of $F$, together with a suitable integration surface $\bdG$ defined by a vector field $\bdeta$. If $\tilde{\bdx}\cdot\bdeta(\bdxi^\star)<0$, then this point is inactive. Else, choose a constant vector $\bdeta'$ small enough such that the surface patch parametrised by $\bdxi^r+i\bdeta'$ for $\bdxi^r$ in a small neighbourhood of $\bdxi^\star$ is flattable (always possible since $\bdxi^\star\notin\sigma'$, see Theorem \ref{th:flattened}). Then the linear transformation passing from the initial patch of $\bdG$ to the new one does not cross the singularity either. Hence, by definition \ref{def:acitveandinactive}, $\bdxi^\star$ is also inactive in that case. This can be summarised as follows:

\begin{proposition}
	If $\bdxi^\star\notin\sigma'$, then $\bdxi^\star$ is inactive. In other words, any non-singular point is inactive.
\end{proposition}

%
%
%

\subsection{Points belonging to a single singularity trace} \label{sec:pointonsinglesingu}
Having dealt with the case of non-singular points, let us now consider $\bdxi^\star$ belonging to an irreducible real trace component $\sigma_j'$ and to no other irreducible components. We also assume that a bridge configuration has been chosen to describe the behaviour of $\bdG$ near $\sigma_j'$. Let $g_j$ be the defining function of $\sigma_j$ and let $\tva^\star$ and $\tvb^\star$ be defined as in (\ref{eq:realproperty}). Our aim here is to decide whether $\bfxi^\star$ is active or not.


\begin{proposition}
\label{pr:non_orthogonal}
If $\tmmathbf{\tilde x}$ is not orthogonal to $\sigma_j'$
at $\bfxi^\star$, then $\bfxi^\star$ is inactive.
\end{proposition}

\begin{proof}
Let $\bfeta^\star$ be the value of $\bfeta$ (describing
$\bdG$) at $\bfxi^\star$.
If $\bfeta^\star$ obeys the inequality 
$\bfeta^\star \cdot \tmmathbf{\tilde x} <0$
then no deformation is needed, and the point $\bfxi^\star$ is inactive. 
If $\bfeta^\star \cdot \tmmathbf{\tilde x} \ge 0$, then choose 
a small enough vector $\bfeta'$, such that: 
a) $\bfeta'$ points to the same size of $\sigma_j'$
as $\bfeta^\star$; b) $\bfeta' \cdot \tmmathbf{\tilde x} < 0$
(see Figure~\ref{fig:non_orthogonal}, left).
Indeed, since $\tmmathbf{\tilde x}$ is not orthogonal to $\sigma_j'$, 
then it is always possible to select such a vector $\bfeta'$. 
Consider a small neighbourhood of $\bfxi^\star$ such that $\bdG$
and $\bdG'$ can be described in it locally by constant vectors 
$\bfeta^\star$ and $\bfeta'$, respectively. Such a neighbourhood exists by Theorem \ref{th:flattened} and the discussion in its proof.

\begin{figure}[h]
  \centering{
  \includegraphics[width=0.8\textwidth]{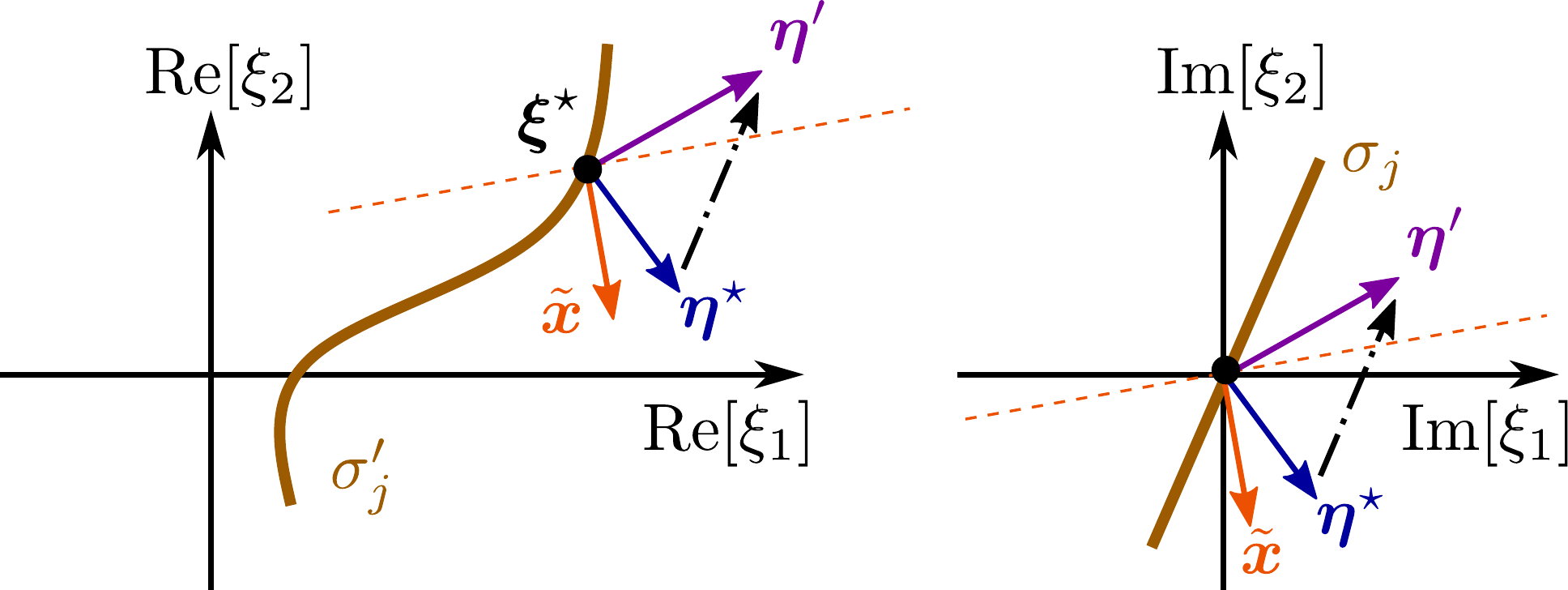}}
  \caption{Deformation of $\bdG$ into $\bdG'$ if $\tmmathbf{\tilde x}$
  is not orthogonal to $\sigma_j'$}
\label{fig:non_orthogonal}
\end{figure}

The local deformation of $\bdG$ into $\bdG'$ is described by a black dot-dash 
arrow in Figure~\ref{fig:non_orthogonal}. We mean that 
on each stage of the deformation the intermediate surface is described 
locally by a constant vector $\bfeta$, whose end 
slides along the dot-dash arrow connecting 
$\bfeta^\star$ and $\bfeta'$. 
Let us show that the surface of integration does not hit the 
singularity in the course of this deformation. 

Let $\mathcal{U}$ be some small complex neighbourhood of $\bfxi^\star$.
Consider the points of $\sigma_j \cap \mathcal{U}$ 
and plot them in the $({\rm Im}[\xi_1] , {\rm Im}[\xi_2])$ real plane
(see Figure~\ref{fig:non_orthogonal}, right).
These points are grouped near a line defined by 
\[
\tva^\star {\rm Im}[\xi_1]+\tvb^\star{\rm Im}[\xi_2]=0 .
\]
They may not belong exactly to this line, since, generally, one has 
to take into account quadratic and higher order terms in the Taylor series for $g_j$ in $\mathcal{U}$. 
Note that the graph of $\sigma_j'$ in the left part of Figure~\ref{fig:non_orthogonal} is locally approximated by
\[
\tva^\star(\xi_1^r - \xi_1^\star)+\tvb^\star(\xi_2^r - \xi_2^\star)=0,
\] 
i.e.\ the slope of the tangent line is the same as in the right part of the figure. 

In the $({\rm Im}[\xi_1] , {\rm Im}[\xi_2])$ plane, the surface $\bdG \cap \mathcal{U}$
is shown by a single point, which is the end of the vector $\bfeta^\star$.
Similarly, $\bdG' \cap \mathcal{U}$ is the point corresponding to the vector $\bfeta'$. All intermediate surface 
emerging during the deformation 
$\bdG \to \bdG'$ are shown by the dot-dash arrow in Figure \ref{fig:non_orthogonal} (right). Since this arrow is not crossing the set $\sigma_j$, the singularity cannot be hit during this deformation. 

The simple reasoning outlined above is in fact the key idea of the paper. 
On the one hand, if the vectors $\bfeta^\star$ and $\bfeta'$ point to the same side 
of the $\sigma_j$-line in the right side of Figure~\ref{fig:non_orthogonal} 
(or, what is the same, to the same side of $\sigma_j'$ in  the left side of the figure), the deformation 
$\bdG \to \bdG'$ is locally ``safe'', i.e.\ the singularity is not hit. 
On the other hand, the fact that these vectors point to the same size of $\sigma_j'$
means that the bridge symbols of $\bdG$ and $\bdG'$ with respect to $\sigma_j'$
are the same. Hence the new surface $\bdG'$ is also admissible from a physical point of view, as per the point b) of Section~\ref{sec:surfaceparam}. 

We can make a conjecture that the inverse statement is also valid: if the bridge symbol 
is changed, the deformation cannot be admissible. We believe that this is true, but since this 
statement is not needed below, we omit the proof. 
\end{proof}

We will now conclude this section with the two situations that have not yet been considered.

\begin{proposition}
\label{pr:orthogonal_1}
Let $\tmmathbf{\tilde x}$  be orthogonal 
to $\sigma_j'$ at a point $\bfxi^\star$. If the bridge configuration for $\sigma_j'$ is chosen in such a way 
that $\bfeta^\star \cdot \tmmathbf{\tilde x} < 0$, where $\bdeta^\star=\bdeta(\bdxi^\star)$, then $\bfxi^\star$ is inactive.   
\end{proposition}

This proposition is indeed trivial (no deformation is needed). The vectors 
$\tmmathbf{\tilde x}$, $\bfeta^\star$, and the bridge notation are shown in 
Figure~\ref{fig:orthogonal_1} (left).

\begin{proposition}
\label{pr:orthogonal_2}
Let $\tmmathbf{\tilde x}$  be orthogonal 
to $\sigma_j'$ at a point $\bfxi^\star$. If the bridge configuration for $\sigma_j'$ is chosen in such a way 
that $\bfeta^\star \cdot \tmmathbf{\tilde x} > 0$, where $\bdeta^\star=\bdeta(\bdxi^\star)$, then $\bfxi^\star$ is active.   
\end{proposition}

\begin{proof}
Consider the restriction of the function $\bfeta'$
(describing $\bdG'$) onto the set $\sigma_j'$. For the deformation to be valid, the bridge configuration of $\bdG$ and $\bdG'$ should be the same (by Proposition \ref{pr:bridge_invariant}). Hence, $\bdeta$ and $\bdeta'$ should point to the same side of $\sigma'_j$ as $\bfeta^\star$. Hence, since $\tilde{\bdx}$ is orthogonal to $\sigma_j'$ at $\bdxi^\star$, $\bdeta\cdot\tilde{\bdx}$ and $\bdeta'\cdot\tilde{\bdx}$ should have the same sign. This implies that in this case we cannot have a deformation leading to $\bdeta'\cdot\tilde{\bdx}<0$, so $\bdxi^\star$ is active. An illustration 
can be found in Figure~\ref{fig:orthogonal_1} (right). One can see that 
one could in principle define a valid vector $\bfeta'$ obeying the necessary properties 
everywhere on $\sigma_j'$ except on $\bfxi^\star$.	
\end{proof}

Note that here we do not conclude regarding the contributing nature of this active point. This will be done later when we build an asymptotics approximation of the integral in a neighbourhood of such point. We will see then that such a point is indeed contributing, i.e.\ the resulting asymptotic component is not exponentially small.

   
   \begin{figure}[h]
   	\centering{
   		\includegraphics[width=0.5\textwidth]{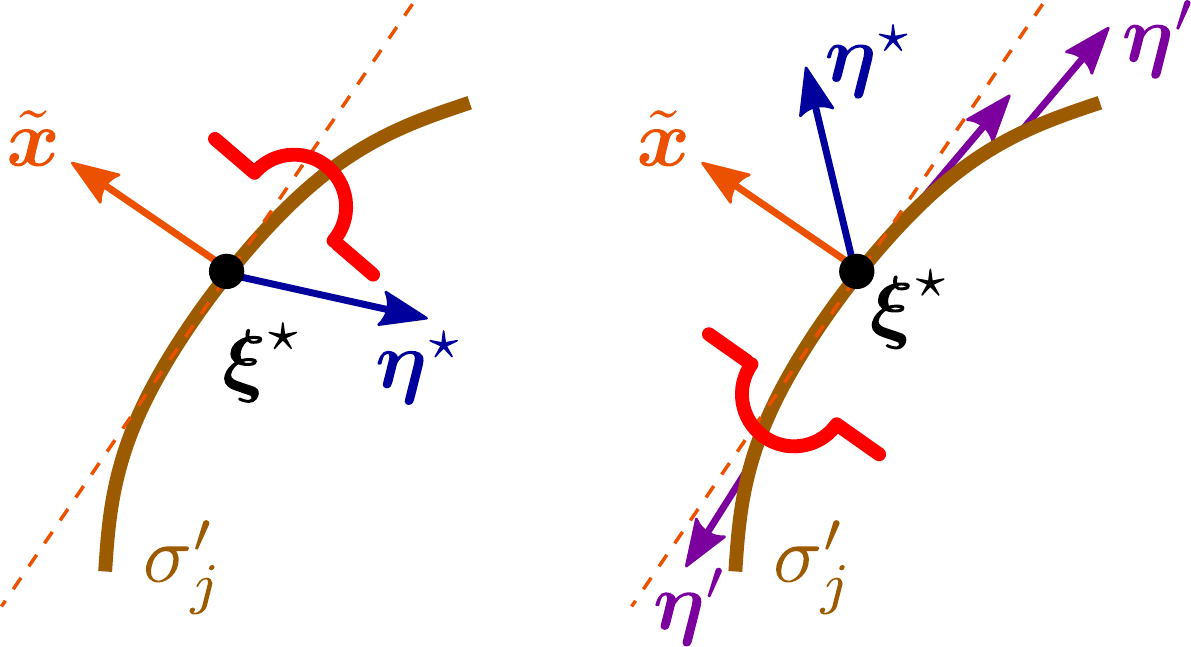}}
   	\caption{The vectors $\tmmathbf{\tilde x}$, $\bfeta^\star$, and the bridge notation for 
   		Propositions~\ref{pr:orthogonal_1} (left) and \ref{pr:orthogonal_2} (right)}
   	\label{fig:orthogonal_1}
   \end{figure}

Before finishing this section, we provide an alternative characterisation of what it means for $\tmmathbf{\tilde x}$ to  be orthogonal 
to $\sigma_j'$ at a point $\bfxi^\star$, as in the case in Propositions \ref{pr:orthogonal_1} and \ref{pr:orthogonal_2}.

\begin{defi}[\RED{Saddle on a singularity, SOS}]
	\label{def:saddle}
	Since $\sigma_j$ is a \RED{complex} manifold, it can, at least
	locally, be parametrised using only one complex variable $\psi$ say.
	That is, in the neighbourhood of a point $\bdxi^\star\in\sigma_j$ say, we have $\tmmathbf{\xi}= (\xi_1, \xi_2) \in \sigma$
	implies that $\bdxi=\bdxi^{\sigma_j}(\psi)=(\xi_1^{\sigma_j} (\psi),\xi_2^{\sigma_j}(\psi))$, where $\xi_{1, 2}^{\sigma_j}$ are holomorphic functions of $\psi$.
	Let us denote by $\psi^{\star}$ the value of $\psi$ such that $\bdxi^{\sigma_j} (\psi^{\star}) = \bdxi^\star$, and take a vector $\bdx=r \tilde{\bdx}\in\mathbb{R}^2$. Consider now the function $h (\psi) \overset{\tmop{def}}{=} e^{-i\bdx\cdot\bdxi^{\sigma_j}(\psi)}$ and its derivative $h'(\psi)$. We say that $\tmmathbf{\xi}^{\star}$ is a {\em saddle on \RED{the singularity} $\sigma_j$ with respect to $\tilde{\bdx}$} if $h' (\psi^{\star}) = 0$. \RED{For brevity, and to avoid potential confusion with usual multidimensional saddle points, we will say that $\bdxi^\star$ is a SOS on $\sigma_j$ with respect to $\tilde{\bdx}$.}
\end{defi}

We can now formulate and prove the sought-after characterisation.

\begin{proposition}
	\label{prop:perpsaddle}$\tilde{\tmmathbf{x}}$ is orthogonal to $\sigma_j'$ at
	$\tmmathbf{\xi}^{\star}$ if and only if $\tmmathbf{\xi}^{\star}$ is a \RED{SOS} on $\sigma_j$ with respect to $\tilde{\bdx}$.
\end{proposition}

\begin{proof}
	Let $g_j$ be the defining function of $\sigma_j$ and $\tva^\star$ and $\tvb^\star$ be defined as in (\ref{eq:realproperty}). Since $g_j$ is
	holomorphic (i.e.\ $\bar{\partial} g_j = 0$), we have $\mathd g_j (\tmmathbf{\xi}^{\star})=\text{\tmverbatim{a}}^{\star}
	\mathd \xi_1 + \tvb^{\star} \mathd \xi_2$.
	Now, if we assume that $(\xi_1, \xi_2)$ is restricted to lie on $\sigma_j$,
	that is $\xi_{1, 2} = \xi_{1, 2}^{\sigma_j} (\psi)$ for some analytic
	functions $\xi_{1, 2}^{\sigma_j}$ of a complex variable $\psi$ and let
	$\psi^{\star}$ be the value of $\psi$ such that $\xi_{1, 2}^{\sigma_j}
	(\psi^{\star}) = \xi_{1, 2}^{\star}$, then $0 = \mathd g_j|_{\sigma_j} = \left(
	\tva^{\star} (\xi_1^{\sigma_j})' (\psi^{\star}) +
	\tvb^{\star} (\xi_2^{\sigma_j})' (\psi^{\star}) \right) \mathd \psi$
	and we have
	\begin{align}
		\tva^{\star} (\xi_1^{\sigma_j})' (\psi^{\star}) +
		\tvb^{\star} (\xi_2^{\sigma_j})' (\psi^{\star})  =  0. 
		\label{eq:xi1primexi2primeab}
	\end{align}
	Let us now consider the function $h (\psi) = e^{-i\bdx\cdot\bdxi^{\sigma_j}(\psi)}$. Naturally, we have
	\begin{align}
		h'(\psi^{\star})  =  - i (x_1
		(\xi_1^{\sigma_j})' (\psi^{\star}) + x_2 (\xi_2^{\sigma_j})' (\psi^{\star}))
		e^{- i\tmmathbf{x} \cdot \tmmathbf{\xi}^{\star}}.
		\label{eq:simplederhdash}
	\end{align}
	We can hence show that
	\begin{eqnarray*}
		h'(\psi^{\star}) = 0 & \underset{\left( \ref{eq:simplederhdash} \right)}{\Leftrightarrow} &  x_1
		(\xi_1^{\sigma_j})' (\psi^{\star}) + x_2 (\xi_2^{\sigma_j})' (\psi^{\star})
		= 0 \ 
		 \underset{\left( \ref{eq:xi1primexi2primeab} \right)}{\Leftrightarrow} \
		- \tvb^{\star} x_1 + \text{\tmverbatim{a}}^{\star} x_2 = 0 \\
		& \Leftrightarrow & \tmmathbf{x} \cdot \left( \begin{array}{c}
			- \tvb^{\star}\\
			\text{\tmverbatim{a}}^{\star}
		\end{array} \right) = 0 \ \Leftrightarrow \ \tmmathbf{x} \perp \sigma'_j \  \Leftrightarrow \ \tilde{\bdx} \perp \sigma'_j,
	\end{eqnarray*}
	as expected, since we saw in Section \ref{sec:bridgeandarrow} that the
	vector $\left(- \tvb^{\star},	\tva^{\star} \right)^\transpose$ was tangent to $\sigma_j'$. Moreover, we have used the fact that since the point $\bdxi^\star$ is regular, then at least one of the $(\xi^{\sigma_j}_{1,2})'(\psi^\star)$ is non-zero. 
\end{proof}

%

We have therefore seen that the points on $\sigma_j'$ that are not crossing points of several 
\RED{irreducible singularities} can only be active if they are \RED{SOS} as per the Definition \ref{def:saddle}. 
\RED{Note also that the definition of an SOS is different to the usual definition of saddle points for multiple integrals given in introduction. However, there exists a link between the two as discussed briefly in Remark \ref{rem:polynomialjoneslighthill} and also in Section \ref{subsec:2d saddle}.}

\RED{
	\begin{rema}
		\label{rem:polynomialjoneslighthill}
		Both Lighthill \cite{Lighthill78} (Chapter 4.9) and Jones \cite{Jones1982-xs} (Chapter 9.7), when studying the multiple Fourier Transform of the reciprocal of a polynomial in several variables, also had to deal with curved singularities. They realised the importance of what we
		call \RED{SOS}, as well as the fact that
		such points will only lead to a wave field in a direction perpendicular to the singular curve at this point. Their method effectively reduces the study of
		such \RED{SOS} to the study of a usual saddle point of a
		lower-dimensional integral. 
		
		It is also interesting to note that Lighthill discusses a physically-motivated notion of activity of such point by decomposing the singular curve of interest into its active and inactive part. A notion that we render precise here with Proposition \ref{pr:orthogonal_2}.
		
		However, their technique relies on the use of the
		residue theorem, which cannot be used for non-polynomial singularities. Moreover, they do not consider intersections of singularity curves, a situation that we address in the next section.
	\end{rema}
	
}


\subsection{Transverse intersection of two singularity
traces}\label{sec:transverseactivequadrant}

Let us consider two irreducible singularities of $F$, $\sigma_1$ and $\sigma_2$, with defining functions $g_{1,2}$ and real traces $\sigma_{1,2}'$. Let us assume that
$\sigma_1'$ and $\sigma_2'$ cross transversely at a point
$\tmmathbf{\xi}^{\star}$. As per (\ref{eq:realproperty}), we can define the
real quantities
\begin{align}
  \frac{\partial g_{1, 2}}{\partial \xi_1} (\tmmathbf{\xi}^{\star}) =
  \tva^{\star}_{1, 2} \in \mathbb{R}, \hspace{1em}  \frac{\partial g_{1,
  2}}{\partial \xi_2} (\tmmathbf{\xi}^{\star}) = \tvb^{\star}_{1, 2} \in
  \mathbb{R}, \hspace{1em} \left( \tva^{\star}_{1, 2}
  \right)^2 + \left( \tvb^{\star}_{1, 2} \right)^2 \neq 0, 
  \label{eq:realpropertytransverse}
\end{align}
that give some information on the slopes of the real traces near
$\tmmathbf{\xi}^{\star}$. The transverse crossing requirement means that the
two vectors $\left(\tva^{\star}_1, \tvb^{\star}_1\right)^\transpose$ and $\left( \tva^{\star}_2, \tvb^{\star}_2 \right)^\transpose$ are not co-linear.

Let us assume that a bridge symbol is specified for both singularities. As discussed in Section \ref{sec:bridgeandarrow}, this choice fixes
the sides toward which the vector fields $\tmmathbf{\eta}_1$ and
$\tmmathbf{\eta}_2$ are pointing, where $\tmmathbf{\eta}_{1, 2}$ are defined
as $\tmmathbf{\eta}$ above for any point on $\sigma_{1, 2}'$. This is
illustrated in Figure \ref{fig:transverse-crossing-setup} (left).

\begin{figure}[h]
  \centering{\includegraphics[width=0.6\textwidth]{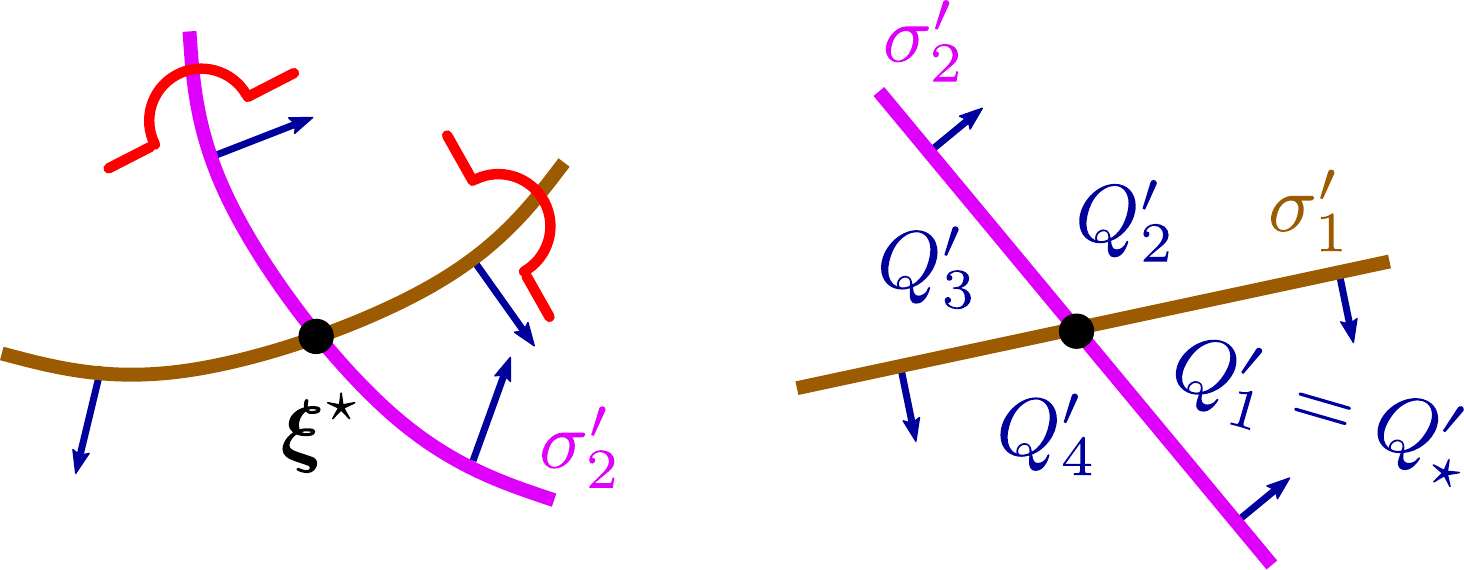}}
  \caption{A bridge configuration of a transverse crossing (left), and definition of the quadrants (right).}
  \label{fig:transverse-crossing-setup}
\end{figure}

What does it tell us about the actual direction of $\tmmathbf{\eta}^{\star}=\bdeta(\bdxi^\star)$? Well, it should be consistent with both the
direction of $\tmmathbf{\eta}_1^{}$ and that of $\tmmathbf{\eta}_2$. In the
configuration of Figure \ref{fig:transverse-crossing-setup}, it is clear that
$\tmmathbf{\eta}^{\star}$ can only point toward one of the four quadrants
$Q_1', \ldots, Q_4'$ defined by the linear approximation of the real traces at
the crossing. This quadrant, that we denote $Q'_\star$, is the one that has arrows of both traces pointing
toward it. For the bridge configuration of Figure \ref{fig:transverse-crossing-setup} (right), we have $Q'_\star=Q_1'$.

Before moving further, it is useful to also draw the lines perpendicular to
$\sigma_1'$ and $\sigma_2'$ at $\tmmathbf{\xi}^{\star}$. To do so, we will use
dashed lines with the same colour as the singularity it corresponds to. These
dashed lines divide the plane into 4 other quadrants that we call
$Q_1^{\perp}, \ldots, Q_4^{\perp}$, as illustrated in Figure \ref{fig:Qperp}. We will call these the
{\em perpendicular quadrants}. Let us also denote by $Q_{\star}^{\perp}$ the
perpendicular quadrant that contains the bisector of $Q_{\star}'$
(see Figure~\ref{fig:Qperp}, right).

\begin{figure}[h]
  \centering{\includegraphics[width=0.6\textwidth]{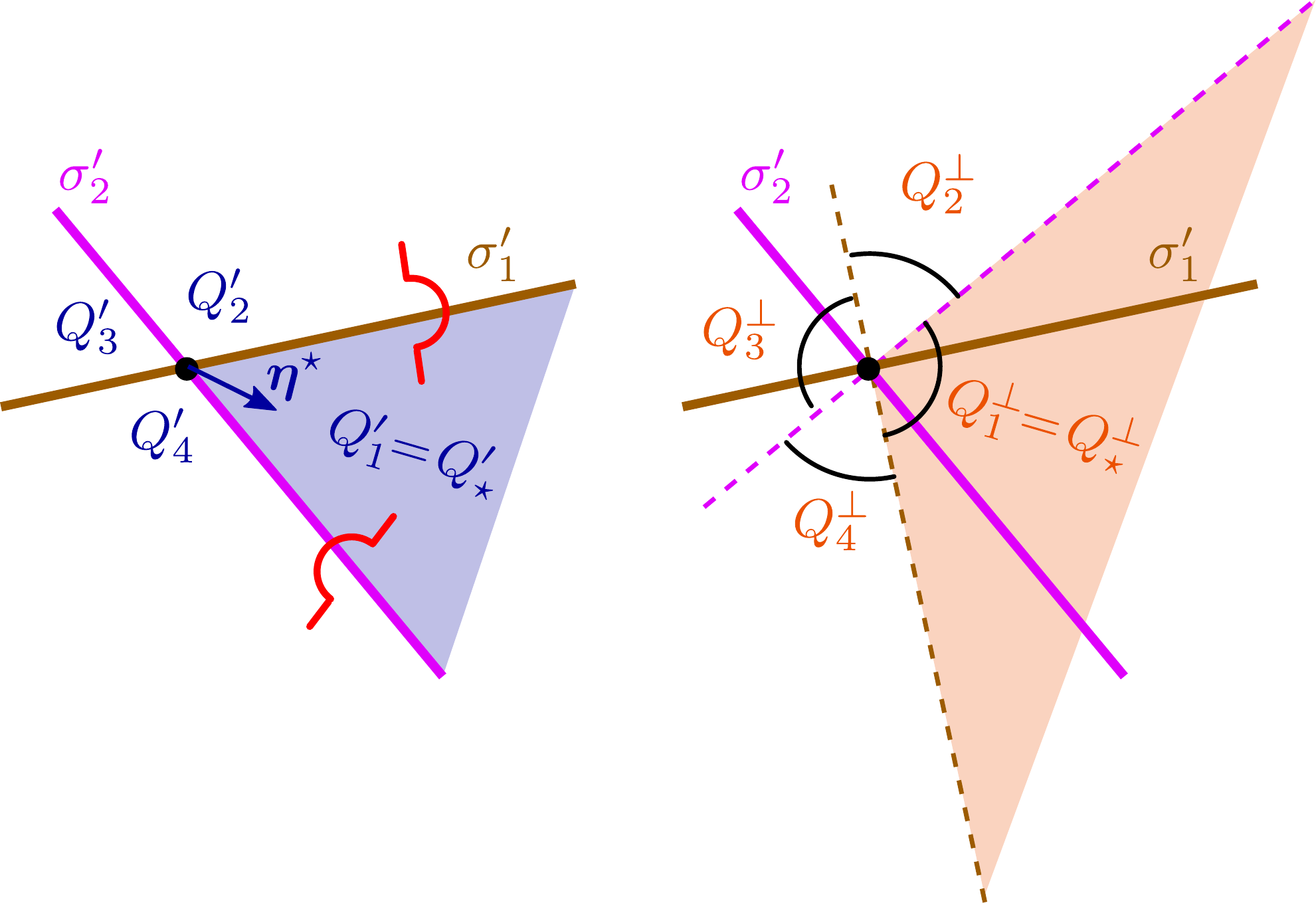}}
  \caption{Definition of the quadrants $Q_1', \ldots, Q_4'$ and $Q_{\star}'$ (left);
  definition of the quadrants   
  $Q_1^{\perp}, \ldots, Q_4^{\perp}$ and $Q_{\star}^{\perp}$ (right).}
  \label{fig:Qperp}
\end{figure}

We are now in a position to formulate (and prove) the main theorem of this
section regarding the activity (or inactivity) of such a transverse crossing
point.

%


\begin{theorem}
\label{th:crossing}
  The transverse crossing point of $\sigma_1'$ and $\sigma_2'$ is active if and only if 
  $\tilde{\tmmathbf{x}}$ points to the quadrant $Q_{\star}^{\perp}$. 
\end{theorem}

\begin{proof}
The proof is very similar to those of Propositions \ref{pr:non_orthogonal}--\ref{pr:orthogonal_2}. Namely, if $\bdeta^\star\cdot\tilde{\bdx}<0$ there is nothing to do and the point is automatically inactive. Else, choose a vector $\bdeta'$ satisfying 	$\bdeta'\cdot\tilde{\bdx}<0$ and such that it does not change the initial bridge configuration. This is always possible if $\tilde{\bdx}$ does not point to $Q^\perp_\star$ and impossible if it does. We note also that if $\tilde{\tmmathbf{x}}$ is on the boundary of $Q^\perp_\star$, then $\tmmathbf{\xi}^{\star}$ is automatically a \RED{SOS} for one of the two singularities, and the point is active.	
\end{proof}

We will hence call $Q^\perp_\star$ the {\em active quadrant} of the crossing. For any such transverse crossing, and for every possible bridge configuration, we always have
one active and three inactive quadrants. All possibilities are summarised in
Figure \ref{fig:allbridgetransversecrossing}. The symbol \mycircle{A} means active.

\begin{figure}[h]
 \centering{\includegraphics[width=0.99\textwidth]{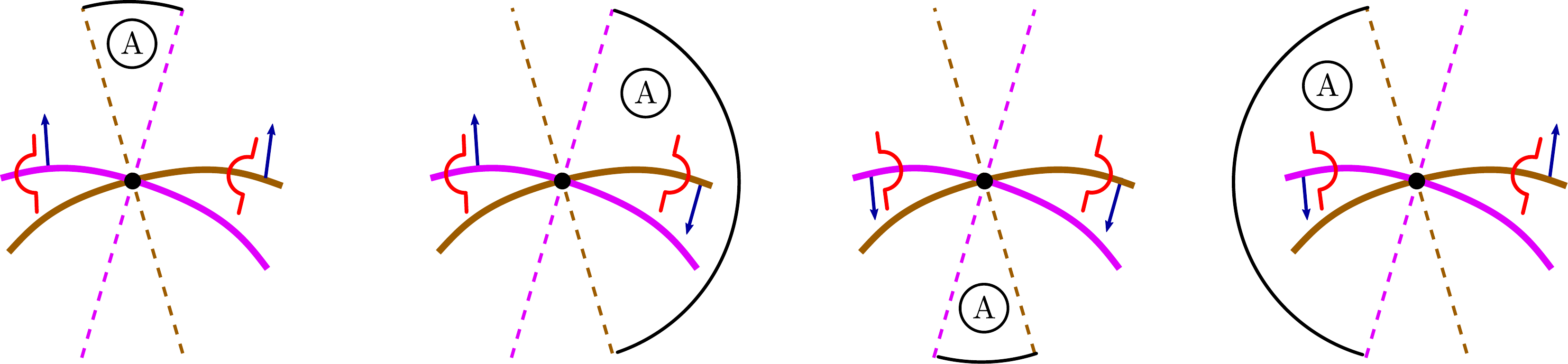}}
  \caption{Active regions of a crossing for all possible bridge
  configurations}
\label{fig:allbridgetransversecrossing}
\end{figure}

Note that, once again, we know that the inactive points are non-contributing, but we do not conclude regarding the contributing nature of the active crossing points. This will be done later when we build an asymptotic approximation of the integral in a neighbourhood of such point. We will see then that such active points are indeed contributing, as long as they do not have the so-called additive property.


\subsection{Tangential crossing of two singularity traces}

Let us again consider two irreducible singularities of $F$, $\sigma_{1,2}$, with defining functions $g_{1,2}$ and real traces $\sigma_{1,2}'$. Let us however assume that,
this time, $\sigma_1'$ and $\sigma_2'$ cross (or touch) tangentially at a point
$\tmmathbf{\xi}^{\star}$. As usual the real quantities $\tva^\star_{1,2}$ and $\tvb^\star_{1,2}$, introduced as in (\ref{eq:realpropertytransverse}), give us some information on the slope of the real traces near
$\tmmathbf{\xi}^{\star}$. In this case, the two slopes should be the same meaning
that $\tva_1^\star \tvb_2^\star=\tvb_1^\star \tva_2^\star$, and as discussed in point (v) of Section \ref{sec:introbridgesymbol}, we can only
have two possible bridge configurations given in Figure
\ref{fig:bridge-tangential}.  
We can now formulate (and prove) the main result of this section.

\begin{theorem}
  A tangential intersection $\bdxi^\star$ of two real traces $\sigma_{1,2}'$ is
  inactive for all but one direction~$\tilde{\tmmathbf{x}}$. The only
  direction $\tilde{\bdx}$ for which $\bdxi^\star$ is active is perpendicular to $\sigma_{1,2}'$ at $\bdxi^\star$, i.e.\  it is such that $\bdxi^\star$ is an active \RED{SOS} on $\sigma_{1,2}$ with respect to this direction $\tilde{\bdx}$.
  \label{th:th20tangentperp}
\end{theorem}

\begin{proof}
  Let us do this for the bridge configuration of Figure
  \ref{fig:bridge-tangential} (left) since the other configuration can be dealt with similarly. The proof is again similar to those of Propositions \ref{pr:non_orthogonal}--\ref{pr:orthogonal_2} and Theorem \ref{th:crossing}. It is clear
  that $\tmmathbf{\eta}^{\star}=\bdeta(\bdxi^\star)$ should point toward the ``interior''
  of the $\sigma_2'$ curve. If, as illustrated in Figure \ref{fig:active-tangential} (left), $\tmmathbf{\eta}^{\star} \cdot \tilde{\tmmathbf{x}}<0$, then nothing needs to be done and the point is inactive. If instead $\tmmathbf{\eta}^{\star} \cdot \tilde{\tmmathbf{x}}\geq0$ and $\tilde{\tmmathbf{x}}$ is not perpendicular to the real traces at $\bdxi^\star$,  as illustrated in Figure \ref{fig:active-tangential} (centre), we can choose a new vector $\bdeta'$ that does not change the bridge configuration and is such that $\tmmathbf{\eta}' \cdot \tilde{\tmmathbf{x}}<0$, i.e.\ the point $\bdxi^\star$ is inactive. Finally, if $\tmmathbf{\eta}^{\star} \cdot \tilde{\tmmathbf{x}}>0$ and $\tilde{\bdx} \perp \sigma_{1,2}'$ at $\bdxi^\star$ then it is impossible to find a suitable vector $\bdeta'$, as illustrated in Figure
  \ref{fig:active-tangential} (right), and $\bdxi^\star$ is active in that case. 
 \begin{figure}[h]
    \centering{\includegraphics[width=0.7\textwidth]{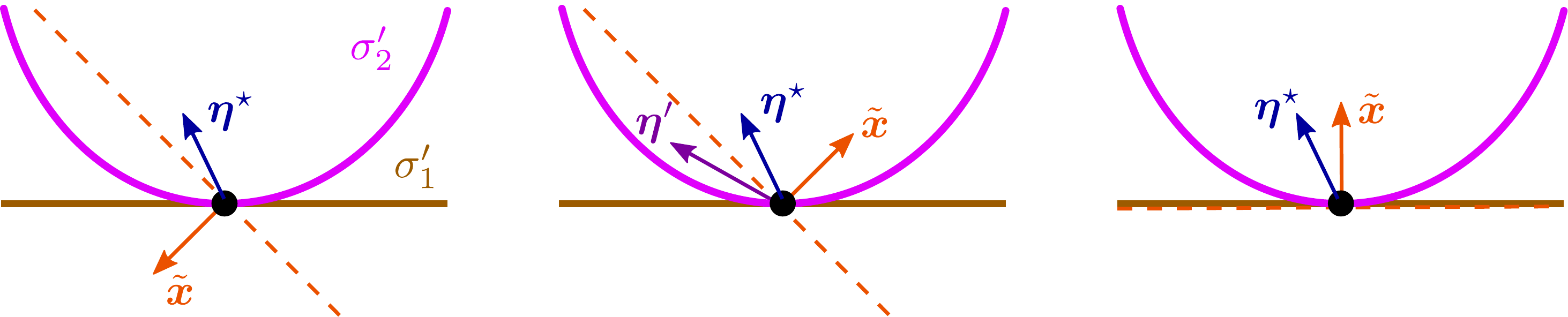}}
   \caption{Visualisation of the proof of Theorem \ref{th:th20tangentperp} for
    different directions $\tilde{\tmmathbf{x}}$.}
\label{fig:active-tangential}
  \end{figure}
\end{proof}

\begin{rema}
  \label{rem:notriple}For the sake of brevity and clarity, we will only
  consider the important (and most common) cases described in this Section
  \ref{sec:sec4}. It needs to be said, however, that more complicated
  situations may occur, such as a crossings involving three or more distinct
  singularities for example. If and when these occur in practice one will need
  to carefully repeat/adapt the arguments displayed in this section. In fact
  an example of a triple crossing will be given and dealt with in
  {\cite{Part6B}}.
\end{rema}

\subsection{Additive crossings}

The important concept of additive crossing for two singularities 
was introduced in \cite{Assier2018a}. 
Namely, consider a function $F$ with the real property and two of its singularities $\sigma_1$ and $\sigma_2$ with defining functions $g_{1,2}$. Assume that their real traces $\sigma_{1,2}'$ cross transversally at a point $\bdxi^\star$, i.e.\ we are in the situation of Section~\ref{sec:transverseactivequadrant}. \RED{We say that the function $F$ has the additive crossing property at $\bdxi^\star$} if, in some neighbourhood of $\bdxi^\star$, $F$ can be locally written
\begin{equation}
	F(\bfxi) = F_1 (\bfxi) + F_2 (\bfxi),
	\label{eq:additive_1}
\end{equation}
where
$F_1$ is regular at the singularity $\sigma_2$, and 
$F_2$ is regular at the singularity $\sigma_1$.


We demonstrated in \cite{Assier2018a}, by means of Puiseux series, that the following condition is sufficient for a crossing to be additive.

\begin{proposition}
\label{pr:additive_1}
Let $F$ be a function that grows algebraically near $\sigma_{1,2}$, with growth exponents \RED{strictly} bigger than~$-1$.
For some point $\bfxi\notin\sigma_1\cup\sigma_2$ near $\bfxi^\star$, let the values 
$F(\bfxi ; \sigma_1)$, $F(\bfxi ; \sigma_2)$ and $F(\bfxi ; \sigma_1 \sigma_2)$
denote the values of $F$ at $\bfxi$ after analytical continuations along some small loops
bypassing the singularities $\sigma_1$, $\sigma_2$, and, successively $\sigma_1$ and $\sigma_2$, 
all in the positive direction. 
$F(\bfxi)$ denotes the value before all bypasses. 
Then, $F$ has the additive crossing property \RED{at} $\bfxi^\star$ if
\begin{equation}
F(\bfxi) + F(\bfxi ; \sigma_1 \sigma_2) = 
F(\bfxi; \sigma_1 ) + F(\bfxi ; \sigma_2) 
\end{equation}   
for $\bfxi$ taken in some neighbourhood of $\bfxi^\star$. 
\end{proposition}

This proposition is a convenient tool to check whether a crossing of two branch lines 
is additive or not. But importantly, and more related to the topic of the present work, we can show the following:

\begin{theorem}
	\label{th:additivedoesnotcontribute}
Let $F$ be a function with the real property and with the additive crossing property for two singularities $\sigma_{1,2}$ and their real crossing $\bdxi^\star\in\sigma_1'\cap\sigma_2'$.
Consider an observation direction $\tmmathbf{\tilde x}$ that is not orthogonal to $\sigma_1'$
nor to $\sigma_2'$ at $\bfxi^\star$. Then, the point $\bdxi^\star$ is non-contributing to the asymptotics of the resulting physical field $u$.
\end{theorem}

\begin{proof}
Represent $F$ as $F_1 + F_2$ in some neighbourhood $U$ of $\bfxi^\star$, $F_1$ being regular at $\sigma_2$ and $F_2$ being regular at $\sigma_1$. 
Split the integral (\ref{eq:FourierIntegralIndented}) as follows: 
\[
u(\tmmathbf{x} ; k) = 
\iint_{\bdG \setminus U} F(\bfxi)e^{-i \tmmathbf{x} \cdot \bfxi} {\rm d} \bfxi + 
\iint_{\bdG \cap U} F_1 (\bfxi)e^{-i \tmmathbf{x} \cdot \bfxi} {\rm d} \bfxi + 
\iint_{\bdG \cap U} F_2 (\bfxi)e^{-i \tmmathbf{x} \cdot \bfxi} {\rm d} \bfxi .
\]
Consider the second and the third terms. According to Proposition~\ref{pr:non_orthogonal}, 
one can deform the integration surfaces for each of these terms making the 
integrand exponentially small. Note that the deformation for these terms may be different, 
i.e.\ there may not exist a {\em common} deformation for both integrals.
\end{proof}

The meaning of the theorem is important: even if such additive crossing point may formally be {\em active} as seen in Section \ref{sec:transverseactivequadrant}, it is always {\em non-contributing} and should hence be discarded when building the asymptotics of the physical field.

Let us summarise this section. The only special points that we need to consider are the only potentially contributing points, that is the \RED{SOS} points in the specific sense introduced in Section~\ref{sec:pointonsinglesingu} (see Proposition \ref{pr:orthogonal_2}) and the active transverse crossing points (see Theorem~\ref{th:crossing}) that do not have the additive crossing property (see Theorem \ref{th:additivedoesnotcontribute}). All the other points are non-contributing. 


%
%


\section{Estimation of the integrals near the active special points}\label{sec:nearactive}

Here we describe how the integral of the form (\ref{eq:FourierIntegralIndented}) can be estimated 
near the special points that have been declared potentially contributing in the previous section. In doing so, we will explicitly construct the corresponding leading term in the asymptotic expansion (as $r\to\infty$) of the physical field $u$. In particular, we will show that these leading terms are not exponentially vanishing, therefore proving that these special points are indeed contributing.

We adopt the following strategy to estimate the integral near a special point: 

\begin{itemize}
 
\item 
For such a special point $\bfxi^\star$, we take a neighbourhood $\mathcal{B}^\star
 \subset \mathbb{R}^2$ and build a (flattable)
deformed integration surface $\bdG'_\text{loc}$ over~$\mathcal{B}^\star$. We require that no other special point belong to this neighbourhood, and that the vector $\bfeta'$ describing 
$\bdG'_\text{loc}$ provides exponential decay over the boundary 
$\mathcal{C}^\star = \ptl \mathcal{B}^\star$, i.e.\
that 
\begin{equation}
\bfeta'(\bfxi^r) \cdot \tmmathbf{\tilde x} < 0
\quad 
\mbox{for}
\quad
\bfxi^r \in \mathcal{C}^\star. 
\label{eq:boundary}
\end{equation} 
This surface $\bdG'_\text{loc}$ should not cross the singularities of $F$ and it should have the same bridge configuration as $\bdG$. 

\item 
The integrand $F(\bfxi) \exp \{ - i \tmmathbf{x}\cdot\bfxi\}$ 
is simplified by means of a local approximation as $\bdxi \approx \bdxi^\star$.

\item
The integral of this simplified integrand is estimated.
Formally, the integral 
is taken over  $\bdG'_\text{loc}$, but the structure of the integral is such that we can extend the surface $\bdG'_\text{loc}$ over $\mathcal{B^\star}$ to a surface $\bdG'$ over the whole of $\mathbb{R}^2$ in such a way that the integral over $\bdG' \setminus \bdG'_\text{loc}$ is exponentially small. Thus, 
the integral over $\bdG'_\text{loc}$ can be replaced by the integral over 
$\bdG'$ without changing the leading term of the integral. 
\end{itemize}

For each considered case, this practical construction of $\bdeta'$ and $\bdG'_{\text{loc}}$ will be facilitated by an appropriate change of variables $\psi: (\xi_1,\xi_2) \rightarrow (\zeta_1,\zeta_2)$ and the construction of a vector field $\bdeta''$, a local surface $\bdG_{\text{loc}}''=\psi(\bdG'_{\text{loc}})$ and a neighbourhood $\mathcal{B}''=\psi(\mathcal{B}^\star)$.

\subsection{The case of a saddle on a singularity \RED{(SOS)}}\label{sec:awaysaddle}
\label{subsec:saddle}

Let $\sigma_j$ be an irreducible singularity component of $F$ with the real property and with defining function~$g_j$. Let
$\sigma_j'$ be its real trace, and let $\tmmathbf{\xi}^{\star} \in \sigma_j'$
be such that it does not belong to any other singularity components. Assume now that $\bdxi^\star$ is an active \RED{SOS} on $\sigma_j$ with respect to some observation direction $\tilde{\bdx}$, as defined in Section \ref{sec:pointonsinglesingu}. Assume also that $\bdxi^\star$ is an isolated active point, i.e.\ no other points are active in a neighbourhood of $\bdxi^\star$. As always, introduce the real quantities $\tva^\star$ and $\tvb^\star$ as in (\ref{eq:realproperty}). We will further assume that the function $g_j$ has a quadratic term, i.e.\ there exists real numbers $(\tilde{\alpha},\tilde{\beta},\tilde{\gamma})\neq(0,0,0)$ such that as $\bdxi\to\bdxi^\star$
\begin{align}
	g_j (\xi_1, \xi_2) 
	& =  \text{\tmverbatim{a}}^{\star} (\xi_1 -
	\xi_1^{\star}) + \text{\tmverbatim{b}}^{\star} (\xi_2 - \xi_2^{\star}) +
	\tilde{\alpha} (\xi_1 - \xi_1^{\star})^2 +\tilde{\beta}(\xi_2 - \xi_2^{\star})^2+\tilde{\gamma}(\xi_1 - \xi_1^{\star})(\xi_2 - \xi_2^{\star}) \nonumber \\ 
	&+\mathcal{O}(\text{h.t.}),
	\label{eq:gjquadratic}
\end{align}
where h.t. represents terms of order 3 and higher. 
Let us now introduce the quantities
\begin{align}
	\zeta_1 = \tvb^\star (\xi_1-\xi_1^\star)-\tva^\star(\xi_2-\xi_2^\star) ,\quad \zeta_2 = g_j(\xi_1 , \xi_2), \quad \Lambda=\text{\tmverbatim{a}}^{\star} (\xi_1 -
	\xi_1^{\star}) + \text{\tmverbatim{b}}^{\star} (\xi_2 - \xi_2^{\star})
	\label{eq:changeofvarzeta}
\end{align}
and consider the change the variables $\psi: \, (\xi_1 , \xi_2) \to (\zeta_1 , \zeta_2)$. The function $g_j$ having the real property, this change of variables is real on any real neighbourhood $\mathcal{B}^\star$ that is hence transformed by $\psi$ into some domain $\mathcal{B}''$, which is a real neighbourhood of zero. In the new variables, the singularity is simply defined by the line $\zeta_2 = 0$. The Jacobian of this transformation is equal to $((\tva^\star)^2+(\tvb^\star)^2)^{-1}$, and the unit vector $\bde_{\zeta_2}$ can be shown to be equal to $\bdn^\star$ at $\bdxi^\star$, where the unit vector $\bdn^\star$ is defined in (\ref{eq:ntvect}). Note that using the quantities (\ref{eq:changeofvarzeta}) and Appendix \ref{app:quadratic}, and noting that $\mathcal{O}(\zeta_2)=\mathcal{O}(\Lambda)$, we can show that there is a unique real number $\alpha$ (assumed thereafter to be non-zero) \RED{ proportional to the curvature of $\sigma_j'$ (as shown in Appendix \ref{app:quadratic})} such that (\ref{eq:gjquadratic}) can be rewritten
\begin{align}
	\zeta_2=\Lambda-\alpha \zeta_1^2+\mathcal{O}(\zeta_2^2+\zeta_1 \zeta_2+\text{h.t.}).
	\label{eq:grotatedvarsaddle}
\end{align}

We want to build a flattable surface $\tmmathbf{\Gamma}'_{\tmop{loc}}$ over a
real neighbourhood $\mathcal{B}^{\star}$ of $\tmmathbf{\xi}^{\star}$, such
that its defining vector field $\tmmathbf{\eta}'$ satisfies $\tmmathbf{\eta}'
\cdot \tilde{\tmmathbf{x}} < 0$ on $\mathcal{C}^{\star} = \partial
\mathcal{B}^{\star}$, with the aim of estimating the local integral
\begin{align}
	u_\text{loc}(\bdx)=\iint_{\bdG'_{\text{loc}}} F(\bdxi) e^{-i\bdx\cdot\bdxi}\, \mathd \bdxi = \iint_{\bdG''_{\text{loc}}}
	F(\bfxi(\tmmathbf{\zeta})) 
	e^{- i \bdx \cdot \bdxi(\tmmathbf \zeta)} 
	\left|\frac{\ptl(\xi_1 , \xi_2)}{\ptl(\zeta_1 , \zeta_2)}\right| 
	{\rm d} \tmmathbf{\zeta},
	\label{eq:localintfirst}
\end{align}
where $\bdG''_{\text{loc}}$ is a flattable surface over $\mathcal{B}''=\psi(\mathcal{B}^\star)$ in the $(\zeta_1 , \zeta_2)$
coordinates. Let this surface be described by the vector $\bfeta''$, which we will build explicitly.

Now, since $\bdxi^\star$ is an active \RED{SOS}, by Proposition \ref{prop:perpsaddle}, we know that $\tilde{\bdx}$ is orthogonal to $\sigma_j'$ at $\bdxi^\star$. Hence, we can write
\begin{align}
	\tilde{\bdx}=s \bdn^\star=s(\tva^\star,\tvb^\star)^\transpose/\sqrt{(\tva^\star)^2+(\tvb^\star)^2},
	\label{eq:xtildewiths}
\end{align}
where $s$ is either $+1$ or $-1$ and $\bdn^\star$ is the unit normal vector introduced in (\ref{eq:ntvect}). Each value of $s$ naturally corresponds to a different bridge configuration to ensure that the point is active. More precisely, $s=+1$ (resp. $s=-1$) corresponds to the bridge configuration where $\bdG$ bypasses $\sigma_j'$ from above (resp. below) at $\bdxi^\star$ in the $\bdn^\star$ direction, or , equivalently, in the $\bde_{\zeta_2}$ direction.

Using (\ref{eq:grotatedvarsaddle}) and (\ref{eq:xtildewiths}), we find that:
\begin{align}
	\tmmathbf{x} \cdot \tmmathbf{\xi} &=  \tmmathbf{x} \cdot \tmmathbf{\xi}^{\star} + sr^\star (\zeta_2 +
	\alpha \zeta_1^2)+\mathcal{O}(\zeta_2^2 + \zeta_1 \zeta_2+\text{h.t.}), \text{ with } r^\star=r/\sqrt{\left( \tva^{\star} \right)^2 +
		\left( \tvb^{\star} \right)^2},
	\label{eq:rewritexdotxiinzeta}
\end{align}

Let  $\mathcal{B}''$ be the disk of radius $\rho$ centered at $\boldsymbol{0}$: 
\[
\mathcal{B}'' = \{ \tmmathbf{\zeta}^r\in\mathbb{R}^2 : (\zeta_1^r)^2 + (\zeta_2^r)^2 \le \rho^2  \},
\]
and choose  $\mathcal{B}^\star$ to be $\psi^{-1} (\mathcal{B}'')$. 
Take $\rho$ small enough so that the
terms $\mathcal{O}(\zeta_2^2 + \zeta_1 \zeta_2+\text{h.t.})$ can be neglected.
Let us represent the surface $\bdG''_\text{loc}$ by the vector field $\bfeta'' (\tmmathbf{\zeta}^r)$ defined by 
\begin{equation}
	\eta''_1 = - \text{sign}(\alpha)\beta s \zeta^r_1,
	\qquad 
	\eta''_2 = |\alpha| \beta s (\rho^2 - 2 (\zeta^r_2)^2),
	\label{eq:good_surf}
\end{equation}
where $\tmmathbf{\zeta}^r$ is the real part of $\tmmathbf{\zeta}$, 
and $\beta$ is an arbitrary small positive value. 
Note that, using (\ref{eq:rewritexdotxiinzeta}) and (\ref{eq:good_surf}), for $\tmmathbf{\zeta}=(\zeta_1,\zeta_2)\in\bdG''_\text{loc}$, we have
\[
{\rm Im}[\bdx\cdot\bdxi(\bdzeta)]\approx {\rm Im}[s r^\star (\zeta_2 + \alpha \zeta_1^2) ] =  |\alpha| \beta r^\star (\rho^2 - 2 |\tmmathbf{\zeta}^r|^2). 
\] 
One can see that this value is strictly negative on the boundary of $\mathcal{B}''$, where $|\bdzeta^r|=\rho$. Hence, the integrand of (\ref{eq:localintfirst}) is exponentially small 
there, as expected. Moreover, on $\mathcal{B}''\cap \sigma_j'$, i.e.\ where $\zeta_2^r = 0$, we have $\eta''_2=|\alpha| \beta s \rho^2 \neq 0$. This means that $\bdeta''$ is not zero on the singularity, and neither is it tangent to it. Hence there is always a small enough choice of $\beta$ so that the surface does not cross 
the singularity. The presence of $s$ ensures that $\bdeta''$ coincides with the correct bridge configuration.

The surface $\bdG'$ can then simply be defined as
$\psi^{-1} (\bdG'')$, and a schematic representation of the resulting vector field $\bfeta'$ on $\mathcal{C}^\star$ is shown in Figure~\ref{fig:set_saddle} (right) for the bridge configuration given in Figure~\ref{fig:set_saddle} (left).

\begin{figure}[h]
	\centering{\includegraphics[width=0.6\textwidth]{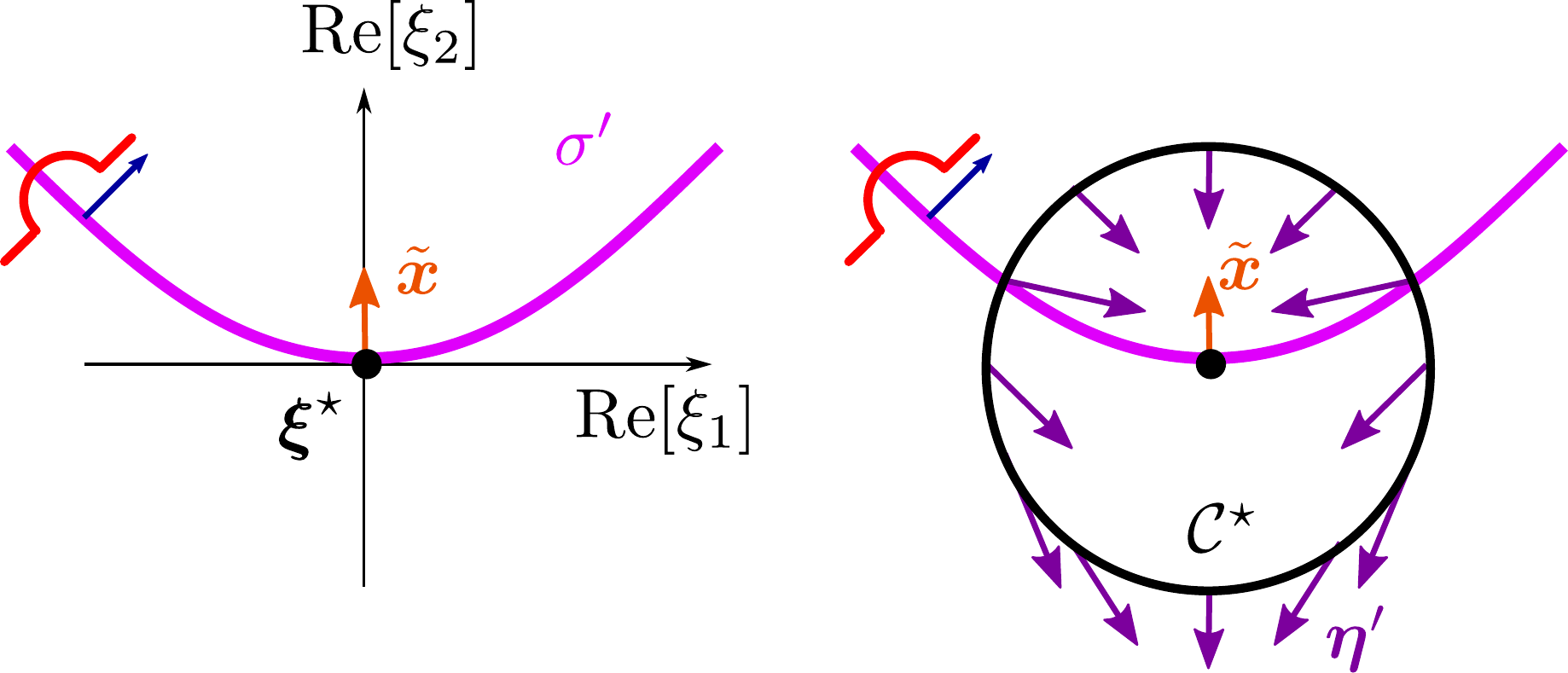}}
	\caption{
		Bridge configuration for an active \RED{SOS} (left); 
		vector $\bfeta'$ on $\mathcal{C}^\star$ near this \RED{SOS} (right)
	}
	\label{fig:set_saddle}
\end{figure}

Let us assume that the leading component of $F$ responsible for the activity of
$\tmmathbf{\xi}^{\star}$ has the following behaviour as $\bdxi\to\bdxi^\star$:
\begin{eqnarray}
	F(\bdxi) \approx A
	(g_j (\tmmathbf{\xi}))^{-\mu}, & \text{so that } & F(\bdxi(\bdzeta)) \approx A \zeta_2^{-\mu},
	\label{eq:approxFsaddle}
\end{eqnarray}
for some constants $A$ and $\mu$. Using (\ref{eq:rewritexdotxiinzeta}), (\ref{eq:approxFsaddle}) and the fact that the Jacobian tends to $((\tva^\star)^2+(\tvb^\star)^2)^{-1}$, the integral (\ref{eq:localintfirst}) is shown to behave as follows 
\begin{align}
	u_{\tmop{loc}} (\tmmathbf{x}) & \approx  \frac{Ae^{- i\tmmathbf{x} \cdot
			\tmmathbf{\xi}^\star}}{(\tva^\star)^2+(\tvb^\star)^2}
	\iint_{\tmmathbf{\Gamma}''_{\tmop{loc}}} (\zeta_2)^{-\mu}   e^{- is r^\star
		(\zeta_2 + \alpha \zeta_1^2)} \mathd
	\tmmathbf{\zeta}.
	\label{eq:saddlesecondlocint}
\end{align}

Now extend $\bdG''_{\text{loc}}$ to a surface $\bdG''$ defined over the whole of $\mathbb{R}^2$, such that the integral over $\bdG''\setminus\bdG''_{\text{loc}}$  is exponentially small. Since the exponential term in (\ref{eq:saddlesecondlocint}) is simplified compared to that of (\ref{eq:localintfirst}), this can simply be done by choosing $\bdeta''$ as in (\ref{eq:good_surf}), but over the whole of $\mathbb{R}^2$. In this case it is clear that the integrand is exponentially small everywhere outside $\mathcal{B}''$. Thus the behaviour of $u_{\text{loc}}$ is unchanged by this extension and we obtain:
\begin{align}
	u_{\tmop{loc}} (\tmmathbf{x}) & \approx  \frac{Ae^{- i\tmmathbf{x} \cdot
			\tmmathbf{\xi}^\star}}{(\tva^\star)^2+(\tvb^\star)^2}
	\iint_{\tmmathbf{\Gamma}''} (\zeta_2)^{-\mu}   e^{- is r^\star
		(\zeta_2 + \alpha \zeta_1^2)} \mathd
	\tmmathbf{\zeta}. \label{eq:saddlethirdextended}
\end{align}

It is now possible to deform $\bdG''$ to a product. Indeed, since $\bdG''$ bypasses the singularity $\sigma_j'$ from above (resp. below) in the $\bde_{\zeta_2}$ direction when $s=+1$ (resp. $s=-1$), we can deform $\bdG''$ to $\mathbb{R}\times(\mathbb{R}+is\epsilon)$ for some small $\epsilon>0$, and the integral can be taken iteratively to obtain
\begin{align}
	u_{\tmop{loc}} (\tmmathbf{x}) & \approx  \frac{Ae^{- i\tmmathbf{x} \cdot
			\tmmathbf{\xi} \star}}{(\tva^\star)^2+(\tvb^\star)^2}  \int_{\mathbb{R}}
	e^{- is r^\star \alpha \zeta_1^2} \mathd \zeta_1
	\times \int_{\mathbb{R}+ is \epsilon} (\zeta_2)^{-\mu} e^{- is r^\star \zeta_2} \mathd \zeta_2. \nonumber \\
		& \approx \frac{Ae^{- i\tmmathbf{x} \cdot
				\tmmathbf{\xi} \star}}{(\tva^\star)^2+(\tvb^\star)^2}\, \mathcal{I}\left(s r^\star \alpha\right) \times \mathcal{J}\left( \mu,r^\star;s \right).
		\label{eq:uIJsaddle}
\end{align}

The two functions $\mathcal{I}$ and $\mathcal{J}$ are defined by
\begin{eqnarray}
	\mathcal{I} (a) = \int_{- \infty}^{\infty} e^{- ia z^2}
	\mathd z & \text{ and } & \mathcal{J} (\mu, a ; s) = \int_{-
		\infty + i s\epsilon }^{\infty + i s\epsilon} z^{-\mu} e^{-i s a
			z} \mathd z,
		\label{eq:defIandJintegrals}
\end{eqnarray}
and can be found\footnote{The $\mathcal{J}$ integral can be found using the results 6\&7 section 3.382 of \cite{GRADSHTEYN1980211}, for $\text{Re}[\mu]>0$. This restriction can actually be lifted \RED{by either analytically continuing the result, or by considering the $\mathcal{J}$ integral as the Fourier transform of a generalised function as in \cite{Jones1982-xs} for instance. Note that the seemingly problematic case of $\mu$ being a non-positive integer (for which $\mathcal{J}$ becomes zero) actually makes sense. Indeed for such values of $\mu$, $F$ is not singular at $\bdxi^\star$ and therefore $\bdxi^\star$ is not a contributing point.}} to simplify to
\begin{align}
	\mathcal{I} (a) = \left\{ \begin{array}{ccc}
		e^{- i \pi / 4} \sqrt{\pi / a} & \tmop{if} & a > 0\\
		e^{i \pi / 4} \sqrt{- \pi / a} & \tmop{if} & a < 0
	\end{array} \right. & \text{ and }  \mathcal{J} (\mu, a ; s) = \left\{ \begin{array}{ccc}
	2 \pi e^{- \frac{i s \mu \pi}{2}} a^{\mu - 1} / \Gamma (\mu) & \tmop{if} & a >
	0\\
	0 & \tmop{if} & a < 0
\end{array} \right.,
\label{eq:exactvaluesIandJ}
\end{align}
where $\Gamma$ is the Euler Gamma function, leading to a closed-form expression for the leading term of $u_\text{loc}(\bdx)$. Note that, since $r^\star>0$ and $\alpha\neq 0$, (\ref{eq:uIJsaddle}) is not zero nor exponentially decaying, so $\bdxi^\star$ is indeed contributing, as expected. For convenience, we give this contribution in terms of the original direction $\bdx$ below:
\begin{align}
	u_{\tmop{loc}} (\tmmathbf{x}) & \approx  \frac{2 \pi A e^{- i\tmmathbf{x} \cdot
			\tmmathbf{\xi}^{\star}} \sqrt{\pi}e^{- i s \mu \pi/2}}{((\tva^\star)^2+(\tvb^\star)^2) \Gamma
		(\mu)} \left( \frac{r}{\sqrt{(\tva^\star)^2 + (\tvb^\star)^2}} \right)^{\mu - 3 / 2} \times
	\left\{ \begin{array}{cc}
		e^{- i \pi / 4} / \sqrt{s \alpha} & \text{if } s \alpha > 0\\
		e^{i \pi / 4} / \sqrt{- s \alpha} & \text{if } s \alpha < 0
	\end{array} \right.
\end{align}

%
%

\begin{rema}
	\label{rem:isolatedsaddle}
	We may be in a situation when the \RED{SOS} is not isolated. Think for
	example of $\sigma_j'$ being a straight line. Then if one point of $\sigma_j'$
	is an active \RED{SOS}, all points of $\sigma'_j$ are active \RED{SOS}. In this case, we cannot find a suitable ball/circle, and a whole strip around $\sigma'_j$
	should be removed, not just a ball (see \cite{Ice2021}). However, for brevity and clarity of
	exposition, we will not consider these special cases in the present article. 
\end{rema}


\subsection{The case of a transverse crossing}\label{sec:awaycrossing}

\label{subsec:transversal}

Let us now consider an active transverse crossing $\tmmathbf{\xi}^{\star}$
of two singularities $\sigma_1$ and~$\sigma_2$ with defining functions $g_1$ and $g_2$. As always, we assume that these singularities have the real property and we denote their real traces by $\sigma_{1,2}'$, we define $\tva^\star_{1,2}$ and $\tvb^\star_{1,2}$ as in (\ref{eq:realpropertytransverse}). The fact that the crossing is transverse implies that the quantity
\[ \Delta^\star=\tva_1^\star \tvb_2^\star-\tva_2^\star \tvb_1^\star\]
is not zero. Let us assume without loss of generalities that $\Delta^\star>0$. Apply  the change of variables
$\psi: \, (\xi_1 , \xi_2) \to (\zeta_1 , \zeta_2)$  
defined by
\[
\zeta_1 = g_1(\xi_1 , \xi_2) ,\qquad \zeta_2 = g_2(\xi_1 , \xi_2), 
\]
The singularities in the new coordinates are therefore the lines 
$\zeta_1 = 0$ and $\zeta_2$ = 0. Moreover, we have
\[
\left( \begin{array}{c}
\xi_1 - \xi_1^\star \\
\xi_2 - \xi_2^\star
\end{array} \right) = \Psi \, 
\left( \begin{array}{c}
\zeta_1  \\
\zeta_2 
\end{array} \right) + \mathcal{O}(\zeta_1^2 + \zeta_1 \zeta_2 + \zeta_2^2), \quad \text{where} \quad \Psi=\frac{1}{\Delta^{\star}} \left( \begin{array}{cc}
\tvb_2^\star & -\tvb_1^\star\\
-\tva_2^\star & \tva_1^\star
\end{array} \right).
\]
As before, we note that this change of variable naturally defines a real change of variable, and that the associated Jacobian is $1/\Delta^\star$ at $\bdxi^\star$. We can also show that, at $\bdxi^\star$, the unit vectors $\bde_{\zeta_1}$ and $\bde_{\zeta_2}$ are given by $\bde_{\zeta_1}=-\boldsymbol{t}_2^\star$ and $\bde_{\zeta_2}=+\boldsymbol{t}_1^\star$, where the tangent vectors $\boldsymbol{t}^\star_{1,2}$ are defined as in (\ref{eq:ntvect}) by $\boldsymbol{t}^\star_{1,2}=(-\tvb^\star_{1,2},\tva^\star_{1,2})^\transpose/\sqrt{(\tva^\star_{1,2})^2+(\tvb^\star_{1,2})^2}$. As a consequence, we can show that
\begin{align}
\bde_{\zeta_1}\cdot \bdn_1^\star=\Delta^\star>0 \quad \text{ and } \quad \bde_{\zeta_2} \cdot \bdn_2^\star=\Delta^\star>0,
\label{eq:dotezetan12}
\end{align}
where the normal unit vectors $\bdn^\star_{1,2}$ are defined as in (\ref{eq:ntvect}) by $\boldsymbol{n}^\star_{1,2}=(\tva^\star_{1,2},\tvb^\star_{1,2})^\transpose/\sqrt{(\tva^\star_{1,2})^2+(\tvb^\star_{1,2})^2}$.

In order to consider all possible bridge configurations at once, let us introduce two parameters $s_1$ and $s_2$ defined such that:
\begin{itemize}
	\item If $s_1=+1$ (resp. $s_1=-1$) then $\bdG$ bypasses $\sigma_1'$ from above (resp. below) at $\bdxi^\star$ in the direction $\bdn_1^\star$ and also, because of (\ref{eq:dotezetan12}), in the direction $\bde_{\zeta_1}$.
	\item If $s_2=+1$ (resp. $s_2=-1$) then $\bdG$ bypasses $\sigma_2'$ from above (resp. below) at $\bdxi^\star$ in the direction $\bdn_2^\star$ and also, because of (\ref{eq:dotezetan12}), in the direction $\bde_{\zeta_2}$.
\end{itemize}

Let us consider a given bridge configuration given by a pair $(s_1,s_2)$, and an observation direction $\tilde{\bdx}$ that points to the associated active quadrant. Let us also introduce the vector $\bdx'=(x_1',x_2')^\transpose$ defined by $(\bdx')^\transpose=(\tilde{\bdx})^\transpose \Psi$ so that $\tilde{\bdx}\cdot\Psi\bdzeta=\bdx'\cdot\bdzeta$.
The integrand's exponential factor can hence be written in the new coordinates as follows 
\begin{equation}
	e^{ - i r \tmmathbf{\tilde x} \cdot \bfxi(\tmmathbf{\zeta}) }
	= e^{-i\bdx\cdot\bdxi^\star}
	e^{ - i r (x_1' \zeta_1 + x_2' \zeta_2 + \mathcal{O}(\zeta_1^2 + \zeta_1 \zeta_2 + \zeta_2^2)) },
	\label{eq:exp2}
\end{equation}
where
\begin{equation}
	x_1'=\tfrac{1}{\Delta^\star} (\tilde{x}_1 \tvb_2^\star-\tilde{x}_2 \tva_2^\star) \quad \text{ and } \quad x_2'=\tfrac{1}{\Delta^\star} (-\tilde{x}_1 \tvb_1^\star+\tilde{x}_2 \tva_1^\star).
\end{equation}
Since $\tilde{\bdx}$ points to the active quadrant in the $\bdxi$ real plane, $\bdx'$ points to the active quadrant in the $\bdzeta$ real plane. In other words, we have $\text{sign}(x_i')=s_i$ or, equivalently, $|x_i'|=s_i x_i'$.
An illustration in the real $(\zeta_1,\zeta_2)$ plane of the bridge configuration corresponding to $s_1=+1$ and $s_2=+1$,  together  with an active direction $\bdx'$ is shown in Figure~\ref{fig:set_trans} (left).
\begin{figure}[h]
  \centering{\includegraphics[width=0.6\textwidth]{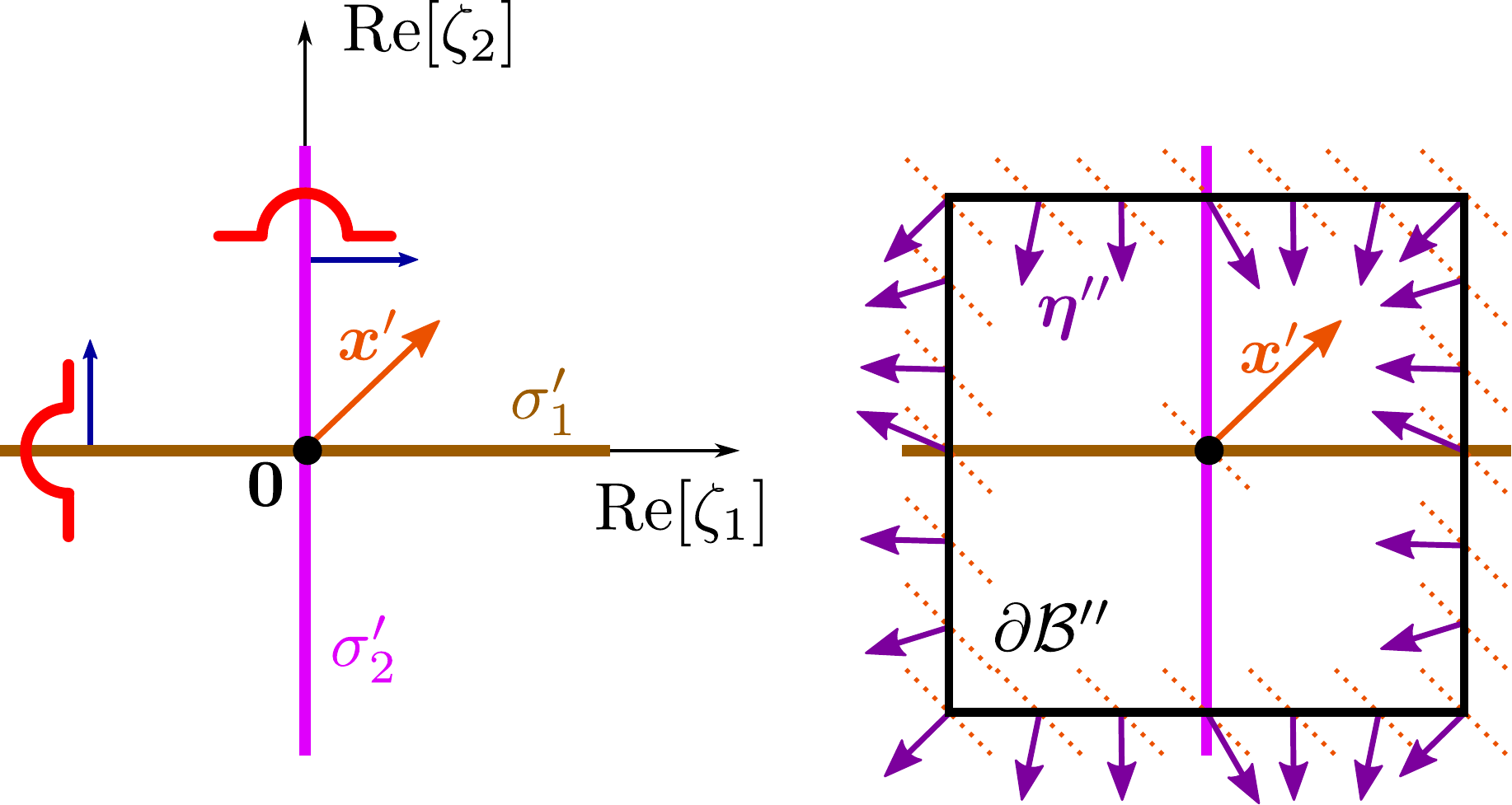}}
  \caption{
  a transverse crossing (left); 
  vector $\bfeta''$ on the boundary of $\mathcal{B}''$ (right)
  }
\label{fig:set_trans}
\end{figure}

We want to construct a suitable surface over a real neighbourhood $\mathcal{B}^\star$ of $\bdxi^\star$. In order to do this explicitly, we will work in the $\bdzeta$ real plane and consider a neighbourhood of $\boldsymbol{0}$, $\mathcal{B}''=\psi(\mathcal{B}^\star)$, that we will define explicitly and on which we will build an explicit vector field $\bdeta''(\bdzeta^r)$ defining a suitable surface $\bdG''_{\text{loc}}$ guaranteeing an exponential decay over $\partial \mathcal{B}''$ (here, again, $\tmmathbf{\zeta}^r$ stands for the real part of $\tmmathbf{\zeta}$).

Define $\mathcal{B}''$ as the square $|\zeta^r_1| < \rho$, $|\zeta^r_2| < \rho$, for some small enough $\rho$ to ignore the $\mathcal{O}$-terms in (\ref{eq:exp2}), and for some small $\beta>0$, define the vector field $\bdeta''$ over $\mathcal{B}''$ by
\begin{equation}
	\eta''_1 =  \beta\left(s_1\rho-2\frac{|x'_1|+|x'_2|}{x'_1}|\zeta^r_1|\right), \quad \eta''_2 = 
		 \beta\left(s_2\rho-2\frac{|x'_1|+|x'_2|}{x'_2}|\zeta^r_2|\right).
	\label{eq:good_surf_1}
\end{equation}
It can be shown that for $\bdzeta\in\bdG''_\text{loc}$, we have
\begin{align}
	\text{Im}[\bdx'\cdot\bdzeta]=\bdx'\cdot\bdeta''=\beta(|x_1'|+|x_2'|)(\rho-2(|\zeta_1^r|+|\zeta_2^r|)),
\end{align}
so that, over $\partial\mathcal{B}''$, $\bdx'\cdot\bdeta''=-\beta \rho (|x'_1|+|x'_2|)<0$, thus the integrand is exponentially decaying there as expected.
Remembering that the Jacobian is approximately $1/\Delta^\star$, using (\ref{eq:exp2}) and introducing $\bdG'_\text{loc}=\psi^{-1}(\bdG''_\text{loc})$, we have that
\begin{align}
	u_{\text{loc}}(\bdx) &= \iint_{\bdG'_\text{loc}} F(\bdxi) e^{-i \bdx \cdot \bdxi} \, \mathd \bdxi
	\approx \frac{e^{-i \bdx \cdot \bdxi^\star}}{\Delta^\star} \iint_{\bdG''_\text{loc}} F(\bdxi(\bdzeta)) e^{-i r \bdx' \cdot \bdzeta}\, \mathd \bdzeta.
	\label{eq:ulocfirststeptransverse}
\end{align}
 Let us now assume that the leading component of $F$ responsible for the activity of
 $\tmmathbf{\xi}^{\star}$ has the following behaviour as $\bdxi\to\bdxi^\star$:
 \begin{eqnarray}
 	F(\bdxi) \approx A
 	(g_1 (\tmmathbf{\xi}))^{-\mu_1}(g_2 (\tmmathbf{\xi}))^{-\mu_2}, & \text{so that } & F(\bdxi(\bdzeta)) \approx A \zeta_1^{-\mu_1}\zeta_2^{-\mu_2},
 	\label{eq:approxFtransverse}
 \end{eqnarray}
 for some constants $A$ and $\mu_{1,2}$. In this case, using (\ref{eq:ulocfirststeptransverse}), we get
 \begin{align}
 	u_{\text{loc}}(\bdx) &\approx A \frac{e^{-i \bdx \cdot \bdxi^\star}}{\Delta^\star} \iint_{\bdG''_\text{loc}} \zeta_1^{-\mu_1} \zeta_2^{-\mu_2}  e^{-i r(x'_1 \zeta_1+x_2' \zeta_2)}\, \mathd \bdzeta.
 	\label{eq:ulocsecondsteptransverse}
 \end{align}
As in the previous section, using the vector field $\bdeta''$ defined by $(\ref{eq:good_surf_1})$ we can now extend the surface $\bdG''_\text{loc}$ to a surface
defined over the whole of $\mathbb{R}^2$, such that the integrand of (\ref{eq:ulocsecondsteptransverse}) is exponentially small on $\bdG'' \setminus \bdG''_\text{loc}$. Therefore this extension does not affect the behaviour of $u_\text{loc}$, and we have
 \begin{align}
u_{\text{loc}}(\bdx) &\approx A \frac{e^{-i \bdx \cdot \bdxi^\star}}{\Delta^\star} \iint_{\bdG''} \zeta_1^{-\mu_1} \zeta_2^{-\mu_2}  e^{-i r(x'_1 \zeta_1+x_2' \zeta_2)}\, \mathd \bdzeta.
\label{eq:ulocthirdsteptransverse}
 \end{align}
It is now possible to deform $\bdG''$ to the product $(\mathbb{R}+is_1\epsilon)\times(\mathbb{R}+i s_2 \epsilon)$ without crossing any singularities and conserving the correct bridge configuration to get
\begin{align}
\label{eq:cross_est}
	u_{\text{loc}}(\bdx) &\approx A \frac{e^{-i \bdx \cdot \bdxi^\star}}{\Delta^\star} \, \mathcal{J}(\mu_1,r|x_1'|;s_1)\times\mathcal{J}(\mu_2,r|x_2'|;s_2),
\end{align}
where $\mathcal{J}$ is defined in (\ref{eq:defIandJintegrals}). Therefore, using (\ref{eq:exactvaluesIandJ}), we obtain an explicit, closed-form approximation for $u_\text{loc}(\bdx)$ when $\bdx$ corresponds to an active transverse crossing. Given the fact that $r|x_i'|>0$, this is always non-zero and non exponentially decaying. We can hence conclude that this active transverse crossing is indeed contributing. For convenience, we give this contribution in terms of the original direction $\bdx$ below:
\begin{equation}
	u_\text{loc}(\bdx)\approx\frac{4\pi^2Ae^{-i\bdx\cdot\bdxi^\star}e^{-i\frac{\pi}{2}(s_1\mu_1+s_2\mu_2)}}{\Gamma(\mu_1)\Gamma(\mu_2) (\Delta^\star)^{\mu_1+\mu_2-1}} |x_1 \tvb_2^\star-x_2 \tva_2^\star|^{\mu_1-1}  |-x_1 \tvb_1^\star+x_2 \tva_1^\star|^{\mu_2-1}. 
\end{equation}
Remembering that we have made the assumption $\Delta^\star>0$, this contribution only occurs for directions $\bdx$ rendering the point active, that is, whenever we have
\begin{align}
	\text{sign}(\tilde{x}_1 \tvb_2^\star-\tilde{x}_2 \tva_2^\star)=s_1 \quad \text{ and } \quad \text{sign}(-\tilde{x}_1 \tvb_1^\star+\tilde{x}_2 \tva_1^\star)=s_2.
\end{align}
\RED{
\subsection{Connections with the two-dimensional saddle point method}
\label{subsec:2d saddle}
As explained in introduction, substantial research has been carried out regarding the estimation of  multiple integrals. Most of these contributions are focussed on the multidimensional saddle point method (see e.g.\ \cite{Borovikov1994, Wong2001-cd, Jones1982-xs, Bleistein1987-xa, Felsen1994-hb,Lighthill78,Jones1958-nf}). Let us summarise it briefly. Let the functions $F(\bdxi)$ and $G(\bdxi)$ in the integral
\[
I(\lambda)=\int_D F({\tmmathbf{\xi}}) e^{\lambda G({\tmmathbf{\xi}})}d\tmmathbf{\xi}
\]
be infinitely differentiable within and on the integration domain $D$. Assume further that $G(\bdxi)$  has a saddle point ${\tmmathbf{\xi}}_0$ (i.e.\ we have $\nabla G(\bdxi_0)=\boldsymbol{0}$) with a non-degenerate Hessian matrix $\rm H$ at $\bdxi_0$.
Then, the contribution from the saddle point as $\lambda\to\infty$ can be estimated as follows:
\begin{equation}
\label{eq:saddle_estimation}
I(\lambda) \approx a_0 \lambda^{-1},\quad \text{where} \quad a_0 = \frac{2\pi F({\tmmathbf{\xi}_0})e^{i\pi\delta/2}}{\sqrt{|\det({\rm H})|}},
\end{equation} 
and $\delta=1$ if both eigenvalues of the matrix ${\rm H}$ are positive, $\delta=0$ if they have opposite signs, and $\delta=-1$ if they are both negative.

A connection between SOS, transverse crossings and two-dimensional saddle points can be established. First, let us notice that SOS can be viewed as a particular case of a transverse crossing. 
Indeed, consider an integral similar to (\ref{eq:ulocthirdsteptransverse}) that typically occurs for transverse crossings
\begin{equation}
\label{eq:cross_int}
{\mathbb{I}}(\mu_1,\mu_2)=\iint_{\bdG} \zeta_1^{-\mu_1} \zeta_2^{-\mu_2}  e^{-i r(x'_1 \zeta_1+x_2' \zeta_2)}\, \mathd \bdzeta,
\end{equation}
where $\bdG$ is a surface of integration close to the real plane. For the sake of argument, assume that $\mu_1=1/2$ and use the change of variables $\tilde{\bdzeta}\leftrightarrow \bdzeta$ defined by
\[
\tilde\zeta_1 =  \zeta_1^{1/2},\quad \tilde\zeta_2 = \zeta_2, 
\]
to obtain
\begin{equation}
\label{eq:saddle_int}
{\mathbb{I}}(1/2,\mu_2)=2\iint_{\tilde \bdG} \tilde{\zeta}_2^{-\mu_2}  e^{-i r(x'_1 {\tilde\zeta}^2_1+x_2' {\tilde\zeta}_2)}\, \mathd \tilde{\bdzeta}.
\end{equation}
The integral (\ref{eq:saddle_int}) is now typically similar to the SOS integral  (\ref{eq:saddlethirdextended}). Second, let us now show that a multidimensional saddle integral can be recovered from a SOS. Assume that $\mu_2=1/2$ in (\ref{eq:saddle_int}) and use the change of variables $\hat{\bdzeta} \leftrightarrow \tilde{\bdzeta}$ defined by
\[
\hat\zeta_1 =  \tilde{\zeta}_1,\quad \hat\zeta_2 = \tilde{\zeta}_2^{1/2}, 
\]
to obtain
\begin{equation}
	\label{eq:raphraph2dsaddleraph}
{\mathbb{I}}(1/2,1/2)=4\iint_{\hat\bdG} e^{-i r(x'_1 {\hat\zeta}^2_1+x_2' {\hat\zeta}^2_2)}\, \mathd \hat\bdzeta.
\end{equation} 
The integral (\ref{eq:raphraph2dsaddleraph}) is now a two-dimensional saddle point integral with the saddle point at the origin. It can be estimated using (\ref{eq:saddle_estimation}). Of course, the result is consistent with (\ref{eq:cross_est}) and (\ref{eq:uIJsaddle}).  

To be more general, assuming that $\mu_{1,2}\neq1$ let us introduce the following change of variables  in ${\mathbb{I}}(\mu_1,\mu_2)$:
 \[
\tilde \zeta_1 = \zeta_1^{1-\mu_1}, \quad \tilde \zeta_2 =  \zeta_2^{1-\mu_2}.
\]
The integral (\ref{eq:cross_int})  reduces to 
\begin{equation}
{\mathbb{I}}(\mu_1,\mu_2)=\frac{1}{(1-\mu_1)(1-\mu_2)}\iint_{\tilde\bdG} e^{-i r(x'_1 \tilde\zeta_1^{1/(1-\mu_1)}+x_2' \tilde\zeta_2^{1/(1-\mu_2)})}\, \mathd \tilde\bdzeta.
\label{eq:saddle_crossing}
\end{equation}
Then, rewriting 
\begin{equation}
\label{eq:reg_cond_crossing}
1-\mu_1 = 1/n,\quad 1-\mu_2 = 1/m,
\end{equation}
for some $m$ and $n$, we obtain a saddle point integral. The saddle point at the origin is non-degenerate (i.e.\ the associated Hessian matrix is non-degenerate) if and only if $n = m = 2$. Otherwise, we get a saddle point of a higher order. Following the reasoning above, one can consider transverse crossings as ``topological phenomena'' that contain both SOS and two-dimensional saddles as particular cases.  
}

\section{The surface $\bold{\Gamma}$ away from the active points and the locality principle}
\label{sec:integrationawayfromactive}

In this section, we are going to present a sketch of a proof of the main statement 
of the paper. This is the {\em locality principle:}

\begin{theorem}
\label{th:locality}
The integral $u(r \tmmathbf{\tilde x})$ can be estimated 
as $r \to \infty$ as an asymptotic series.
For almost all $\tmmathbf{\tilde x}$,
the terms
of the asymptotic expansion (up to the terms decaying exponentially
as $r \to \infty$) can be obtained by estimating the 
integral in neighbourhoods of contributing points.   
\end{theorem}


As seen in the previous two sections, for a given direction $\tilde{\tmmathbf{x}}$,
we can decide which points of $\mathbb{R}^2$ are active or not, and which of the active points are actually contributing points. In this
section we wish to show that we can construct a surface of integration
$\bdG'$ away from the vicinity of those contributing points such that on
this part of $\bdG'$, the integrand of
(\ref{eq:FourierIntegralIndented}) is exponentially decaying as $r = |
\tmmathbf{x} | \rightarrow \infty$ and does not contributes to the asymptotic
behaviour of $u (\tmmathbf{x})$.

Let us hence consider the contributing points. As discussed in Remarks
\ref{rem:notriple} and \ref{rem:isolatedsaddle}, we will only consider two sorts: transverse non-additive crossings of
two real traces or isolated \RED{SOS}. These contributing points are assumed
to be separated from each other, that is, there exists a small ball (in
$\mathbb{R}^2$) around each of these points such that:
\begin{itemize}
  \item there is only one contributing point inside the ball
  
  \item the only real traces present within the ball are those responsible for
  the active nature of the contributing point
\end{itemize}


Below we provide a sketch of a proof. One can find a rigorous formulation 
of the locality principle and a detailed proof, but made for a rather special case, 
in \cite{Ice2021}. 

\textbf{Step 1. }Assume that a \RED{function} $F$ with singularity set $\sigma$  has $N$ such contributing points denoted by $\bdxi_j^\star$, $j=1,...,N$, and exclude some small neighbourhoods $\mathcal{B}_j^\star$ of each contributing point from $\mathbb{R}^2$. These neighbourhoods have boundary $\mathcal{C}_j^\star=\partial \mathcal{B}_j^\star$. 
An example of this process is shown in Figure~\ref{fig:planewithholes_1}. In this example,  the singularities $\sigma_1'$ and $\sigma_2'$ in the figure are the line and the 
circle, and there are two contributing points (a \RED{SOS} $\bdxi_1^\star$ and a transverse crossing $\bdxi_2^\star$). 

\textbf{Step 2. } Introduce the set of curves 
\[
\mathcal{G} = 
(\sigma' \setminus (\cup_j \mathcal{B}_j^\star)) \cup (\cup_j \mathcal{C}^\star_j).
\]
Define the vector $\bfeta'$ on $\mathcal{G}$, such that:

\begin{itemize}

\item
$\bfeta'$ is piecewise smooth on $\mathcal{G}$.

\item
$\bfeta' \cdot \tmmathbf{\tilde x} < 0$. 
\end{itemize}  
This is possible due to two circumstances: 1) all points of 
$\sigma' \setminus (\cup_j \mathcal{B}_j^\star)$ are inactive, thus in each 
small neighbourhood of each point of this set such a choice is possible; 
2) such vector field on each boundary $\mathcal{C}^\star_j$ is constructed 
explicitly in Sections~\ref{subsec:saddle} and~\ref{subsec:transversal}.
Some subtle reasoning is needed to prove that such a choice can be made over the whole 
$\mathcal{G}$, and the result can be made smooth. For brevity, we do not go into such details here. 

{\textbf{Step 3. } The set of lines $\mathcal{G}$ divides $\mathbb{R}^2 \setminus (\cup_j \mathcal{B}^\star_j)$
into several regions, see Figure \ref{fig:planewithholes_1} (right). A vector field $\bfeta'$ obeying the declared property 
is defined on the boundary of each region. 
If we manage to continue $\bfeta'$ piecewise-smoothly into each region from its boundary, then the field $\bfeta'$
will be defined on the whole of $\mathbb{R}^2 \setminus (\cup_j \mathcal{B}^\star_j)$. This field would then define 
a surface $\bdG'$ on which the integrand is exponentially small. Adding the ``patches'' built for the domains $\mathcal{B}^\star_j$ in Sections~\ref{subsec:saddle} and~\ref{subsec:transversal}, on which the integrand is not exponentially small, finishes the proof of the locality principle. 

\begin{figure}[h]
	\centering{\includegraphics[width=0.8\textwidth]{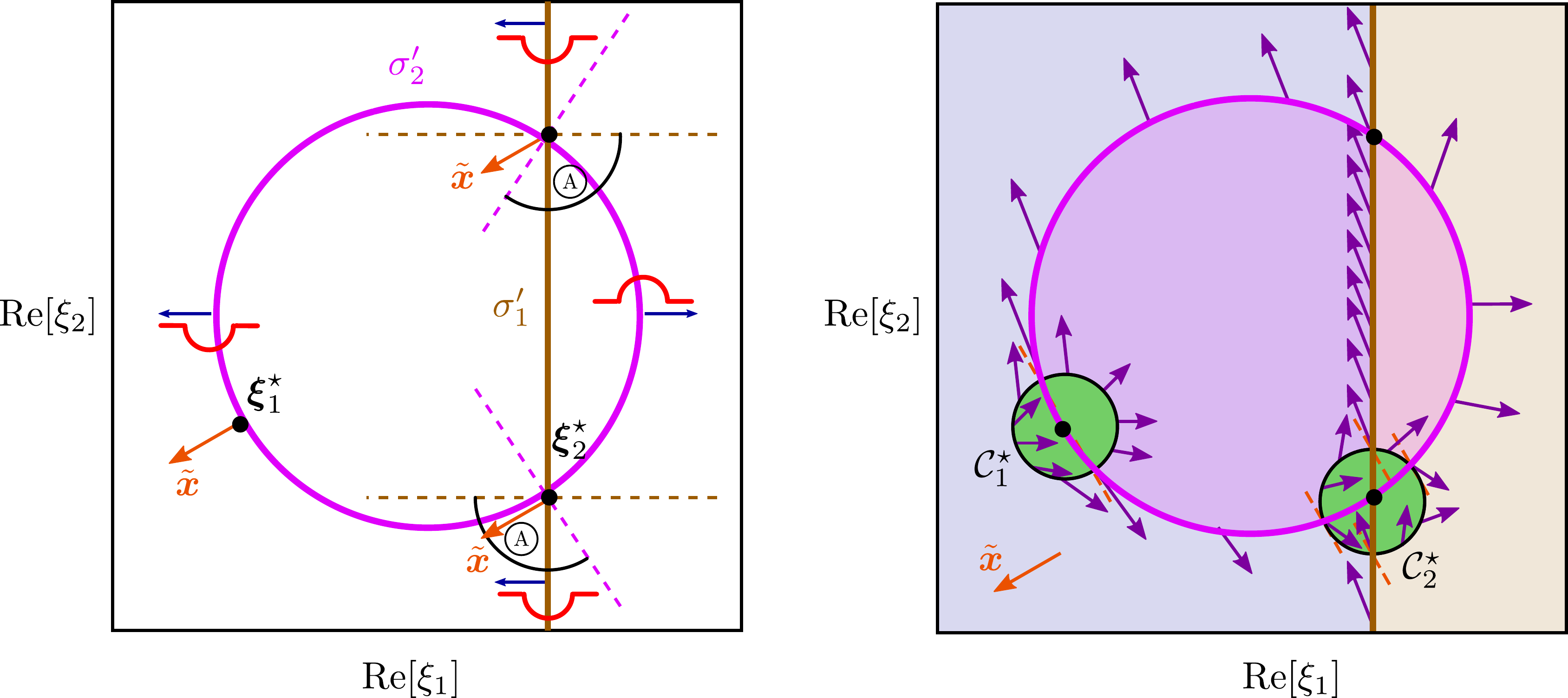}}
	\caption{ Illustration of the proof of Theorem \ref{th:locality} with the bridge configuration and the region of activity of the potentially contributing points (left) and the region decomposition defined by $\mathcal{G}$ together with the vector field $\bdeta'$ on $\mathcal{G}$ (right).}
	\label{fig:planewithholes_1}
\end{figure}

The continuation of $\bfeta'$ onto a region of $\mathbb{R}^2 \setminus (\cup_j \mathcal{B}^\star_j)$ from its boundary is reasonably straightforward. For example, one could introduce some coordinates $(\alpha_1 , \alpha_2)$ on a region,
such that the coordinate transformation is smooth, and the intersections of the coordinate lines 
$\alpha_1 = {\rm const}$
with the region of interest are points or segments, but not semi-infinite intervals. Then, interpolate $\bfeta'$ linearly on each such segment. Since on both ends of the segment the inequality $\bfeta' \cdot \tmmathbf{\tilde x} < 0$ 
is fulfilled by construction, it is fulfilled on the whole segment. The 
resulting vector field $\bfeta'$ is piecewise smooth by construction. 
The procedure is schematically outlined in Figure~\ref{fig:filling}.

\begin{figure}[h]
  \centering{\includegraphics[width=0.3\textwidth]{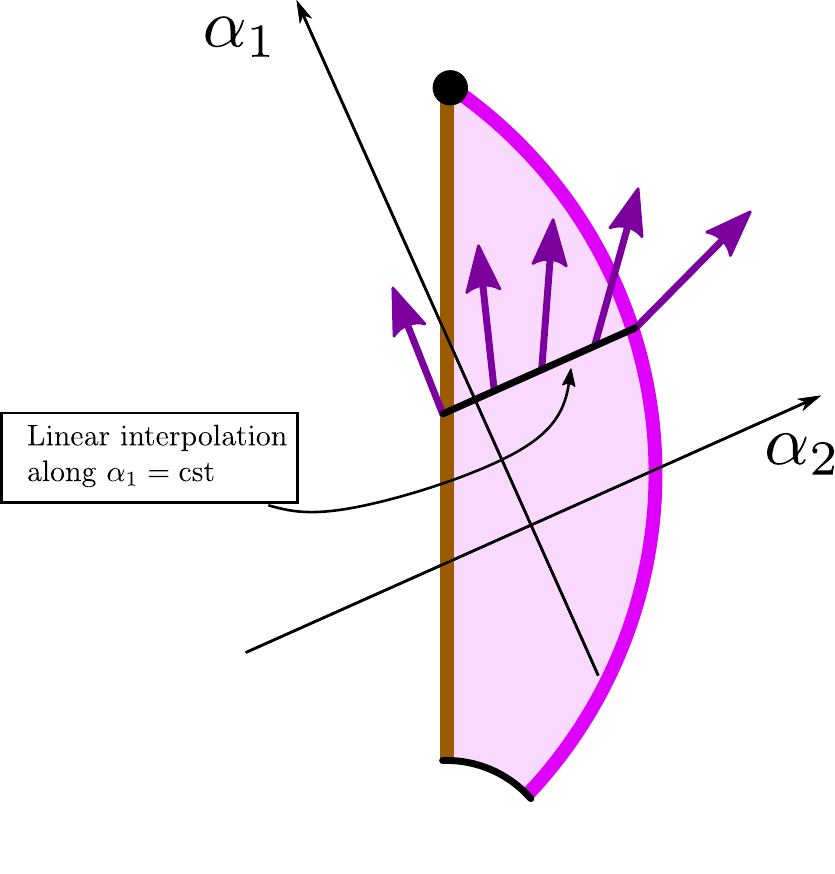}}
  \caption{Illustration of the procedure of continuing $\bfeta'$ from the boundary of a region onto the whole region.}
\label{fig:filling}
\end{figure}

This finishes the sketch of the proof of the locality principle. 


\section{Two illustrative examples}\label{sec:simplenontrivialexamples}

The aim of this section is to demonstrate the validity of the asymptotic procedure developed thus far on two simple, yet non-trivial, examples. 
In order to do so, we will compare direct numerical computations of 2D Fourier integrals with the asymptotic formulae derived in the article, and show that the agreement between the two is excellent in the far-field.

It should be mentioned however that numerically evaluating a 2D Fourier integral is not a simple 
task. On the one hand, it is necessary to truncate the domain of integration, and on the other hand, in general, the integrand is very oscillatory and not necessary small outside the truncated domain. In order to address these issues, we will use the theory developed in Section \ref{sec:nearactive} and \ref{sec:integrationawayfromactive} and deform the surface of integration to a flattable surface $\bdG'$ described by a real vector field $\bdeta'$, built so that the integrand is exponentially small everywhere except in the vicinity of the contributing points. The only constraint on the truncated domain is hence to be large enough in order for it to lie over all the contributing points. We note also that this is an opportunity for the procedure of building the surface $\bdG'$ to be validated on a practical level.


The contributing points of the first example are all transverse crossings, while the second example leads to a combination of a \RED{SOS} and some transverse crossings.

\begin{exa}
	Find the asymptotic behaviour of $u (\tmmathbf{x}) = \underset{\varkappa
		\searrow 0}{\lim}  u (\tmmathbf{x}; \varkappa)$ as $r = | \tmmathbf{x} |
	\rightarrow \infty$ where
	\begin{align}
		u (\tmmathbf{x}; \varkappa) & =  \iint_{\mathbb{R}^2} \frac{e^{-
				i\tmmathbf{x} \cdot \tmmathbf{\xi}}}{(\xi_1 + i \varkappa)^{1 / 2} (\xi_2 +
			i \varkappa)^{1 / 2} (\xi_1 + \xi_2 - 1 + i \varkappa)^{\RED{1 / 2}}} \, \mathd
		\tmmathbf{\xi},
		\label{eq:triang_int}
	\end{align}
and validate it numerically for a specific observation direction $\tilde{\bdx}$.
\label{ex:7.2}
\end{exa}

The integrand has the real property and three irreducible singularity
components $\sigma_{1, 2, 3}$ with respective defining function $g_{1, 2, 3}$ defined by
	$g_1 (\tmmathbf{\xi}) = \xi_1$, $g_2 (\tmmathbf{\xi}) = \xi_2$ and $g_3
	(\tmmathbf{\xi}) = \xi_1 + \xi_2 - 1$.
Using Section \ref{sec:findinggoodbridges}, we obtain the bridge configuration over the real traces $\sigma_{1, 2, 3}'$
displayed in Figure \ref{fig:triang_sing} (left). These singularities lead to
three transverse crossings $\tmmathbf{\xi}_{1, 2, 3}^{\star}$ defined by
$\tmmathbf{\xi}_1^{\star} = (0, 0)$, $\tmmathbf{\xi}_2^{\star} = (0, 1)$ and $\tmmathbf{\xi}_3^{\star} = (1, 0)$,
each with their own region of activity, as depicted in Figure
\ref{fig:triang_sing} (left). For each singularity, we can define the usual
quantities $\text{\tmverbatim{a}}_{1, 2, 3}^{\star}$ and
$\text{\tmverbatim{b}}_{1, 2, 3}^{\star}$, which, given the fact that the real
traces are straight lines, remain the same on the whole of $\sigma_{1, 2, 3}'$.
They are given by $\text{\tmverbatim{a}}_{1, 3}^{\star} =
\text{\tmverbatim{b}}_{2, 3}^{\star} = 1$ and $\text{\tmverbatim{a}}_2^{\star}
= \text{\tmverbatim{b}}_1^{\star} = 0$. We can also define the normals
$\tmmathbf{n}^{\star}_{1, 2, 3}$ by $\tmmathbf{n}_1^{\star}
=\tmmathbf{e}_{\xi_1}$, $\tmmathbf{n}_2^{\star} =\tmmathbf{e}_{\xi_2}$,
$\tmmathbf{n}_3^{\star} = \tfrac{1}{\sqrt{2}} (\tmmathbf{e}_{\xi_1}
+\tmmathbf{e}_{\xi_2})$. Because of this and the bridge configuration, it is
clear in this case that for each crossing the sign factors are always $s_1 =
s_2 = + 1$. To ensure that the determinant $\Delta^{\star}_{1, 2, 3}$ at each
crossing is positive, we consider the real traces in the following order when applying the results of Section
\ref{subsec:transversal}: $\tmmathbf{\xi}_1^{\star} \leftrightarrow
	(\sigma_1', \sigma_2')$, $\tmmathbf{\xi}_2^{\star} \leftrightarrow  (\sigma_1', \sigma_3')$ and $\tmmathbf{\xi}_3^{\star} \leftrightarrow (\sigma_3', \sigma_2')$. As a result we have $\Delta_{1, 2, 3}^{\star} = 1$. Moreover, looking at the
	local behaviour of the integrand near each crossing,  the constant in (\ref{eq:approxFtransverse}) is given for each crossing by $A_1 = \RED{-i}$ and
	$A_{2, 3} = 1$. The resulting asymptotic formula is 
\begin{align}
	u (\tmmathbf{x}) & \approx \RED{-}4 \pi \frac{\mathcal{H} (x_1)}{\sqrt{x_1}} 
	\frac{\mathcal{H} (x_2)}{\sqrt{x_2}} - 4 i \pi e^{- i x_2} 
	\frac{\mathcal{H} (x_1 - x_2)}{\sqrt{x_1 - x_2}}  \frac{\mathcal{H}
		(x_2)}{\sqrt{x_2}} - 4 i \pi e^{- ix_1}  \frac{\mathcal{H}
		(x_1)}{\sqrt{x_1}}  \frac{\mathcal{H} (- x_1 + x_2)}{\sqrt{- x_1 + x_2}},
	\label{eq:triang_asympt}
\end{align}
for some $\tmmathbf{x}= r \tilde{\tmmathbf{x}}$, where $\tilde{\tmmathbf{x}}$
is not perpendicular to any of the real traces and $\mathcal{H}$ is the usual
Heaviside function.

Let us now pick the observation direction $\tilde{\bdx}=\tfrac{1}{\sqrt{3.06}}(1.5,0.9)$ as represented in Figure \ref{fig:triang_sing} (left). We note that for this direction, the only contributing crossings are $\bdxi_1^\star$ and $\bdxi_2^\star$. We will now evaluate the integral (\ref{eq:triang_int}) numerically for different values of $r$. To do this we explicitly build a surface of integration $\bdG'$ and its defining vector field $\bdeta'$. In the vicinity of the contributing points, this is done using Section \ref{subsec:transversal}, and, in particular, the formula (\ref{eq:good_surf_1}) and the mapping $\psi$. Outside of these vicinities, we apply the procedure described in Section \ref{sec:integrationawayfromactive}: we decompose the real plane into several subdomains, depicted in Figure \ref{fig:triang_sing} (right), define $\bdeta'$ on the boundaries of these regions carefully and define it inside each region by linear interpolation. 

\begin{figure}[h]
	\centering{\includegraphics[width=0.6\textwidth]{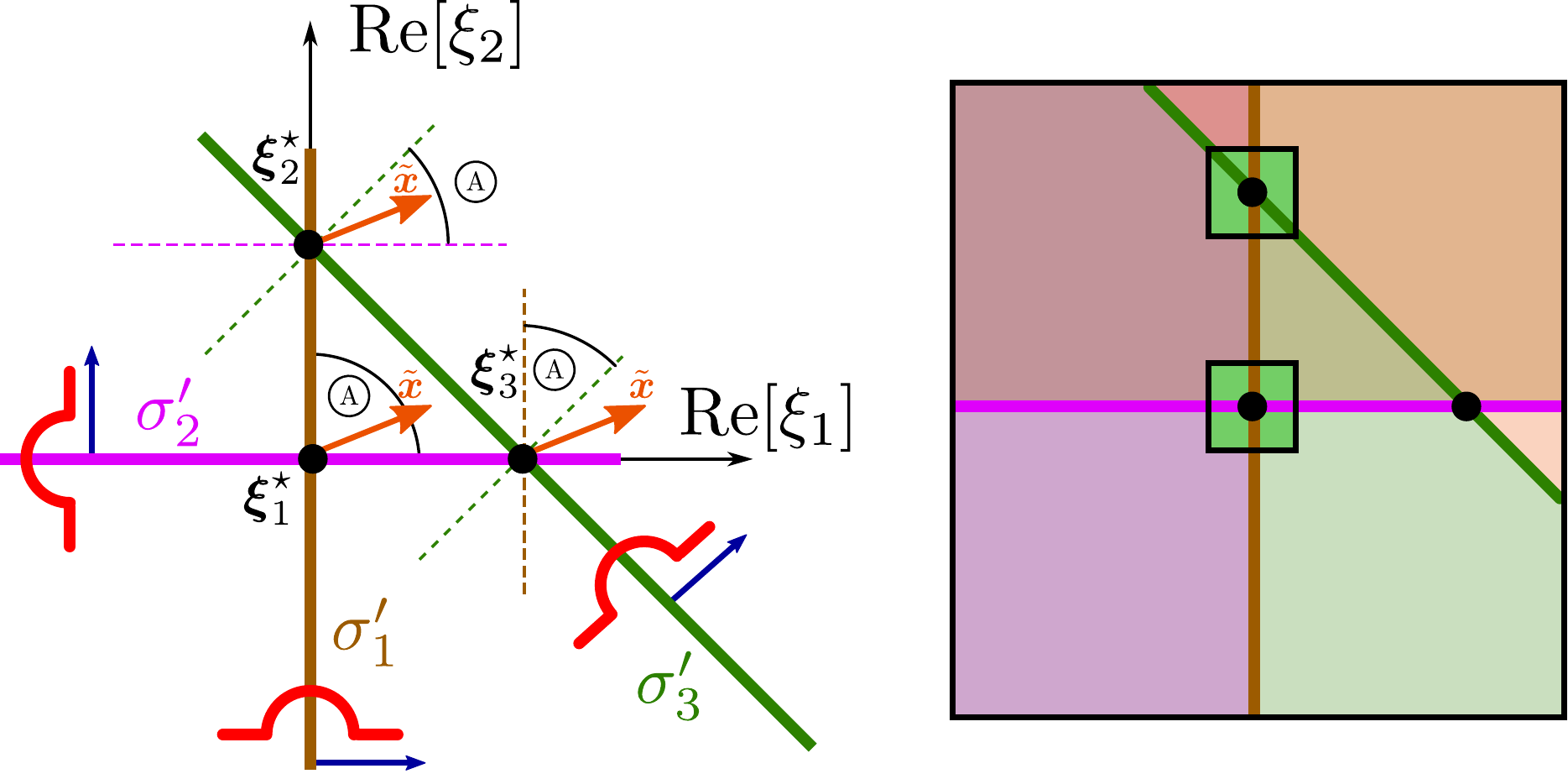}}
	\caption{The real traces associated to the integral (\ref{eq:triang_int}) together with its bridge configuration and the regions of activity of the crossings (left); the domain decomposition used to build $\bdeta'$ (right).}
	\label{fig:triang_sing}
\end{figure}
The resulting vector field $\bdeta'$ is displayed in Figure \ref{fig:eta_heatmaps} (left). We also provide a heat map of $\tilde{\bdx}\cdot \bdeta'$ in Figure \ref{fig:eta_heatmaps} (right) to show that it is indeed negative everywhere outside the vicinity of the two contributing points. 
\begin{figure}[h]
	\centering{\includegraphics[width=0.45\textwidth]{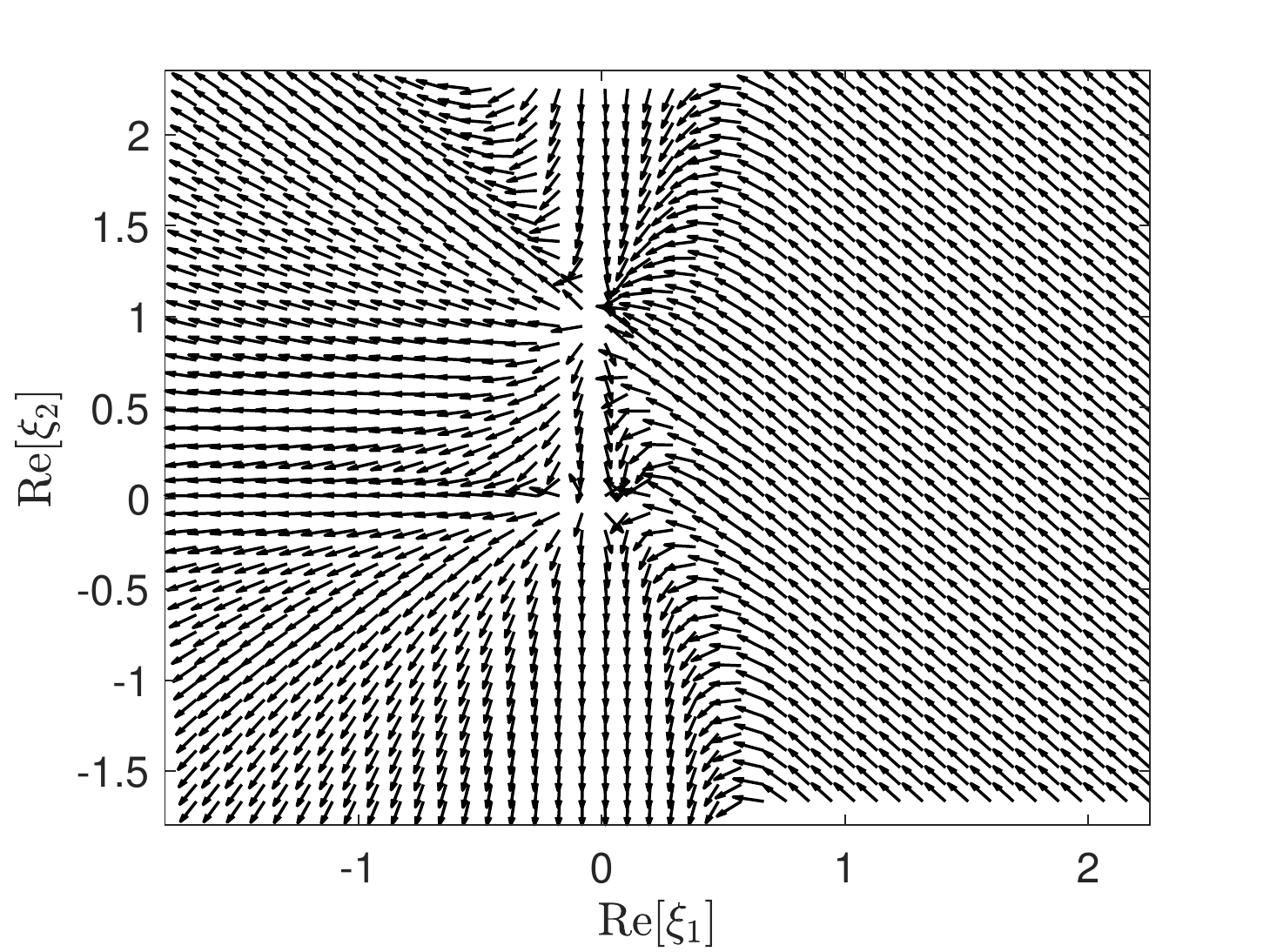}\includegraphics[width=0.45\textwidth]{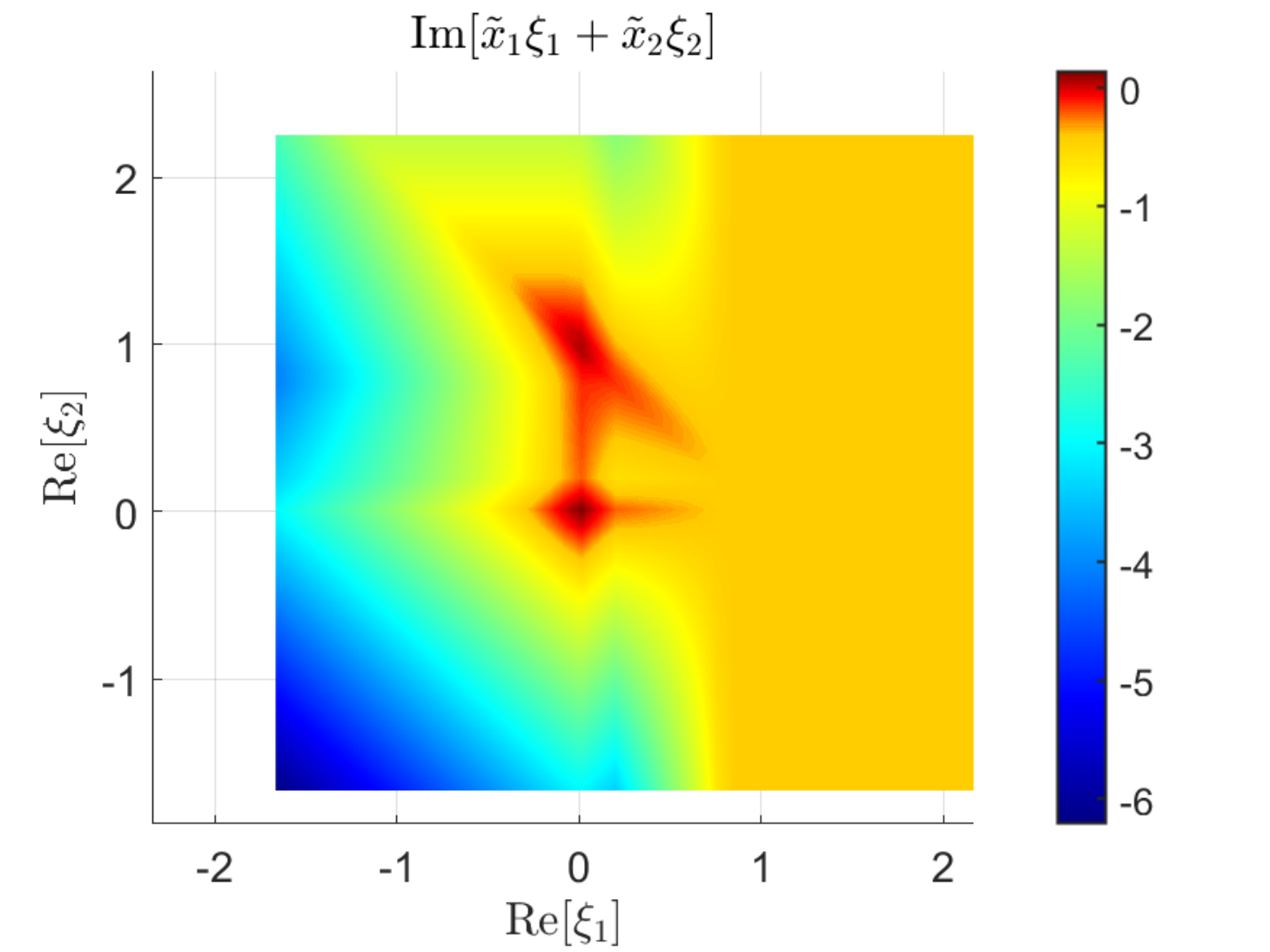}}
	\caption{The vector field $\tmmathbf{\eta'}$ (left); the imaginary part of the phase function ${\rm Im}[\tilde{\bdx}\cdot\bdxi]=\tilde{\bdx}\cdot\bdeta'$ (right).}
	\label{fig:eta_heatmaps}
\end{figure}

	The numerical integration is performed using a midpoint rule. We use the fact  that $\mathd\tmmathbf{\xi}$ is the  wedge product  $\mathd\xi_1\wedge \mathd\xi_2$ in order to find the area of each elementary surface on the integration grid. Indeed, $\mathd\xi_j$ can be considered as a function of some complex vector $\tmmathbf{v}$ that acts as follows: 
\begin{equation}
	\mathd\xi_j({\tmmathbf{v}}) = v_j,
\end{equation}
where $v_j$ is the $j$-th component of the vector (see e.g.\ \cite{Jo-2018}).
Then, by definition 
\begin{equation}
	\mathd\xi_1\wedge \mathd\xi_2 = 
	\begin{vmatrix}
		\mathd\xi_1(\delta\tmmathbf{\xi}_1)&\mathd\xi_1(\delta\tmmathbf{\xi}_2)\\
		\mathd\xi_2(\delta\tmmathbf{\xi}_2)&\mathd\xi_2(\delta\tmmathbf{\xi}_2)
	\end{vmatrix} =
	\begin{vmatrix}
		\mathd\xi_1^r + i({\ptl \eta_1'}/{\ptl \xi_1^r})\mathd\xi_1^r& i({\ptl \eta_1'}/{\ptl \xi_2^r})\mathd\xi_2^r\\
		i({\ptl \eta_2'}/{\ptl \xi_1^r})\mathd\xi_1^r&\mathd\xi_2^r + i({\ptl \eta_2'}/{\ptl \xi_2^r})\mathd\xi_2^r
	\end{vmatrix} 
	,
\end{equation}  
where $\delta\tmmathbf{\xi}_j$ are the vectors that produce  elementary surfaces (parallelograms) for the midpoint integration.
As a last precautionary measure, we visualise the log of the absolute value of the integrand of (\ref{eq:triang_int}) over $\bdG'$ in Figure \ref{fig:int_heatmaps} (left) and show that it is indeed small outside of the vicinity of the contributing points.  In Figure \ref{fig:int_heatmaps} (right) we compare the numerical results with the asymptotic formula (\ref{eq:triang_asympt}) for different values of $r$ and show that the agreement is excellent indeed, even for moderate values of $r$. 
\begin{figure}[h]
	\centering{\includegraphics[width=0.45\textwidth]{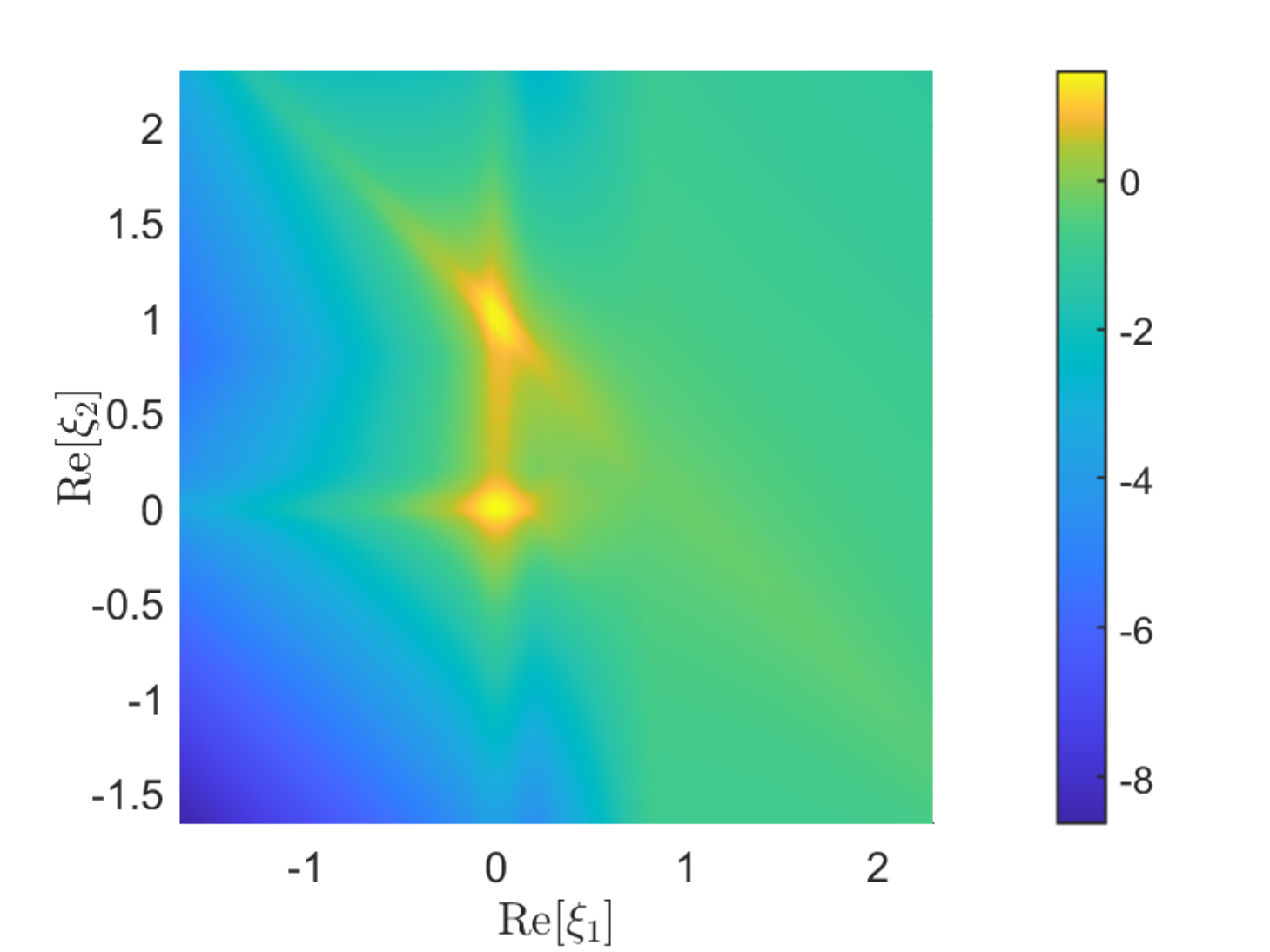}
		\includegraphics[width=0.45\textwidth]{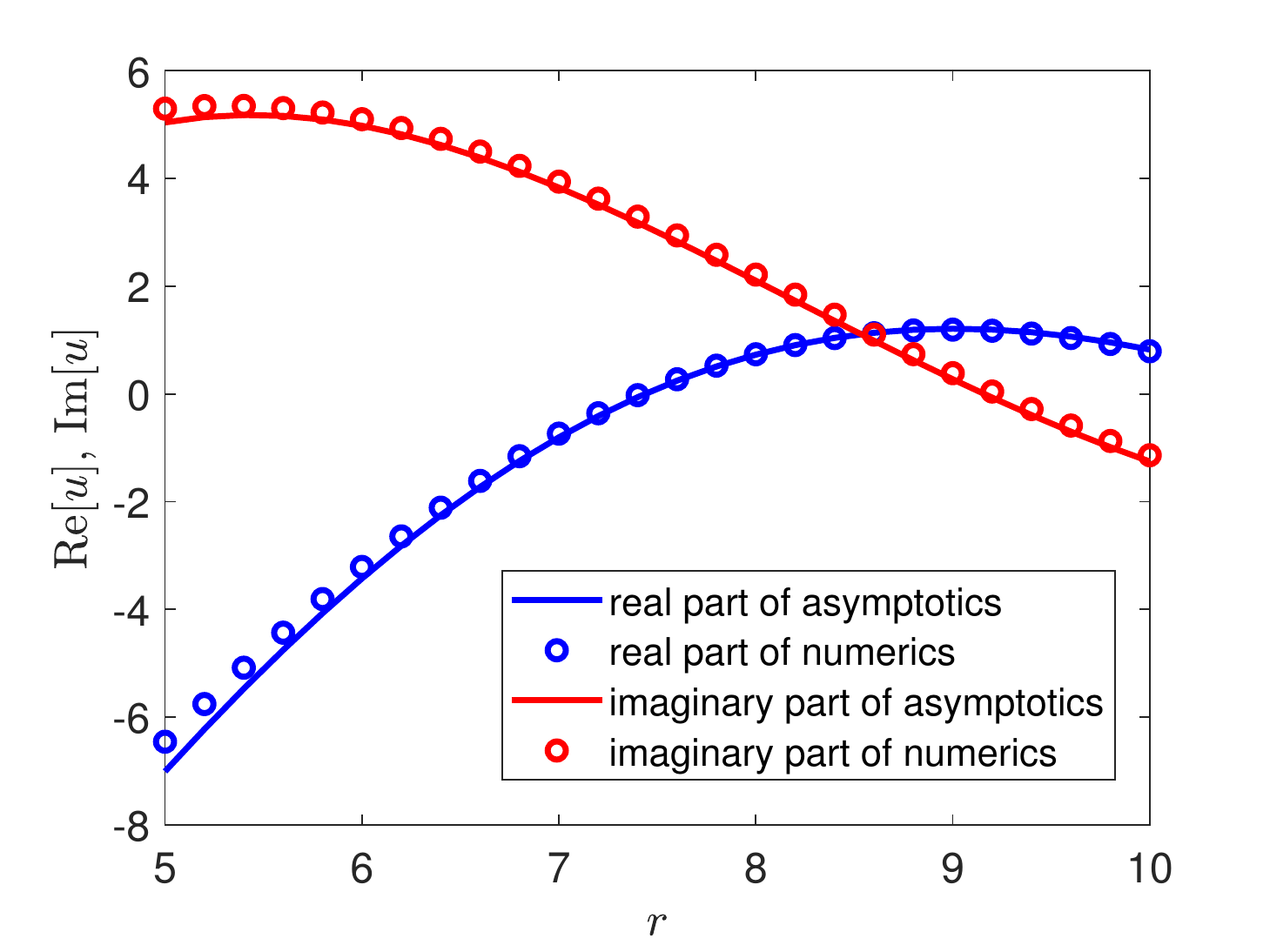}}
	\caption{ Heat map of the log of the absolute value of the integrand of (\ref{eq:triang_int}) on the deformed surface $\bdG'$ for $r=5$ (left); comparison between numerical results and the asymptotic formula (\ref{eq:triang_asympt}) (right).}
	\label{fig:int_heatmaps}
\end{figure}
\RED{One can expect that the next term in the asymptotic estimation (\ref{eq:triang_asympt}) will be of order $r^{-2}$. Let us check it numerically. In Figure~\ref{fig:error_triang_crossing}, we plot the absolute value of the difference  between the numerical  and  asymptotic estimation of (\ref{eq:triang_int}) on a logarithmic scale, and check that this difference decreases as $r^{-2}$.
\begin{figure}[h]
	\centering{\includegraphics[width=0.45\textwidth]{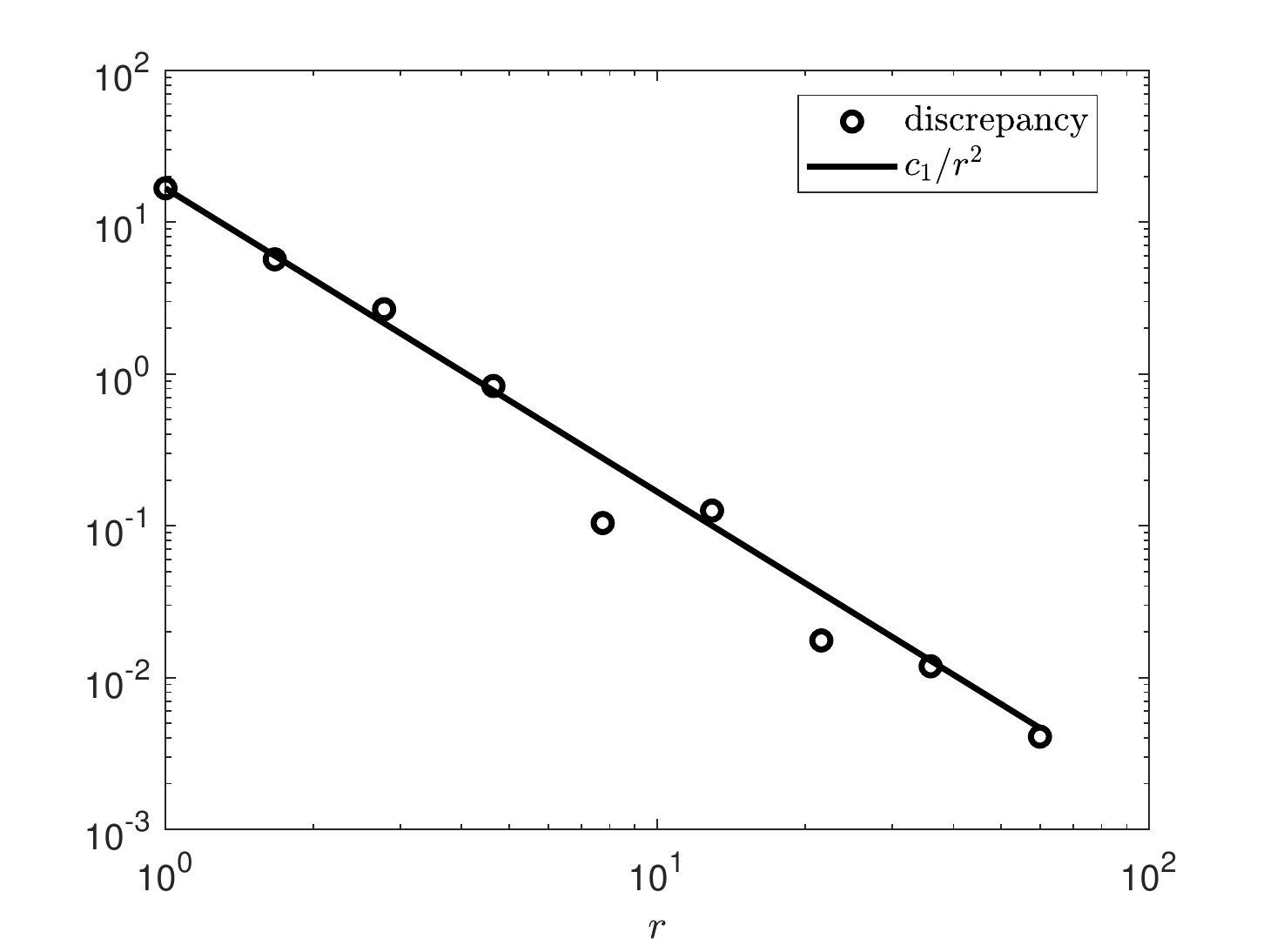}}
	\caption{\RED{Absolute value of the difference (discrepancy) between the numerical evaluation of the integral (\ref{eq:triang_int}) and its asymptotic estimation (\ref{eq:triang_asympt}), together with the function $c_1/r^2$, where $c_1$ is chosen in such a way that both graphs coincide at $r=1$.}}
	\label{fig:error_triang_crossing}
\end{figure}

We should note that the numerical integration scheme described here becomes computationally expensive relatively fast as $r$ grows. The authors were unable to evaluate the integral with $r \sim 100$ on a modern personal computer due to memory and computational issues. A proper way to deal with such a problem would be to build an optimal deformation of the surface of integration and use a dense mesh only in the vicinity of the active points. The latter falls out of  scope of the present paper, and could be a subject of another one.          
}

\begin{exa}
		Find the asymptotic behaviour of $u (\tmmathbf{x}) = \underset{\varkappa
		\searrow 0}{\lim}  u (\tmmathbf{x}; \varkappa)$ as $r = | \tmmathbf{x} |
	\rightarrow \infty$ where
	\begin{align}
		u (\tmmathbf{x}; \varkappa) & =  \iint_{\mathbb{R}^2} \frac{e^{-
				i\tmmathbf{x} \cdot \tmmathbf{\xi}}}{(\xi_2 - 2 + i \varkappa)^{1 / 2}
			(\xi_2 - \xi_1^2 + i \varkappa)^{1 / 2}} \, \mathd \tmmathbf{\xi},
			\label{eq:saddle_parabola}
	\end{align}
and validate it numerically for a specific observation direction $\tilde{\bdx}$.
\label{ex:7.5}
\end{exa}

The integrand has the real property and two irreducible singularity
components $\sigma_{1, 2}$ with respective defining function $g_{1, 2}$ defined by
$g_1 (\tmmathbf{\xi}) = \xi_2-2$, $g_2 (\tmmathbf{\xi}) = \xi_2-\xi_1^2$.
Using Section \ref{sec:findinggoodbridges}, we obtain the bridge configuration over the real traces $\sigma_{1, 2}'$
displayed in Figure \ref{fig:par_cross_geom} (left). These singularities lead to two
transverse crossings $\tmmathbf{\xi}_2^{\star} = \left( - \sqrt{2}, 2 \right)$
and $\tmmathbf{\xi}_3^{\star} = \left( \sqrt{2}, 2 \right)$ whose regions of activity are displayed on Figure \ref{fig:par_cross_geom} (left). We also assume that $\tilde{\bdx}$ is chosen so that we have an active \RED{SOS} $\bdxi_1^\star$. Using what has been done in Section \ref{sec:pointonsinglesingu} we know that $\tilde{\bdx} \perp \sigma_2'$ at $\bdxi_1^\star$. Hence, given this bridge configuration, we have 
\begin{align*}
	\tmmathbf{\xi}_1^{\star} = \left( - \frac{\tilde{x}_1}{2 \tilde{x}_2},
\frac{(\tilde{x}_1)^2}{4 (\tilde{x}_2)^2} \right)=\left( - \frac{x_1}{2 x_2},
\frac{x_1^2}{4 x_2^2} \right) = (\xi_1^s, \xi_2^s).
\end{align*}
\RED{Note also that, given the bridge configuration, if $x_2\leq0$ there can be no active \RED{SOS} on $\sigma_2'$, and that if $x_1=0$, the \RED{SOS} will also be a \RED{SOS} for $\sigma_1'$, which  is a case that we do not consider here.}  

The quantities $\text{\tmverbatim{a}}_1^{\star} = 0$ and
$\text{\tmverbatim{b}}^{\star}_1 = 1$ associated to $\sigma_1'$ remain the
same on the whole of $\sigma_1'$ since it is a straight line, and so does the
normal $\tmmathbf{n}_1^{\star} =\tmmathbf{e}_{\xi_2}$. However for $\sigma_2'$, these quantities will vary according to which point we are at. We denote $\text{\tmverbatim{a}}^{\star}_{2, (i)}$ and
$\text{\tmverbatim{b}}^{\star}_{2, (i)}$ these quantities at the point
$\tmmathbf{\xi}_i^{\star}$, so we get	$\text{\tmverbatim{a}}^{\star}_{2, (1)} = \tfrac{\tilde{x}_1}{\tilde{x}_2}$, 
	$\text{\tmverbatim{a}}^{\star}_{2, (3)} = - 2 \sqrt{2}$, 
	$\text{\tmverbatim{a}}^{\star}_{2, (2)} = + 2 \sqrt{2}$ and $\text{\tmverbatim{b}}^{\star}_{2, (1,2,3)} = 1$. Using these quantities, we can naturally define the normals $\bdn^\star_{2,(1,2,3)}$ and show that all the sign factors are equal to $+1$ in this problem. To ensure that the determinant $\Delta^{\star}_{2, 3}$ at each
	crossing is positive, we consider the real traces in the following order when applying the results of Section
	\ref{subsec:transversal}:  $\tmmathbf{\xi}_2^{\star} \leftrightarrow  (\sigma_2', \sigma_1')$ and $\tmmathbf{\xi}_3^{\star} \leftrightarrow (\sigma_1', \sigma_2')$.  Moreover, looking at the
	local behaviour of the integrand near the two crossings,  the constant in (\ref{eq:approxFtransverse}) is given for each crossing by $A_{2,3} = 1$, while the constant giving the local behaviour (\ref{eq:approxFsaddle}) near the \RED{SOS} is given by $A_1=(\xi_2^s-2)^{-1/2}$. Finally, using the Appendix \ref{app:quadratic}, we find that $\alpha$ in (\ref{eq:grotatedvarsaddle}) is given by $\alpha=(\tilde{x}_2)^4$. Combining the results of Sections \ref{sec:awaysaddle} and \ref{subsec:transversal}, we obtain
	\begin{align}
		u(\bdx) &\approx \frac{- 2 i \pi e^{- i (x_1 \xi_1^s + x_2 \xi_2^s)} }{x_2(\xi_2^s - 2)^{ 1 /
				2}}\RED{\mathcal{H}(x_2)} - 4 i \pi e^{- i \left( - \sqrt{2} x_1 + 2 x_2 \right)}  \frac{\mathcal{H}
			(x_1)}{\sqrt{x_1}} \frac{\mathcal{H} \left( - x_1 + 2 \sqrt{2} x_2
		\right)}{\sqrt{- x_1 + 2 \sqrt{2} x_2}}  \nonumber \\
	 & - 4 i \pi e^{- i \left( \sqrt{2} x_1 + 2 x_2 \right)} \frac{\mathcal{H}
	 	\left( x_1 + 2 \sqrt{2} x_2 \right)}{\sqrt{x_1 + 2 \sqrt{2} x_2}}
	 \frac{\mathcal{H} (- x_1)}{\sqrt{- x_1}},
	 \label{eq:asymptexample2}
	\end{align}
	as long as $\tilde{\bdx}$ is not perpendicular to $\sigma_1'$. \RED{Note that the seemingly problematic case $\xi_2^s=2$ corresponds to a situation (not considered in the present article) when $\bdxi_1^\star$ coincides with $\bdxi_2^\star$ or $\bdxi_3^\star$, i.e.\ when a transverse crossing is also a SOS. For convenience of implementation, it is possible to rewrite the first term of (\ref{eq:asymptexample2}) solely in terms of $x_1$ and $x_2$ by realising that:}
	\RED{\begin{align}
		\frac{- 2 i \pi e^{- i (x_1 \xi_1^s + x_2 \xi_2^s)} }{x_2(\xi_2^s - 2)^{ 1 /
				2}}\mathcal{H}(x_2)=-\frac{4\pi e^{i\frac{\pi}{4}[1+\text{sign}(x_1^2-8x_2^2)]}}{\sqrt{|x_1^2-8 x_2^2|}} e^{i\frac{x_1^2}{4x_2}}\mathcal{H}(x_2).
	\end{align}}

Let us now pick the observation direction $\tilde{\bdx}=\tfrac{1}{\sqrt{5}}(1,2)$ as represented in Figure \ref{fig:par_cross_geom} (left). We note that for this direction, the only contributing crossing is $\bdxi_2^\star$. 
\begin{figure}[h]
	\centering{\includegraphics[width=0.6\textwidth]{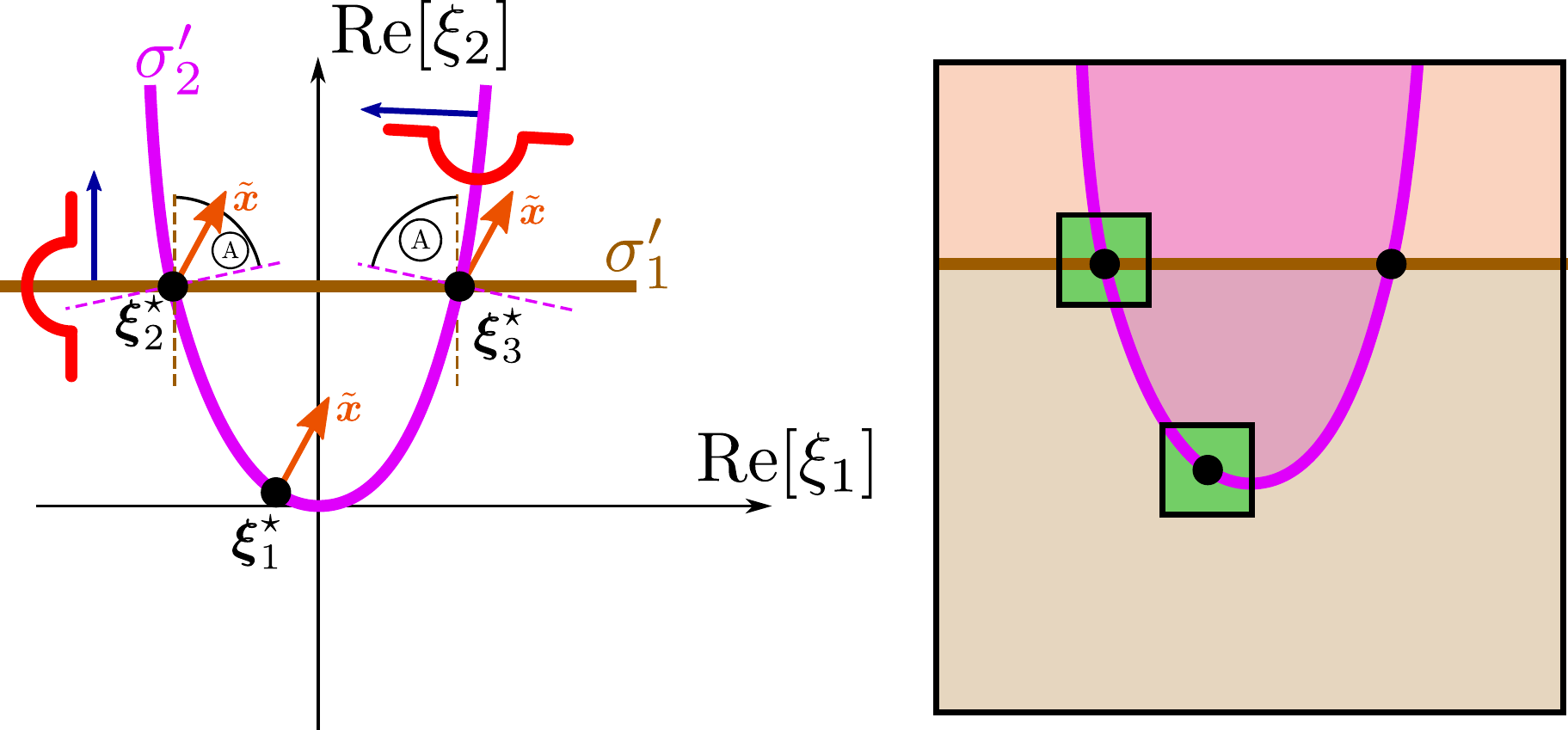}}
	\caption{The real traces associated to the integral (\ref{eq:saddle_parabola}) together with its bridge configuration and the regions of activity of the crossings (left); the domain decomposition used to build $\bdeta'$ (right).}
	\label{fig:par_cross_geom}
\end{figure}
Similarly to what has been done above, in order to evaluate the integral numerically, we build a flattable surface of integration $\bdG'$ with defining vector field $\bdeta'$, using Section \ref{subsec:transversal} and formula (\ref{eq:good_surf_1}) in the vicinity of the transverse crossing, Section \ref{sec:awaysaddle} and formula (\ref{eq:good_surf}) in the vicinity of the \RED{SOS}. Outside of these vicinities, we decompose the domain as in Section \ref{sec:integrationawayfromactive} and as illustrated in Figure \ref{fig:par_cross_geom} (right), and use linear interpolation. The resulting vector field is plotted in Figure \ref{fig:eta_heatmaps_2} (left). In Figure \ref{fig:eta_heatmaps_2} (right) we compare the numerical results with the asymptotic formula (\ref{eq:asymptexample2}) for different values of $r$ and show that the agreement is excellent indeed, even for moderate values of $r$. 
\begin{figure}[ht!]
	\centering{\includegraphics[width=0.45\textwidth]{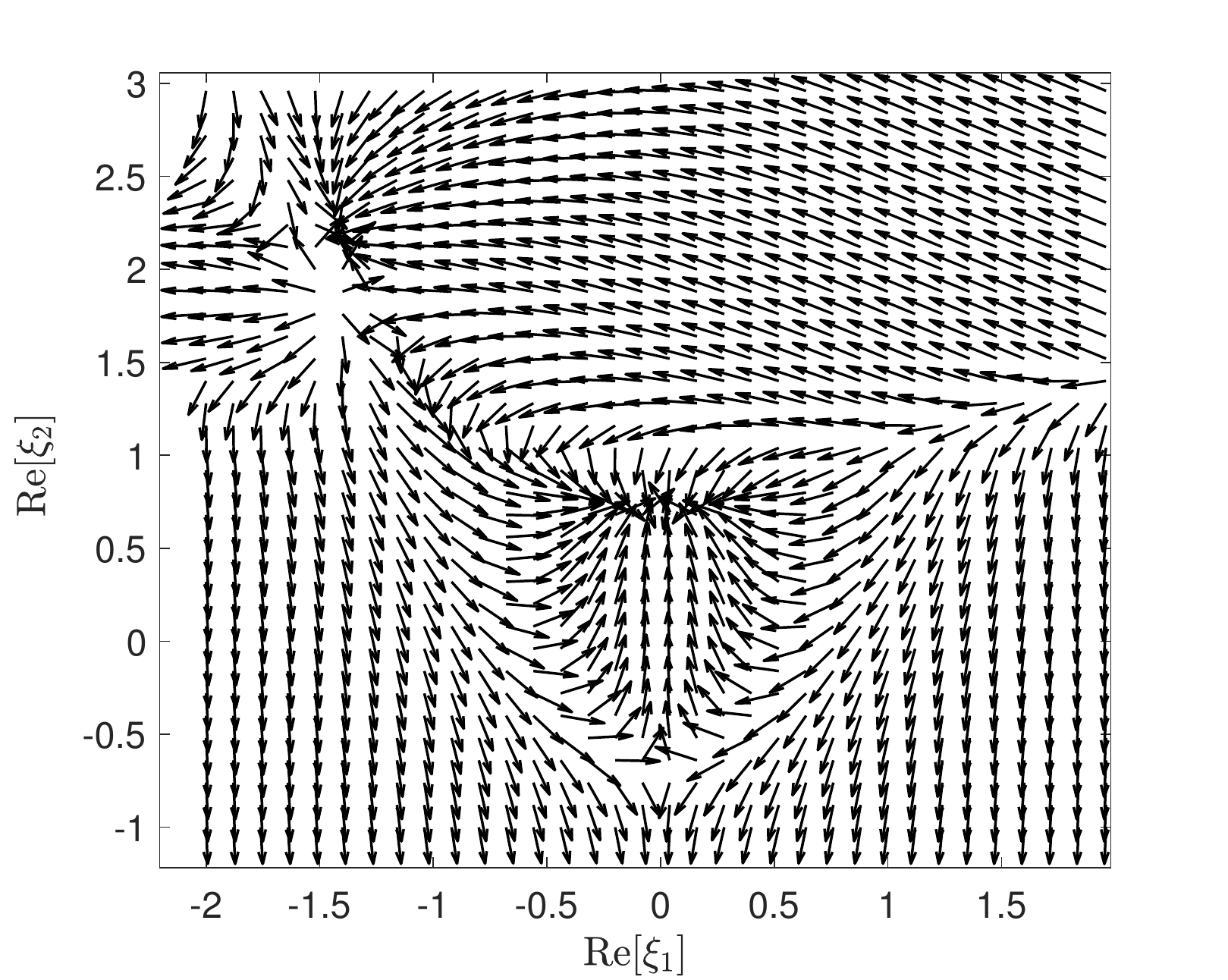}\includegraphics[width=0.48\textwidth]{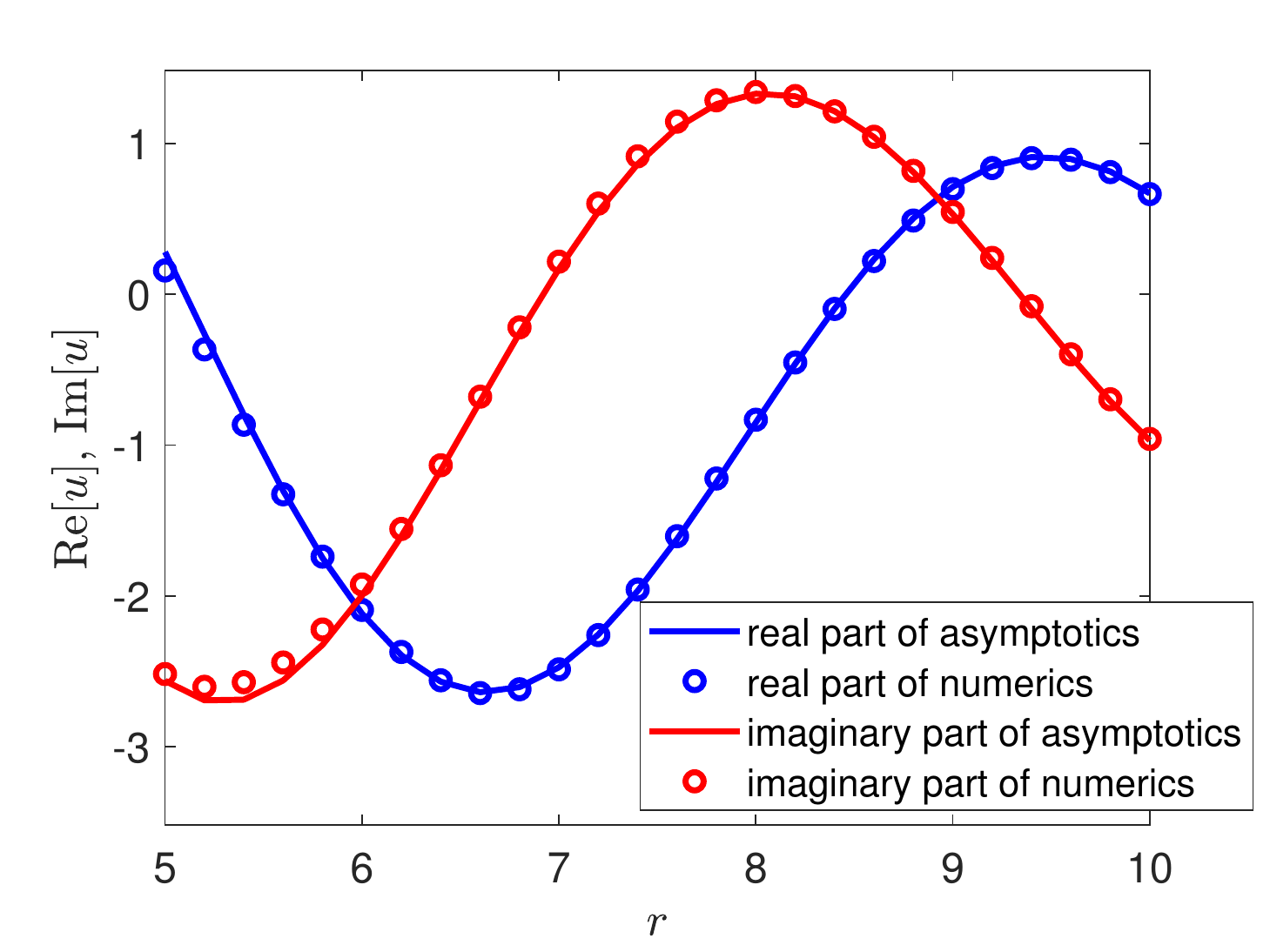}}
	\caption{The vector field $\tmmathbf{\eta'}$ (left); comparison between numerical results and the asymptotic formula (\ref{eq:asymptexample2}).}
	\label{fig:eta_heatmaps_2}
\end{figure}
\RED{Finally, in Figure~\ref{fig:error_saddle_crossing} we present the difference between the numerical and the asymptotic evaluation of (\ref{eq:saddle_parabola}) and check that it decreases as $r^{-2}$, as expected. 
	\begin{figure}[h]
		\centering{\includegraphics[width=0.45\textwidth]{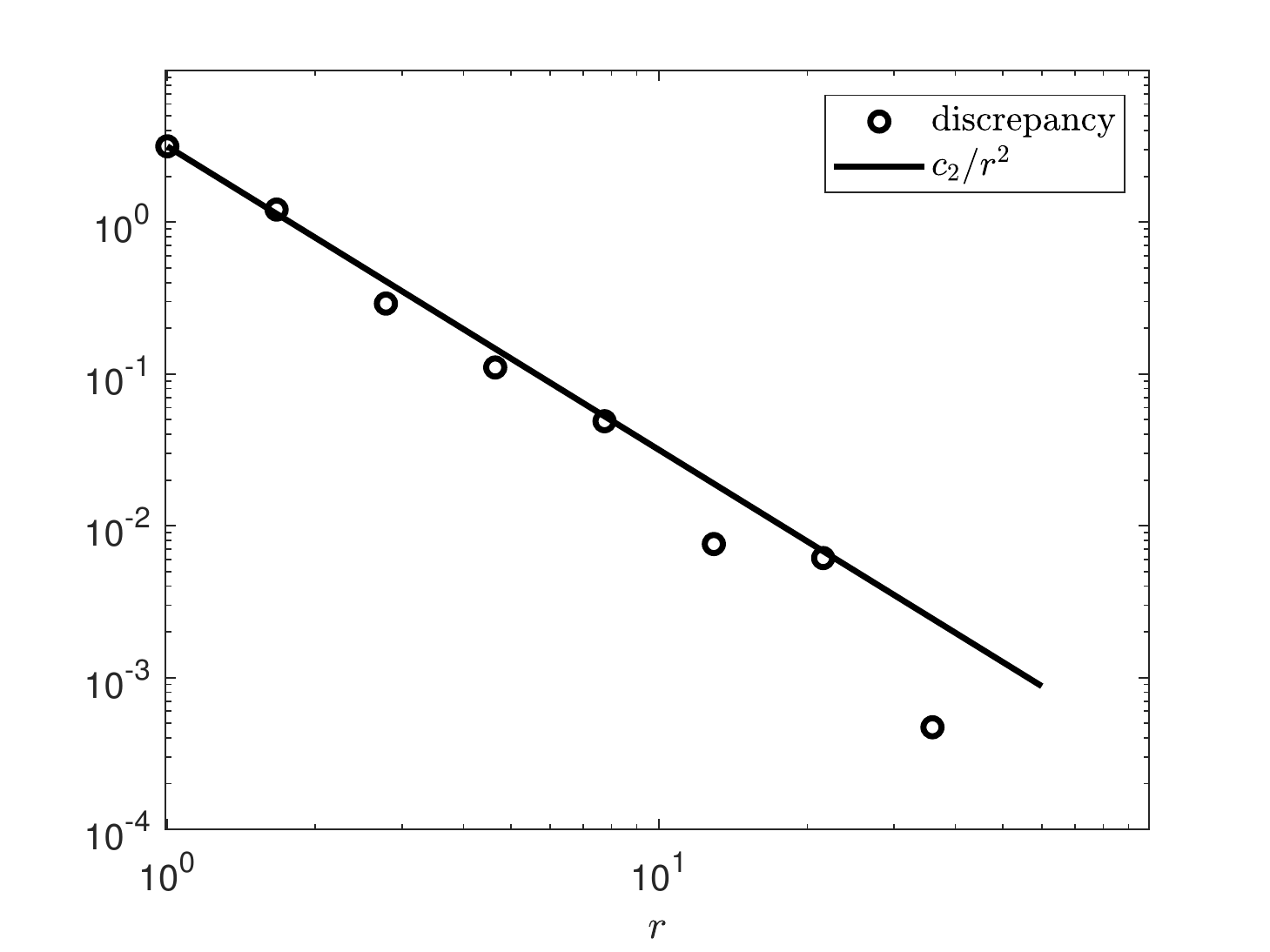}}
		\caption{\RED{Absolute value of the difference (discrepancy) between the numerical evaluation of the integral (\ref{eq:saddle_parabola}) and its asymptotic estimation (\ref{eq:asymptexample2}), together with the function $c_2/r^2$, where $c_2$ is chosen in such a way that both graphs coincide at $r=1$.}}
		\label{fig:error_saddle_crossing}
	\end{figure} 
}

\clearpage

\section{Conclusion}

In this article we proposed a new way to study physical fields $u$ given as
double Inverse Fourier Integrals of a function $F$ with the real-property.
These integrals occur very often in practice, and in diffraction theory in
particular \RED{(see also Appendix \ref{sec:complicated} for an extension of our method to more complicated integrals)}. In order to do so, we introduced the bridge and arrow notation to
describe the relative position of the surface of integration
$\bdG$ and the singularities of $F$ in a simple and visual way.
We then showed that for a given direction $\tmmathbf{x}$, only very specific
points of the real traces of $F$ lead to an asymptotic contribution to the
physical field $u$: the active \RED{SOS} and the active, non-additive transverse crossings.
Through a careful admissible deformation of the integration surface, we have also shown that the asymptotic expansion of the physical field could simply be written as a sum of the contributions of these special point. An explicit, closed-form expression for each contribution was given.
Finally, we concluded by a numerical illustration of the validity of the present theoretical development on two simple but non-trivial examples.
In Part II of this study, {\cite{Part6B}}, we will
demonstrate the strength and practicality of the present theory by considering
the three-dimensional problem of wave diffraction by a quarter-plane.

\textbf{Funding:} The authors would like to acknowledge funding by EPSRC (EP/W018381/1 and EP/N013719/1) for RCA and the RFBR grant 19-29-06048 for AVS and AIK. \\
\RED{\textbf{Acknowledgement:} The authors are very grateful to Professor I. David Abrahams for agreeing to play with the integrals of examples \ref{ex:7.2} and \ref{ex:7.5}, successfully recovering these results by other means and discovering a typographical error in (\ref{eq:triang_int})-(\ref{eq:triang_asympt}) in the initially submitted manuscript.}

\bibliography{biblio}
\bibliographystyle{unsrt}

\appendix
\counterwithin{figure}{section}

\section{On the choice of quadratic term} \label{app:quadratic}
In this appendix, we show that any quadratic terms of the form
\begin{align}
	T & = \tilde{\alpha} (\Delta \xi_1)^2 + \tilde{\beta} (\Delta \xi_2)^2 + \tilde{\gamma} (\Delta
	\xi_1) (\Delta \xi_2), \label{eq:initialquad}
\end{align}
where $\Delta \xi_{1,2}=(\xi_{1,2}-\xi_{1,2}^\star)$, can be written in the form
\begin{align}
	T & =  A \zeta_1^2 + B \Lambda^2 + C \Lambda \zeta_1,\quad \text{where} \quad  \zeta_1 = \tvb  \Delta \xi_1 - \tva  \Delta \xi_2 \quad \text{and} \quad \Lambda = \tva
	\Delta \xi_1 + \tvb  \Delta \xi_2, \label{eq:aimquad}
\end{align}
for some pair of numbers $(\tva,\tvb)$. This can be done directly by expanding $T$ as follows:
\begin{eqnarray}
	T & = & A (\tvb \Delta \xi_1 - \tva \Delta \xi_2)^2 + B (\tva \Delta \xi_1 + \tvb 
	\Delta \xi_2)^2 + C (\tvb  \Delta \xi_1 - \tva  \Delta \xi_2) (\tva \Delta \xi_1 +
	\tvb  \Delta \xi_2) \nonumber\\
	& = & (\tvb^2 A + \tva^2 B + \tva \tvb C) (\Delta \xi_1)^2 + (\tva^2 A + \tvb^2 B - \tva \tvb C)
	(\Delta \xi_2)^2 \nonumber\\
	& + & (- 2 \tva \tvb A + 2 \tva \tvb B + (\tvb^2 - \tva^2) C) (\Delta \xi_1) (\Delta \xi_2). 
	\label{eq:exedquad}
\end{eqnarray}
Hence comparing (\ref{eq:exedquad}) and (\ref{eq:initialquad}),
we get
\begin{align*}
	\left( \begin{array}{c}
		\tilde{\alpha}\\
		\tilde{\beta}\\
		\tilde{\gamma}
	\end{array} \right) & =  \left( \begin{array}{ccc}
		\tvb^2 & \tva^2 & \tva \tvb\\
		\tva^2 & \tvb^2 & -\tva \tvb\\
		- 2 \tva \tvb & 2 \tva \tvb & \tvb^2 - \tva^2
	\end{array} \right) \left( \begin{array}{c}
		A\\
		B\\
		C
	\end{array} \right),
\end{align*}
and, for some given $(\tilde{\alpha}, \tilde{\beta}, \tilde{\gamma})$, the value
of $(A,B,C)$ can be recovered by inverting this matrix. In particular, we get
\begin{align}
	A&=\frac{\tvb^2\tilde{\alpha}+\tva^2\tilde{\beta}-\tva\tvb\tilde{\gamma}}{(\tva^2+\tvb^2)^2}
\end{align}

Now imagine that there is a function $g$ that can be written
\begin{align*}
	g (\xi_1, \xi_2) & =  \tva \Delta \xi_1 + \tvb \Delta \xi_2 + \tilde{\alpha} 
	(\Delta \xi_1)^2 + \tilde{\beta} (\Delta \xi_2)^2 + \tilde{\gamma} (\Delta
	\xi_1) (\Delta \xi_2) +\mathcal{O} \left( \text{h.t.} \right),
\end{align*}
where h.t. means terms of order 3 and higher. Then we can always rewrite it as
\begin{align*}
	g (\xi_1, \xi_2) & =  \Lambda - \alpha \zeta_1^2 +\mathcal{O} (\Lambda^2 +
	\zeta_1 \Lambda) + \text{h.t.},
\end{align*}
where $\alpha = - A$. Upon introducing the variable $\zeta_2 = g (\xi_1,
\xi_2)$, it is clear that $\mathcal{O} (g) =\mathcal{O} (\Lambda) =\mathcal{O}
(\zeta_2)$, so that $\Lambda$ can be rewritten as $\Lambda = \zeta_2 + \alpha
\zeta_1^2 +\mathcal{O} \left( (\zeta_2)^2 + \zeta_1 \zeta_2 + \text{h.t.}
\right)$. \RED{Note that using the curvature formula for implicit planar curves (see e.g.\ \cite{Goldman2005-xu}), we can show that $\alpha$ is related to the curvature $\kappa$ of the planar curve defined by $g(\xi_1,\xi_2)=0$ as follows: $|\alpha|=|\kappa|/(2(\tva^2+\tvb^2)^{1/2})$.}

\RED{
	\section{On extending the approach to more complicated integrals}
	\label{sec:complicated}
	In wave motivated problems, one may be interested in the evaluation of a so-called plane wave decomposition integral of the form
	\begin{equation}
		u (x_1,x_2,x_3) = 
		\iint_{\mathbb{R}^2} F (\tmmathbf{\xi}) e^{-ix_1\xi_1 - ix_2\xi_2 + ix_3\sqrt{k^2 - \xi_1^2-\xi_2^2}} \mathd \tmmathbf{\xi},  
		\label{eq:plane_wave-integral}
	\end{equation} 
	where $k$ is some fixed real parameter that is usually referred to as a wavenumber, and $x_3>0$. Note that when $x_3=0$, this integral reduces to the double Fourier integral that has been the main subject of this article. 
	Upon rewriting the physical coordinates as 
	\[
	(x_1 , x_2 , x_3) = r (\tilde x_1, \tilde x_2 , \tilde x_3), 
	\]
	where $\tilde x_1^2 + \tilde x_2^2 + \tilde x_3^2 = 1$, and $r$ is a large parameter,  
	our aim is to estimate this integral as $r \to \infty$ to obtain a far-field approximation to the  three-dimensional physical field $u$. The integral (\ref{eq:plane_wave-integral}) belongs to a larger class of integrals that can be written 
	\begin{equation}
		u  = 
		\iint_{\mathbb{R}^2} F (\tmmathbf{\xi})
		\exp \{- i r  G (\tmmathbf{\xi}) \} 
		\mathd \tmmathbf{\xi} .
		\label{eq:plane_wave-integral_1}
	\end{equation}
	In the specific case of (\ref{eq:plane_wave-integral}) we chose the function $G$ to be defined by
	\begin{equation}
		 G (\xi_1 , \xi_2) = 
		\tilde x_1 \xi_1 +  \tilde x_2 \xi_2 - \tilde x_3\sqrt{k^2 - \xi_1^2-\xi_2^2}.
		\label{eq:particular_G}
	\end{equation}
	
	Despite the fact that the integral (\ref{eq:plane_wave-integral_1}) is clearly different from (\ref{eq:initial-integral}), 
	most of the consideration above can be applied to it. With this appendix, we wish to briefly list some ideas that can be used to modify our method to accommodate for integrals such as (\ref{eq:plane_wave-integral_1}). 
	
	\vskip 6pt
	\noindent
	{\bf 1. }
	Beside the singularities of the function $F (\tmmathbf{\xi})$, one should also study the 
	singularity set of $ G(\tmmathbf{\xi})$. With our choice (\ref{eq:particular_G}), it is the branch set 
	\begin{equation}
		\label{circle_sing}
		\sigma_0: \{\bdxi\in\mathbb{C}^2, \ 
		\xi_1^2 + \xi^2_2 = k^2\}, 
	\end{equation}
	with a real trace~$\sigma_0'$, which is a real circle.
	An indentation of the integration surface around $\sigma_0'$ should be chosen according to 
	the procedure described in 
	Subsection~\ref{sec:findinggoodbridges}. 
	As it is usually done in diffraction theory, the wavenumber parameter 
	$k$ is taken to have a small positive imaginary part $\varkappa$
	(the real part is positive and not small), and the limit $\varkappa \to 0$ is considered. 
	The procedure is described in details in Example~\ref{ex:35}. The result of the procedure is shown in  
	\figurename~\ref{fig:ex8}, left.
	
	Such a choice of bypass (bridge and arrow) makes the
	value of $\sqrt{k^2 - \xi_1^2 - \xi_2^2}$ either positive real or positive imaginary 
	almost everywhere on the indented surface of integration obtained from~$\mathbb{R}^2$.
	Due to the sign of the square root's imaginary part, the integrand has exponential decay in 
	the domain $\xi_1^2 + \xi_2^2 > k^2$ of $\mathbb{R}^2$
	without needing any deformation of the integration surface there. Since $x_3>0$, even if $F$ has active points outside the circle, their associated surface deformation can be made small enough so that the exponential decay due to the square root's sign still dominates. 
	Thus, to estimate the 
	integral, it is enough to consider the inside of the circle 
	$\xi_1^2 + \xi_2^2 < k^2$ and the neighbourhood of~$\sigma_0'$.
	
	\vskip 6pt
	\noindent
	{\bf 2. }
	Let us consider $\tmmathbf{\xi}^{\star}$ to be a real point strictly inside the circle $\sigma_0'$. Because of this,  $ G(\tmmathbf{\xi}^\star)$ is real. Let us further assume that $\nabla  G(\bdxi^\star)\neq\boldsymbol{0}$, where $\nabla  G$ is naturally defined by
	\[
	\nabla  G(\tmmathbf{\xi}) = \left(   \tilde x_1 + \frac{\tilde x_3 \xi_1}{\sqrt{k^2 - \xi_1^2 - \xi_2^2}}  ,
	\tilde x_2 + \frac{\tilde x_3 \xi_2}{\sqrt{k^2 - \xi_1^2 - \xi_2^2}} \right). 
	\]
	Then one can locally change the coordinates
	to some $(\xi_1' , \xi_2')$, such that, in the new coordinates, $ G$ is a linear
	function of $(\xi_1' , \xi_2')$. For example, a possible choice of coordinates is 
	\[
	\xi_1' =  G(\tmmathbf{\xi}), \qquad \xi_2' = \xi_1 \mbox{ or } \xi_2,
	\]
	which yields 
	\[
	\exp\{-i r  G (\tmmathbf{\xi}) \}  = \exp\{-i r \xi_1' \}. 
	\]
	In the new variables, the integral (locally) has the form  (\ref{eq:initial-integral}), 
	and it can be studied using the methods described above. 
	In particular, only the points where $F$ is singular can be active, 
	and the possible classes of active points are the SOS and the transverse crossings. The estimations of the local integral can be made using the method proposed above.
	
	As the result of this analysis, 
	one finds that the vector $\nabla G$ plays the same role as $\tilde{\bdx}$ has played in the main body of this article (note also that when $x_3=0$, they are equal to each other).  
	The SOS of $F$ are now defined as points 
	where some real trace $\sigma'_j$ is orthogonal to~$\nabla G$ and one should use 
	$\nabla G$ instead of $\tilde \bdx$ in the criterion of activity of a crossing point.
	
	\vskip 6pt
	\noindent
	{\bf 3. }
	A real point $\tmmathbf{\xi}^\star$ strictly inside the circle $\sigma_0'$ (i.e.\ with real $G(\tmmathbf{\xi}^\star)$), but this time with $\nabla G(\bdxi^\star)=\boldsymbol{0}$ should be considered separately. Assume further, for the sake of simplicity, that $F$ is not singular at $\bdxi^\star$. 
	This point is a type of active special points that 
	cannot appear in integrals of the type (\ref{eq:initial-integral})
	but
	can appear in integrals of the type
	(\ref{eq:plane_wave-integral_1}).
	Fortunately, such active points are well-known in the literature. 
	They are the usual two-dimensional saddle points (see e.g.\ \cite{Borovikov1994, Wong2001-cd, Jones1982-xs, Bleistein1987-xa, Felsen1994-hb,Lighthill78,Jones1958-nf,fedoryuk1977saddle}), and, as summarised briefly in Section \ref{subsec:2d saddle},  
	one can easily derive an estimate of the type (\ref{eq:saddle_estimation})
	provided that the Hessian of $G$ is not degenerate at~$\tmmathbf{\xi}^\star$. 
	
	It may be interesting to obtain this estimation 
	in the framework of the current paper. 
	Perform a local biholomorphic change of variables $\bdxi \leftrightarrow \bdxi'$, 
	such that $G$ in the new coordinates becomes written as 
	\[
	 G =  G(\tmmathbf{\xi}^\star) \RED{+} (\xi_1')^2 \RED{+} (\xi_2')^2 + O(|\xi_1'|^3 + |\xi_2'|^3).
	\]
	Here we used the fact that the eigenvalues of the Hessian of the function 
	$G$ given in (\ref{eq:particular_G}) are both \RED{positive}, but some other function $G$ may lead to other 
	signs in the quadratic form above. Note also that here we imply that $\tmmathbf{\xi}^\star$ maps to $\boldsymbol{0}$ in the new coordinates.
	
In the spirit of Section \ref{sec:surfaceparam}, build a local deformation of the integration surface described by the vector field $\bdeta'=(\eta_1',\eta_2')$ by choosing
	\[
	\eta_1' = \RED{ -{\rm Re}[\xi_1']}, \qquad \eta_2' = \RED{-{\rm Re}[\xi_2']}. 
	\]
	Such a deformation provides an exponential decay everywhere except at the point $\tmmathbf{\xi}^\star$. 
	The integral can be estimated near $\tmmathbf{\xi}^\star$
	as a repeated Gaussian, leading to
	\begin{equation}
		u \sim -\frac{i \pi }{r}  F(\tmmathbf{\xi}^\star) \exp \{ -i r G(\tmmathbf{\xi}^\star)\} J^\star,
		\label{eq:est_G}
	\end{equation}
where $J^\star$ is the Jacobian of the transformation $\bdxi \leftrightarrow \bdxi'$ at $\bdxi^\star$. 
	
	\vskip 6pt
	\noindent
	{\bf 4. }
	Let us now consider a  neighbourhood of $\sigma'_0$.
	Surprisingly, this set {\em does not contribute to the far field\/} if 
	$\tilde x_3 >0$. To demonstrate this, 
	consider a small real neighbourhood $\mathcal{B}^\star$ of a real  point $\tmmathbf{\xi}^\star$ belonging to the real circle $\sigma_0'$.
	Let us assume that this point is not near $(0, \pm k)$ (otherwise one should swap $\xi_1$ with $\xi_2$ in the discussion below). Assume also that $F$ does not have any singularities (crossings or SOS) on $\mathcal{B}^\star \setminus \sigma_0'$ and that no saddle points are present there either.
	Introduce the local change of coordinate $\bdxi \leftrightarrow \bdxi'$ defined by 
	\[
	\xi_1' = \sqrt{k^2 - \xi_1^2 - \xi_2^2}, 
	\qquad
	\xi_2' = \xi_2.
	\]
	Note that in the neighbourhood under consideration, $\xi_1'$ is small (because we are near $\sigma_0'$), and $|\xi_2'| < k$ (because we are not near $(0, \pm k)$). 
	One can hence see that   
	\begin{equation}
		\xi_1 = \sqrt{k^2 - (\xi_1')^2 - (\xi_2')^2} = 
		\sqrt{k^2 - (\xi_2')^2} - \frac{(\xi_1')^2}{2 \sqrt{k^2 - (\xi_2')^2}} + O(|\xi_1'|^3).
		\label{eq:est_xi_1}
	\end{equation}
	Moreover, the initial surface of integration of the corresponding local integral (an indented version of $\mathcal{B}^\star$) can be seen in the new variables as the product of a contour $\gamma$ in the $\xi_1'$ complex plane (shown in \figurename~\ref{fig:D01}) and a real segment in the $\xi_2'$ plane.
	
	\begin{figure}[ht!]
		\centering{
			\includegraphics[width=0.25\textwidth]{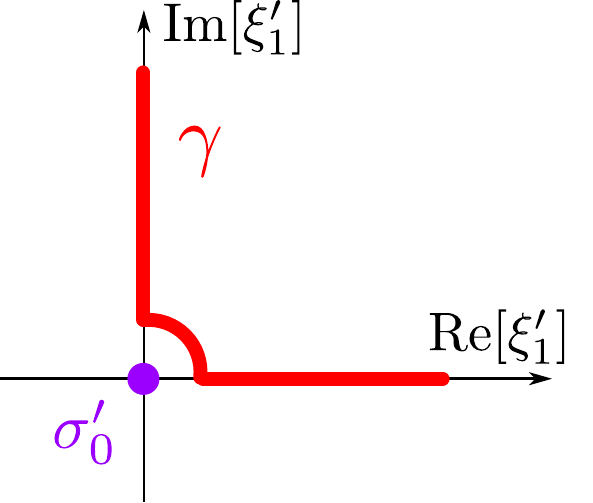}}
		\caption{\RED{Illustration of the contour $\gamma$ in the $\xi_1'$ plane}}
		\label{fig:D01}
	\end{figure}
	
	The contour $\gamma$ consists of a horizontal stem, a vertical stem, and a small arc.  
	The part of the integral corresponding to the vertical stem is exponentially small using point {\bf 1}. The part corresponding to the horizontal stem relates to points inside the circle $\sigma_0'$ and can be 
	made exponentially small as per point {\bf 2} due to the absence of problematic points in $\mathcal{B}^\star \setminus \sigma_0'$.
	Our aim is therefore to show that the integrand is exponentially small on the arc. For this to be true, we need ${\rm Im}[G]<0$ there. In particular, using (\ref{eq:particular_G}) and the fact that $\xi_2$ is real, we need ${\rm Im}[\tilde{x}_3 \xi_1']>{\rm Im}[\tilde{x}_1 \xi_1]$. Using (\ref{eq:est_xi_1}) and the fact that $|\xi_2'| < k$, we can show that ${\rm Im}[\tilde{x}_1 \xi_1]$ is a $O(|\xi_1'|^2)$ (but can be either positive or negative) and that ${\rm Im}[\tilde{x}_3 \xi_1']$ is positive and a $O(|\xi_1'|)$. Hence for $|\xi_1'|$ small enough, we have the required inequality and sought-after exponential decay.

%
	
	
	\vskip 6pt
	\noindent
	{\bf 5. }
	Let us now consider the process of continuously increasing $\tilde x_3$ from zero to some positive values. According to point {\bf 1}, active points (SOS or crossings) located outside the circle when $x_3=0$ should stop providing a contribution as soon as $\tilde x_3$ stops being zero. Moreover, according to point~{\bf 4}, active SOS on $\sigma_0'$ and active crossings of $\sigma_0'$ with other singularities should also stop providing contributions to the asymptotic expansion of $u$ as soon as $\tilde x_3$ stops being zero. These non-trivial facts should be commented upon.
	
	The active points outside the circle actually correspond to surface waves that are only non-exponentially decaying at infinity on the $x_3=0$ plane, and decay exponentially away from it.
	Moreover, a detailed local study shows that an active SOS on $\sigma_0'$ leaves $\sigma_0'$
	and changes its type when $\tilde x_3$ stops being zero. Namely, it becomes a 2D saddle point 
	as described in point~{\bf 3}. Similarly, an active crossing point of $\sigma_0'$ and, say, 
	$\sigma_1'$ leaves $\sigma_0'$, changes its type and becomes a SOS on~$\sigma_1'$.
	Nevertheless, the wave terms provided by these active points vary smoothly with $x_3$, despite the fact that the types of the active special points are changed. 
	
	
	\vskip 6pt
	\noindent
	{\bf 6. } 
	According to points {\bf 1}--{\bf 5},
	active special points of the integral (\ref{eq:plane_wave-integral_1})
	can be found from a local consideration very similar to that developed for the integral 
	(\ref{eq:initial-integral}). 
	Finding the leading asymptotic terms for each active special point is also 
	an elementary task. 
	One can formulate and prove the {\em locality principle\/} for (\ref{eq:plane_wave-integral_1})
	stating that once all active special points are found, it is possible to deform $\mathbb{R}^2$
	into an integration surface $\bdG'$, on which the integrand is exponentially small 
	everywhere except in the neighbourhoods of the active special points.         
}

\end{document}